\documentclass[11pt]{amsart}

\usepackage{graphicx,pinlabel,color,amssymb,hyperref}

\newtheorem{theorem}{Theorem}[section]
\newtheorem{corollary}[theorem]{Corollary}
\newtheorem{proposition}[theorem]{Proposition}
\newtheorem{lemma}[theorem]{Lemma}
\newtheorem{conjecture}[theorem]{Conjecture}

\newtheoremstyle{defn}{12pt}{12pt}{}{}{\bfseries}{.}{ }{}
\theoremstyle{defn}
\newtheorem{definition}[theorem]{Definition}
\newtheorem{defprop}[theorem]{Definition/Proposition}

\theoremstyle{remark}

\newtheorem{remark}[theorem]{Remark}
\newtheorem{example}[theorem]{Example}

\input xy
\xyoption{all}

\def\Z{\mathbb{Z}}
\def\R{\mathbb{R}}

\def\C{\mathbb{C}}

\def\rot{\operatorname{rot}}
\def\alg{\mathcal{A}}
\def\d{\partial}

\def\A{\mathcal{A}}
\def\sft{\mathcal{A}_{\operatorname{SFT}}}
\def\sftold{\mathcal{A}^0_{\operatorname{SFT}}}
\def\F{\mathcal{F}}
\def\W{\mathcal{W}}
\def\SS{\mathcal{S}}

\def\tb{\tilde{\beta}}

\def\tocyc{^{\operatorname{cyc}}}
\def\tocomm{^{\operatorname{comm}}}
\def\sfttocomm{^{\operatorname{comm}}_{\operatorname{SFT}}}
\def\Unshuff{\operatorname{Unshuff}}
\def\kk{\Bbbk}

\def\co{\colon\thinspace}

\def\sgn{\operatorname{sgn}}

\def\vv{\mathbf{v}}
\def\ww{\mathbf{w}}
\def\ee{\mathbf{e}}  
\def\str{{\operatorname{str}}}
\def\tosft{^{\operatorname{SFT}}}
\def\Spec{\operatorname{Spec}}

\begin{document}

\title{An $L_\infty$ structure for Legendrian contact homology}
\author{Lenhard Ng}
\email{\href{mailto:lenhard.ng@duke.edu}{lenhard.ng@duke.edu}}
\address{Department of Mathematics, Duke University, Durham, NC 27708}

\begin{abstract}
For any Legendrian knot or link in $\R^3$, we construct an $L_\infty$ algebra that can be viewed as an extension of the Chekanov--Eliashberg differential graded algebra. The $L_\infty$ structure incorporates information from rational Symplectic Field Theory and can be formulated combinatorially. One consequence is the construction of a Poisson bracket on commutative Legendrian contact homology, and we show that the resulting Poisson algebra is an invariant of Legendrian links under isotopy.
\end{abstract}

\maketitle

\section{Introduction}

\subsection{Main results}

The modern study of the topology of $3$-dimensional contact manifolds is inextricably tied to the study of Legendrian links, which are knots or links in the $3$-manifold that are everywhere tangent to the contact structure. Associated to a Legendrian link is an algebraic invariant known as Legendrian contact homology (LCH), which dates back to work of Chekanov \cite{Che} and Eliashberg--Hofer \cite{Eli98}. LCH, which is the homology of a complex now commonly called the Chekanov--Eliashberg differential graded algebra (DGA), is a key tool for studying both Legendrian links and the topology of associated contact and symplectic manifolds.

In this paper, we will restrict our attention to the basic setting of Legendrian links in $\R^3$ equipped with the standard contact structure $\ker(dz-y\,dx)$. From its original definition in \cite{Che}, the study of Legendrian contact homology in $\R^3$ has developed into a rich subject. In just the $\R^3$ setting, one can use LCH to distinguish non-isotopic Legendrian links and analyze Lagrangian cobordisms between them; there are surgery formulas expressing invariants of symplectic manifolds constructed by critical handle attachment in terms of LCH;
and LCH has connections to the theory of constructible sheaves (see e.g.\ \cite{STZ,NRSSZ}) 
and cluster varieties (see e.g.\ \cite{CGGLSS,GSW}). 
We refer the reader to \cite{ENsurvey} for a fairly recent survey of results concerning LCH in standard contact $\R^3$.

The purpose of this paper is to extend the Chekanov--Eliashberg DGA for Legendrian links in $\R^3$ to a larger algebraic structure, namely the structure of an $L_\infty$ algebra (homotopy Lie algebra). If we write $(\A,\d)$ for the DGA, then the $L_\infty$ structure consists of a sequence of graded-symmetric multilinear maps $\ell_k :\thinspace \A^{\otimes k} \to \A$, $k\geq 1$, satisfying a sequence of $L_\infty$ relations. The first $L_\infty$ operation $\ell_1$ is the differential $\d$; the second operation $\ell_2$ is a homotopy Lie bracket, with $\ell_1$ being a derivation with respect to $\ell_2$ and $\ell_2$ satisfying the Jacobi identity up to correction terms involving $\ell_1$ and $\ell_3$. 
In particular, $\ell_2$ induces a Lie bracket on the homology $H_*(\A,\ell_1)$.

We now state our main results more precisely. We say that a \textit{pointed Legendrian link} is an oriented Legendrian link with a collection of base points, at least one on each component. Let $\Lambda$ be a pointed Legendrian link with $s$ base points and $n$ \textit{Reeb chords}, which are integral curves of the Reeb vector field ($\partial/\partial z$ for standard contact $\R^3$) with endpoints on $\Lambda$. We
associate to $\Lambda$ the graded polynomial ring
\[
\A\tocomm = \Z[q_1,\ldots,q_n,t_1^{\pm 1},\ldots,t_s^{\pm 1}]
\]
where the $q_j$'s represent Reeb chords, the $t_i$'s represent base points, and the grading is induced by a Maslov potential on $\Lambda$. 
This is the algebra underlying a differential graded algebra $(\A\tocomm,\partial)$, which is the commutative version of the Chekanov--Eliashberg DGA of $\Lambda$.
Here the differential $\partial$ counts holomorphic disks in the symplectization $\R\times\R^3$ with boundary on the Lagrangian $\R\times\Lambda$, asymptotic to a single Reeb chord of $\Lambda$ at $+\infty$ and some collection of Reeb chords at $-\infty$. 

We note that the standard definition of the Chekanov--Eliashberg DGA involves a tensor algebra $\A$ generated by the same generators as $\A\tocomm$: $q_1,\ldots,q_n,t_1^{\pm 1},\ldots,t_s^{\pm 1}$. However, in the usual DGA, the $q_j$ and $t_i$ variables do not commute with each other; $\A\tocomm$ is the quotient of $\A$ by the two-sided ideal generated by graded commutators. In order for the $L_\infty$ relations to hold, we will be forced to use $\A\tocomm$ rather than $\A$; see however Section~\ref{ssec:knot} for a variant $L_\infty$ structure in the case of single-component Legendrian knots that does retain some of the noncommutativity of $\A$.

The homology of the usual Chekanov--Eliashberg DGA $(\A,\d)$ is an invariant of the pointed Legendrian link $\Lambda$ under Legendrian isotopy, and is called the Legendrian contact homology of $\Lambda$. The same is true for the commutative DGA once we tensor by a field $\kk$ of characteristic $0$: it follows from \cite[Theorem~3.14]{ENS} that the homology of $(\A\tocomm\otimes\kk,\d)$ is a Legendrian-isotopy invariant of $\Lambda$. We will refer to $H_*(\A\tocomm\otimes\kk,\d)$ as the commutative Legendrian contact homology of $\Lambda$ and denote it by $LCH\tocomm_*(\Lambda)$. 

The main construction of this paper is an $L_\infty$ algebra structure on $\A\tocomm$. In fact it follows from the construction of the $\ell_k$ maps that each $\ell_k$ satisfies the Leibniz rule with respect to standard (associative) multiplication on $\A\tocomm$; an $L_\infty$ algebra satisfying Leibniz is usually called a \textit{homotopy Poisson algebra}.

\begin{theorem}[see Definition~\ref{def:main} and Proposition~\ref{prop:htpyPoisson}]
Let $\Lambda$ be any pointed Legendrian link in $\R^3$ with commutative Chekanov--Eliashberg DGA $(\A\tocomm,\d)$. There are maps $\ell_k :\thinspace (\A\tocomm)^{\otimes k} \to \A\tocomm$ for $k\geq 1$, with $\ell_1 = \partial$, such that $(\A\tocomm,\{\ell_k\})$ is a homotopy Poisson algebra.
\label{thm:main-intro}
\end{theorem}

\noindent
The operation $\ell_k$ counts holomorphic disks in $(\R\times\R^3,\R\times\Lambda)$ with $k$ positive ends, as we discuss further in Section~\ref{ssec:motivation}.
As an immediate consequence of Theorem~\ref{thm:main-intro}, $\ell_2$ induces a Poisson bracket on the commutative LCH of $\Lambda$:

\begin{corollary}[see Corollary~\ref{cor:Poisson}]
For any pointed Legendrian link $\Lambda$, $(LCH\tocomm_*(\Lambda),\ell_2)$ is a Poisson algebra.
\end{corollary}

In this paper, we will also study, but not completely resolve, the question of invariance of the $L_\infty$ algebra under Legendrian isotopy. We conjecture that our construction is invariant:

\begin{conjecture}[see Conjecture~\ref{conj:invariance}] 
The homotopy Poisson algebra $(\A\tocomm,\{\ell_k\})$ of a pointed Legendrian link $\Lambda$ is invariant under Legendrian isotopy of $\Lambda$, up to $L_\infty$ equivalence.
\end{conjecture}

\noindent Here an $L_\infty$ equivalence is a $L_\infty$ morphism between $L_\infty$ algebras, given by a collection of maps $\{f_k\}_{k=1}^\infty$, inducing an isomorphism on $\ell_1$ homology; see Section~\ref{ssec:invariance-statement} for details. We will partially resolve this conjecture, constructing the maps $f_1,f_2$ and thus establishing invariance up through $\ell_2$.

\begin{proposition}[see Proposition~\ref{prop:invariance}]
The homotopy Poisson algebra $(\A\tocomm,\{\ell_k\})$ of $\Lambda$ is invariant under Legendrian isotopy of $\Lambda$, up through the $\ell_2$ operation.
\end{proposition}

\begin{corollary}[see Proposition~\ref{prop:invariance}]
The Poisson algebra $(LCH\tocomm_*(\Lambda),\ell_2)$ is invariant under Legendrian isotopy of $\Lambda$.
\end{corollary}

Our construction of the $L_\infty$ structure, like Legendrian contact homology itself in $\R^3$, can be described in completely combinatorial (non-analytic) terms. As a consequence, one can readily calculate the $L_\infty$ structure for any given Legendrian link.
One particularly interesting family of Legendrian links is the collection of $(-1)$-closures of admissible positive braids (see Section~\ref{ssec:closures}), including rainbow closures of positive braids. For this class of Legendrians, the $L_\infty$ algebra is strict in the sense that $\ell_k=0$ for $k \geq 3$, and it follows that the commutative Chekanov--Eliashberg DGA is a DG Poisson algebra. In particular, for an admissible positive braid with $N$ strands and $k$ crossings, the $\ell_2$ operation on $\A\tocomm$ induces a Poisson bracket on the polynomial ring $\C[q_1,\ldots,q_k,t_1^{\pm 1},\ldots,t_N^{\pm 1}]$. In the specific case where $\Lambda$ is a Legendrian $(2,2n)$ torus link given by the $(-1)$-closure of the $2$-braid $\sigma_1^{2n+2}$, this Poisson bracket agrees with the Flaschka--Newell bracket; see Proposition~\ref{prop:FN}. We believe that the general family of Poisson brackets associated to arbitrary admissible positive braids may be new.

In upcoming work \cite{CGNSW}, the $L_\infty$ structure constructed in this paper will be used to construct a symplectic form on the augmentation variety of a large family of Legendrian links. Here the augmentation variety is a Legendrian-isotopy invariant constructed from the Chekanov--Eliashberg DGA---roughly speaking, it is given by $\Spec LCH\tocomm_0(\Lambda)$---and the $\ell_2$ operation dualizes to a closed differential $2$-form on the augmentation variety. The fact that this $2$-form is nondegenerate follows from Sabloff duality \cite{EESab}; conversely, it will follow that the $L_\infty$ structures in this paper are nondegenerate in a suitable sense.

In another direction, we conjecture that the $L_\infty$ structure is functorial under exact Lagrangian cobordism:

\begin{conjecture}
Let $\Lambda^\pm$ be pointed Legendrian links, and
suppose that there is an exact Lagrangian cobordism from $\Lambda^-$ to $\Lambda^+$. Then there is an $L_\infty$ morphism from the $L_\infty$ algebra $((\alg^+)\tocomm,\{\ell_k^+\})$ of $\Lambda^+$ to the $L_\infty$ algebra $((\alg^-)\tocomm,\{\ell_k^-\})$ of $\Lambda^-$ extending the usual cobordism map of Chekanov--Eliashberg DGAs $(\alg^+)\tocomm \to (\alg^-)\tocomm$.
\end{conjecture}

\noindent
We hope to return to this conjecture in future work.

\subsection{Geometric motivation}
\label{ssec:motivation}

We now discuss the construction of the $L_\infty$ structure in slightly more detail. 
The $\ell_k$ operations that constitute our $L_\infty$ structure are defined using rational Symplectic Field Theory (SFT) \cite{EGH}. Where LCH counts rigid holomorphic disks in the symplectization $\R\times\R^3$ with boundary on $\R\times\Lambda$ with a single positive end at a Reeb chord (and any number of negative ends), rational SFT counts rigid holomorphic disks with any number of positive ends; the holomorphic disks with $k$ positive ends contribute to $\ell_k$ for each $k\geq 1$. See Figure~\ref{fig:disks} for an illustration of terms in $\ell_1=\d$ and $\ell_2$.

\begin{figure}
\labellist
\small\hair 2pt
\pinlabel ${\color{blue} \R\times\Lambda}$ at 81 57
\pinlabel $\R\times \R^3$ at 104 31
\pinlabel $q_i$ at 45 85
\pinlabel $q_{j_1}$ at 39 16
\pinlabel $q_{j_2}$ at 57 16
\pinlabel $q_{j_3}$ at 68 18
\pinlabel {$\partial(q_i) = q_{j_1}q_{j_2}q_{j_3} + \cdots$} at 49 2
\pinlabel $\R$ at 120 84
\pinlabel $q_{i_1}$ at 182 82
\pinlabel $q_{i_2}$ at 204 82
\pinlabel $q_{j_1}$ at 182 16
\pinlabel $q_{j_2}$ at 205 16
\pinlabel {$\ell_2(q_{i_1},q_{i_2}) = q_{j_1}q_{j_2}+\cdots$} at 192 2
\endlabellist
\centering
\includegraphics[width=\textwidth]{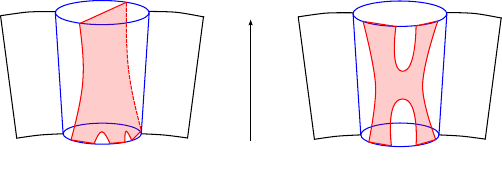}
\caption{
Left: a holomorphic disk contributing to the differential $\d$ in the (commutative) Chekanov--Eliashberg DGA $(\A\tocomm,\d)$ of a Legendrian $\Lambda$; here $q_i,q_{j_1},q_{j_2},q_{j_3}$ are Reeb chords of $\Lambda$. Right: a holomorphic disk contributing to the $\ell_2$ operation on $\A\tocomm$.}
\label{fig:disks}
\end{figure}
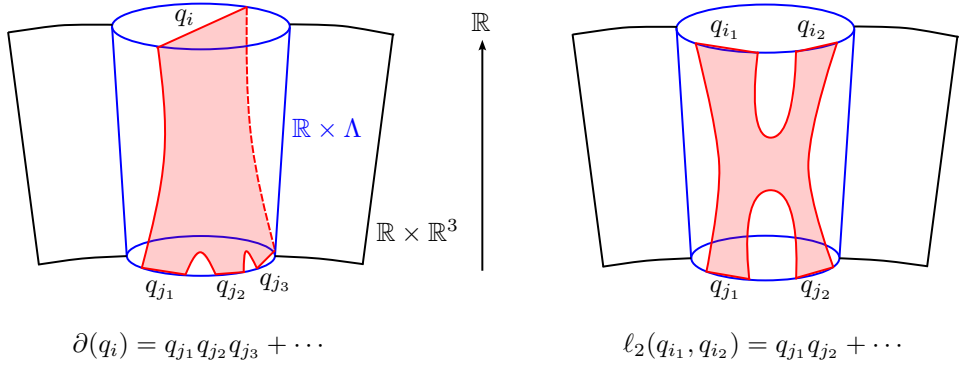

The version of rational SFT that we consider in this paper is a relative version of rational SFT for closed contact manifolds $V$, which counts punctured holomorphic spheres in the symplectization $\R\times V$ whose punctures are mapped to Reeb orbits at $\pm\infty$ in the $\R$ direction. (In string-theoretic language, the closed setting considers closed strings in $V$, while the relative setting considers open strings in $V$ with boundary on $\Lambda$.) In the closed case, it is well-known that one can use rational SFT to produce an $L_\infty$ structure on the polynomial algebra generated by Reeb orbits: see \cite{Siegel,MZ}. More concretely, $\ell_k(q_{j_1},\ldots,q_{j_k})$ is a sum, over all rigid punctured holomorphic spheres with positive ends at the Reeb orbits $q_{j_1},\ldots,q_{j_k}$, of the product of the negative ends of each sphere.

In the closed case, the claim that the $\ell_k$ operations satisfy the $L_\infty$ relations is derived from the fact that $1$-dimensional moduli spaces of punctured holomorphic spheres can be compactified by adding two-story buildings consisting of two rigid punctured holomorphic spheres with a positive end of one sphere glued to a negative end of the other. It then follows, as in standard Floer-theoretic $\d^2=0$ arguments, that the endpoints of any $1$-dimensional moduli space contribute canceling terms to the $L_\infty$ relations. The standard way to express the fact about two-story buildings is an equation of the form
\[
\{h,h\} = 0.
\]
Here $h$ is a Hamiltonian in the variables $q_j$ and dual variables $p_j$ given by summing the product of the positive ends (recorded as $p_j$'s) and the negative ends (recorded as $q_j$'s) of all rigid punctured holomorphic spheres, and the SFT bracket $\{\cdot,\cdot\}$ is defined by $\{p_{j_1},q_{j_2}\} = -\{q_{j_2},p_{j_1}\} = \delta_{j_1j_2}$: geometrically this glues a punctured sphere with a negative end at some Reeb orbit to a punctured sphere with a positive end at the same Reeb orbit. See \cite{Lat} for further discussion of this perspective on the $L_\infty$ structure in the closed setting.

\begin{figure}
\centering
\includegraphics[width=0.8\textwidth]{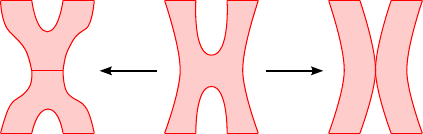}
\caption{
Possible degenerations of a $1$-dimensional moduli space of holomorphic disks (center). Left: a two-story building given by gluing together a positive end of one rigid disk to a negative end of another rigid disk. Right: two rigid disks joined at a point on their boundaries.
}
\label{fig:degenerations}
\end{figure}

Compared to the closed case, the relative case of rational SFT has a significant complication: $1$-dimensional moduli spaces of punctured holomorphic disks can degenerate to two-story buildings, but there are also ``boundary degenerations'' (essentially boundary bubbling) consisting of two holomorphic disks joined by a common point on their boundary. See Figure~\ref{fig:degenerations} for an illustration of each type of degeneration. This is a well-known issue, and for Legendrians in $\R^3$ there are at least two (related) approaches to work around this difficulty. One approach, due to Ekholm \cite{EkholmSFT}, is to restrict to multi-component Legendrian links and count only holomorphic disks whose boundary components lie on all different components of the Legendrian; this eliminates boundary bubbling. 

Another approach, and the one that underlies the present paper, is given in previous work of the author \cite{SFT}. This approach fleshes out a strategy laid out by Cieliebak, Latschev, and Mohnke \cite{CL}, using ideas from string topology to treat the boundary-degeneration problem. As in the closed case, one can define a Hamiltonian $h$ associated to a Legendrian $\Lambda$, but now the equation $\{h,h\}=0$ is replaced by
\[
\frac{1}{2}\{h,h\} = \delta h,
\]
where $\delta$ is an operation that we call the string coproduct. In \cite{SFT}, this equation (the ``quantum master equation'') is used to produce a curved DGA associated to any Legendrian knot (the ``LSFT algebra''), which is invariant in a suitable sense under Legendrian isotopy.

Analogously to the closed case, one can try to define an $L_\infty$ structure on the commutative Chekanov--Eliashberg DGA $\A\tocomm$ of a Legendrian $\Lambda$ by defining $\ell_k(q_{i_1},\ldots,q_{i_k})$ to be a sum over all rigid holomorphic disks with positive ends at the Reeb chords $q_{i_1},\ldots,q_{i_k}$ of the product of the negative ends of the disk; see Figure~\ref{fig:disks} for illustrations for $k=1$ and $k=2$. For $k=1$ this recovers the usual differential $\d$ on $\A\tocomm$. However, the presence of $\delta h$ in the quantum master equation means that these $\ell_k$'s do not satisfy the $L_\infty$ relations, unlike in the closed case.

We will solve this problem in this paper by defining $\ell_k$ as described above for all $k\neq 2$, but adding a correction term $\ell_2^\str$ to $\ell_2$. This correction term, which is of the form $\ell_2^\str(q_{j_1},q_{j_2}) = \alpha_{{j_1},j_2} q_{j_1}q_{j_2}$ for some $\alpha_{j_1j_2} \in \frac{1}{2}\Z$, is a string-topology contribution to $\ell_2$ and depends on the placement of base points on $\Lambda$. By choosing these coefficients $\alpha_{j_1j_2}$ appropriately, it is possible to cancel out the effect of the $\delta h$ term in the quantum master equation and thus produce operations $\{\ell_k\}_{k=1}^\infty$ satisfying the $L_\infty$ relations. As an added benefit, even though rational SFT was only defined for Legendrian knots and not multi-component Legendrian links in \cite{SFT}, we will show that our construction of $\{\ell_k\}$ extends to links as well as knots, and that the $L_\infty$ relations are satisfied in all cases.

We remark that even though defining the $L_\infty$ structure combinatorially is relatively straightforward---see Section~\ref{sec:l-infty}---verifying the $L_\infty$ relations is technically rather intricate, and this verification forms the heart of this paper. One specific notable technical issue is that because $\ell_2$ needs to be graded symmetric, (anti)symmetrization forces us to define string contributions $\ell_2^\str$ to $\ell_2$ that have coefficients in $\frac{1}{2}\Z$ rather than in $\Z$. This means that our $L_\infty$ algebras are forced to be defined in fields of characteristic $\neq 2$. In turn, this forces us to keep careful track of signs (or, geometrically, orientations of moduli spaces) at all times. A fair amount of the difficulty in verifying the $L_\infty$ relations comes from working through the resulting profusion of signs.

In the special case where $\Lambda$ is a single-component Legendrian knot, the string corrections to $\ell_2$ simplify significantly, and the full $\ell_2$ operation can be expressed in terms of the differential on the LSFT algebra, which was constructed in \cite{SFT} to encode rational SFT. Indeed, for Legendrian knots, the $L_\infty$ structure defined in this paper is equivalent to a version of the LSFT algebra, the \textit{commutative complex} defined in \cite[section 2.4]{SFT}, in the sense that one can read the commutative complex from the $L_\infty$ structure and vice versa; see Proposition~\ref{prop:commsft}. Thus we can view the $L_\infty$ structure in this case as a repackaging of a previously-defined Legendrian invariant. However, we note two points: first, it is unclear that invariance of the commutative complex as shown in \cite{SFT} implies invariance of the $L_\infty$ structure, or vice versa; second, our construction of the $L_\infty$ structure is entirely new for Legendrian links, or indeed for Legendrian knots with multiple base points.

\subsection{Outline of paper}

In Section~\ref{sec:sft}, we describe the ingredients that go into defining the $L_\infty$ operations, including broken closed strings, the Hamiltonian, the string coproduct, and the quantum master equation, and establish some of their fundamental properties. This is an upgrade of material from \cite{SFT} from the setting of Legendrian knots to the setting of pointed Legendrian links. In particular, we include a discussion of grading for the Chekanov--Eliashberg DGA in the presence of multiple base points that we believe is new and may be of independent interest.

We present the construction of the $L_\infty$ structure, with examples, in Section~\ref{sec:l-infty}. We then establish the $L_\infty$ relations in Section~\ref{sec:l-infty-proof} and prove invariance under Legendrian isotopy in Section~\ref{sec:invariance}.

\subsection*{Acknowledgments}
The immediate motivation for this work was to provide some technical background for future collaborative work on augmentations and clusters stemming from a SQuaRE project sponsored by the American Institute of Mathematics. I am deeply indebted to my collaborators on this project, Roger Casals, Honghao Gao, Linhui Shen, Daping Weng, and Eric Zaslow, for many insightful discussions that led to the present paper, and to AIM for providing a perfect setting for this group to work. I also thank Tobias Ekholm, Kenji Fukaya, Augustin Moreno, and Vivek Shende for illuminating conversations related to $L_\infty$ structures in symplectic topology, and an anonymous referee for many helpful comments on an earlier version of this paper. This work was partially supported by NSF grant DMS-2003404.

\section{Rational Symplectic Field Theory for Legendrian links}
\label{sec:sft}

In this section, we present background material that goes into defining rational SFT for $1$-dimensional Legendrians, following \cite{SFT}; the constructions described in this section will then be used to define the $L_\infty$ structure in Section~\ref{sec:l-infty}. We begin in Sections~\ref{ssec:algebras} and~\ref{ssec:hopf-ex} with a description of the algebraic framework; this is followed in Sections~\ref{ssec:capping} through \ref{ssec:bcs} by a discussion of capping paths and the more general notion of broken closed strings. We introduce the Hamiltonian for rational SFT in Section~\ref{ssec:h} and the algebraic operations of the SFT bracket and string coproduct in Sections~\ref{ssec:sft} through~\ref{ssec:derivation}. We conclude by discussing the quantum master equation in Section~\ref{ssec:qme}.

Many of the constructions in this section were originally presented in \cite{SFT}, and we will make frequent reference to that paper, but we have tried to keep our exposition as self-contained as possible. We will also extend the constructions from \cite{SFT}, which were defined specifically for only Legendrian knots with a single base point, to the more general case of Legendrian links with arbitrarily many base points. This adds a significant level of complication to the arguments as compared to \cite{SFT}: for instance, the string coproduct is no longer a derivation with respect to the SFT bracket, cf.\ Proposition~\ref{prop:delta}.

Throughout this section and the paper, we will be considering $\R^3$ with standard coordinates $x,y,z$ and the standard contact structure $\ker(dz-y\,dx)$. Smoothly embedded links in $\R^3$ are Legendrian if they are everywhere tangent to the plane field given by the contact structure. One reference for relevant background material about Legendrian links in standard contact $\R^3$ is the survey \cite{ENsurvey}.

\subsection{Algebras and modules associated to Legendrian SFT}
\label{ssec:algebras}

Here we review the basic algebraic structures from rational Legendrian Symplectic Field Theory \cite{SFT}, with some modifications for our purposes. These are the building blocks that we will use to define the $L_\infty$ structure.

\begin{definition}
A \textit{pointed Legendrian link} in $\R^3$ is an oriented Legendrian link equipped with a collection of base points, with each connected component of $\Lambda$ containing at least one base point. A pointed Legendrian link is \textit{diagram-generic} if the only singularities of the immersed curve $\Pi_{xy}(\Lambda)$ are transverse double points and no base point is mapped by $\Pi_{xy}$ to one of these double points.
\end{definition}

Let $\Lambda$ be a pointed Legendrian link in $\R^3$. We label the base points on $\Lambda$ by $\bullet_1,\ldots,\bullet_s$, and the crossings of $\Pi_{xy}(\Lambda)$, which are the Reeb chords of $\Lambda$, by $a_1,\ldots,a_n$. 
For bookkeeping of paths on $\Lambda$, we introduce a partner to each base point $\bullet_i$: a point $\star_i$ immediately preceding $\bullet_i$ with respect to the orientation of $\Lambda$. To distinguish $\star_i$ from $\bullet_i$, we will refer to it as a ``marked point'' rather than a base point.

We associate variables to the base points and Reeb chords as follows: 
\begin{itemize}
\item
to each marked point $\star_i$, an invertible variable $t_i$; 
\item
to each Reeb chord $a_j$, two variables $q_j$ and $p_j$.
\end{itemize}
Throughout the paper we will write
\[
\SS := \{q_1,p_1,\ldots,q_n,p_n,t_1^{\pm 1},\ldots,t_s^{\pm 1}\}
\]
for the collection of these variables.
We will define a $\Z$-valued grading on each variable in $\SS$ in Section~\ref{ssec:gradings}; for now we assume these gradings are given, and proceed to construct some algebras and $\Z$-modules out of these variables.

Let
\[
\A = \Z\langle q_1,\ldots,q_n,t_1^{\pm},\ldots,t_s^{\pm}\rangle
\]
be the noncommutative graded algebra generated by the $q$'s and $t$'s: as a $\Z$-module, this is freely generated by words in the $q$'s and $t$'s (i.e., words in the elements of $\SS$ besides $p_1,\ldots,p_n$). This is the algebra underlying the usual Chekanov--Eliashberg DGA of $\Lambda$ (more precisely, the ``fully noncommutative'' version of the C-E DGA).

Next define
\[
\sftold = \Z\langle q_1,\ldots,q_n,p_1,\ldots,p_n,t_1^{\pm},\ldots,t_s^{\pm}\rangle:
\]
as a $\Z$-module, this is freely generated by words in the elements of $\SS$. This has a filtration $\sftold = \F^0\sftold \supset \F^1\sftold \supset \F^2\sftold \supset \cdots$, where $\F^k\sftold$ is the $\Z$-submodule freely generated by words with at least $k$ $p$'s. We can use this filtration to define the completion
\[
\sft = \widehat{\sftold}
\]
whose elements are formal sums of the form $\sum_{k=0}^\infty z_k$ with $z_k \in \F^k\sftold$. The algebra $\sft$ is the ``LSFT algebra'' from \cite{SFT}.  Note that $\sft$ inherits the filtration from $\sftold$, and that $\A \cong \sft/\F^1\sft$.

We can take two successive quotients of each of $\A$ and $\sft$. The first is the cyclic quotient, obtained by quotienting by the $\Z$-submodule generated by commutators $[x,y] = xy-(-1)^{|x||y|}yx$. This produces the graded $\Z$-modules (not $\Z$-algebras)
\[
\A\tocyc, \sft\tocyc:
\]
each of these is generated as a $\Z$-module by cyclic words in the $q$'s and $t$'s (and $p$'s for $\sft\tocyc$). 
Given a word $v_1\cdots v_k \in \sft$ with $v_i \in \SS$ for all $i$, any cyclic permutation $v_{j+1}\cdots v_k v_1 \cdots v_j$ of the word represents the same element in $\sft\tocyc$, up to a sign:
\[
v_{j+1}\cdots v_k v_1 \cdots v_j = (-1)^{(|v_1|+\cdots+|v_j|)(|v_{j+1}|+\cdots+|v_k|)} v_1\cdots v_k \in \sft\tocyc.
\]

The second quotient of $\A$ and $\sft$ is the commutative quotient, obtained by quotienting by the subalgebra generated by commutators. This produces the graded $\Z$-algebras
\begin{align*}
\A\tocomm &= \Z[q_1,\ldots,q_n,t_1^{\pm 1},\ldots,t_s^{\pm 1}] \\
\sft\tocomm &= \Z[q_1,\ldots,q_n,p_1,\ldots,p_n,t_1^{\pm 1},\ldots,t_s^{\pm 1}].
\end{align*}
These are graded polynomial rings with the usual commutation relations $yx=(-1)^{|x||y|}xy$.

\subsection{Running example}
\label{ssec:hopf-ex}

For the constructions in the remainder of this section, it may help to have an illustrative running example of a pointed Legendrian link. We will use the Legendrian positive Hopf link shown in Figure~\ref{fig:hopf-ex}. This is a two-component link such that each component has rotation number $0$, with three base points distributed between the two components. There are four Reeb chords $a_1,a_2,a_3,a_4$. The Chekanov--Eliashberg algebra of this link is
\[
\A = \Z\langle q_1,q_2,q_3,q_4,t_1^{\pm 1},t_2^{\pm 1},t_3^{\pm 1}\rangle
\]
and the LSFT algebra $\sft$ is the completion of
\[
\Z\langle q_1,q_2,q_3,q_4,p_1,p_2,p_3,p_4,t_1^{\pm 1},t_2^{\pm 1},t_3^{\pm 1}\rangle.
\]

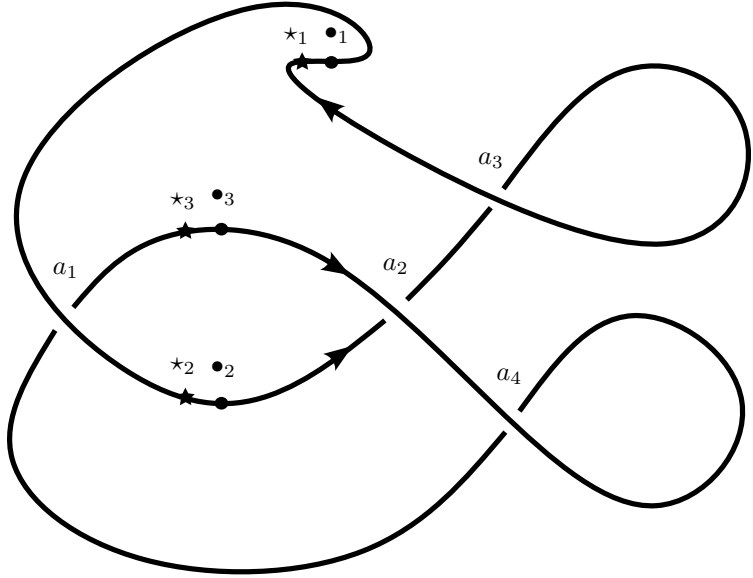
\begin{figure}
\labellist
\small\hair 2pt
\pinlabel $\bullet_1$ at 90 147 
\pinlabel $\star_1$ at 79 147
\pinlabel $\bullet_2$ at 59 56
\pinlabel $\star_2$ at 48 57
\pinlabel $\bullet_3$ at 59 103
\pinlabel $\star_3$ at 48 102
\pinlabel $a_1$ at 16 83
\pinlabel $a_2$ at 106 84
\pinlabel $a_3$ at 132 113
\pinlabel $a_4$ at 137 54
\endlabellist
\centering
\includegraphics[height=3in]{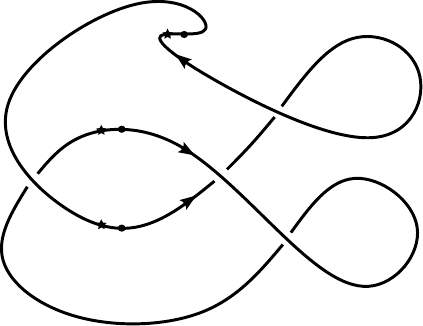}
\caption{
Hopf link in the $xy$ projection. Reeb chords are labeled $a_1,a_2,a_3,a_4$. The three base points $\bullet_i$ and three marked points $\star_i$ are as marked. The zigzag near $\bullet_1$ ensures that the tangent vectors at all of the base points are parallel, for grading purposes; see Remark~\ref{rmk:Maslov}.
}
\label{fig:hopf-ex}
\end{figure}

\subsection{Capping paths}
\label{ssec:capping}

We now again consider a general pointed Legendrian link $\Lambda$. Our next step is to
construct a collection of capping paths, which are paths on $\Lambda$ associated to each of the elements of $\SS$. These paths are closely related to broken closed strings, to be discussed in Section~\ref{ssec:bcs}, and will also be used to define gradings in Section~\ref{ssec:gradings}.

The capping paths 
are defined using the positions of the base points $\bullet_1,\ldots,\bullet_s$. Recall that we have a marked point $\star_i$ immediately preceding each base point $\bullet_i$ as we follow the orientation of $\Lambda$. Removing $\{\star_1,\ldots,\star_s\}$ from $\Lambda$ disconnects $\Lambda$ into a disjoint union of $s$ oriented segments, each containing one of the marked points $\bullet_i$ (at the beginning of the segment). We define two functions
\[
e_+,e_- :\thinspace \SS \to \{1,\ldots,s\}
\]
as follows:
\begin{itemize}
\item
$e_+(t_i) = e_-(t_i^{-1}) = i$;
\item
$e_-(t_i) = e_+(t_i^{-1})$ is the number of the marked point immediately preceding $\bullet_i$ as we follow the orientation of $\Lambda$: that is, if we traverse $\Lambda$ beginning at $\bullet_i$ and traveling against the orientation of $\Lambda$, then $\bullet_{e_-(t_i)}$ is the first base point we encounter after $\bullet_i$;
\item
$e_+(p_j) = e_-(q_j)$ is the number of the marked point on the same segment as the positive endpoint $a_j^+$ of the Reeb chord $a_j$, when the points $\star_1,\ldots,\star_s$ are removed from $\Lambda$;
\item
$e_-(p_j) = e_+(q_j)$ is the number of the marked point on the same segment as the negative endpoint $a_j^-$ of the Reeb chord $a_j$, when the points $\star_1,\ldots,\star_s$ are removed from $\Lambda$.
\end{itemize}

We can now define capping paths $\gamma_v$ for each $v \in \{q_1,p_1,\ldots,q_n,p_n,t_1^{\pm},\ldots,t_s^{\pm}\}$. In all cases, $\gamma_v$ will be a path beginning at $\bullet_{e_-(v)}$ and ending at $\bullet_{e_+(v)}$. For $v=t_i$, we define $\gamma_{t_i}$ to be the oriented segment of $\Lambda$ beginning at $\bullet_{e_-(t_i)}$ and ending at $\bullet_{e_+(t_i)} = \bullet_i$; for $v=t_i^{-1}$, we define $\gamma_{t_i^{-1}}$ to be the reverse of $\gamma_{t_i}$.

For $v=q_j$ and $v=p_j$, the capping path $\gamma_v$ has a discontinuity at the Reeb chord $a_j$ and is constructed from two ``half capping paths''. Specifically, define $\gamma_j^{\pm}$ to be the embedded paths in $\Lambda$, following the orientation of $\Lambda$, from $\bullet_{e_\pm(p_j)} = \bullet_{e_\mp(q_j)}$ to $a_j^\pm$. We then set
\begin{align*}
\gamma_{p_j} &= \gamma_j^- \cdot (-\gamma_j^+) \\
\gamma_{q_j} &= \gamma_j^+ \cdot (-\gamma_j^-),
\end{align*}
where $\cdot$ denotes path concatenation. That is, $\gamma_{p_j}$ is the path $\gamma_j^-$ followed by the reverse of the path $\gamma_j^+$, and similarly for $\gamma_{q_j}$.
We note that the path $\gamma_{p_j}$ has been constructed to begin at $\bullet_{e_-(p_j)}$, have a discontinuity from $a_j^-$ to $a_j^+$, and end at $\bullet_{e_+(p_j)}$. Similarly, $\gamma_{q_j}$, which is the same path in reverse, begins at $\bullet_{e_-(q_j)}$, has a discontinuity from $a_j^+$ to $a_j^-$, and ends at $\bullet_{e_+(q_j)}$.

\begin{example}
For the Hopf link from Figure~\ref{fig:hopf-ex}, the capping paths for $q_1$ and $t_1$ are shown in Figure~\ref{fig:capping-paths}. We have $e_-(q_1)=1$, $e_+(q_1)=3$, and $\gamma_{q_1} = \gamma_1^+ \cdot (-\gamma_1^-)$ begins at $\bullet_1$ and ends at $\bullet_3$, with a discontinuity going down the Reeb chord at $a_1$; and $e_-(t_1)=2$, $e_+(t_1)=1$, and $\gamma_{t_1}$ begins at $\bullet_2$ and ends at $\bullet_1$.

\begin{figure}
\labellist
\small\hair 2pt
\pinlabel $\bullet_1$ at 88 148
\pinlabel ${\color{blue} \gamma_1^+}$ at 12 141
\pinlabel ${\color{blue} -\gamma_1^-}$ at 66 15
\pinlabel $\bullet_3$ at 59 105
\pinlabel $\bullet_2$ at 308 57
\pinlabel $\bullet_1$ at 337 148
\pinlabel ${\color{blue} \gamma_{t_2}}$ at 260 141 
\endlabellist
\centering
\includegraphics[width=\textwidth]{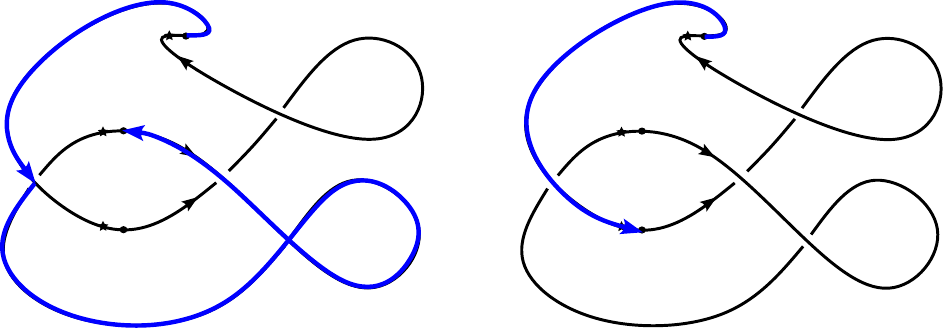}
\caption{
Capping paths $\gamma_{q_1}$ (the concatenation of $\gamma_1^+$ and $-\gamma_1^-$) and $\gamma_{t_2}$, for the Hopf link from Figure~\ref{fig:hopf-ex}.
}
\label{fig:capping-paths}
\end{figure}
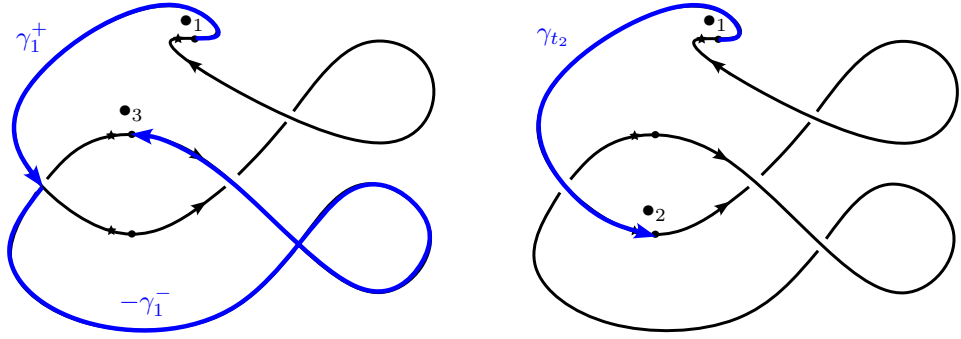
\end{example}

\subsection{Gradings}
\label{ssec:gradings}

Here we discuss how to use the capping paths defined in Section~\ref{ssec:capping} to give $\Z$ gradings to $t_i^{\pm 1},q_j,p_j$, the elements of $\SS$. This will induce gradings on the various algebraic structures $\A,\sft,\sft\tocyc,\sft\tocomm$ associated to a Legendrian link $\Lambda$ in Section~\ref{ssec:algebras}.

In combinatorial form, this story is well-known in the case where $\Lambda$ has a single component with a single base point: see e.g.\ the papers \cite{ENS,SFT}, which build on the original construction of a $\Z/(2\rot(\Lambda)\Z)$ grading by Chekanov \cite{Che}. When $\Lambda$ has multiple components with a single base point on each, the grading is often defined in terms of the front projection of $\Lambda$, see e.g.\ \cite{CLI,NRSSZ}, implicitly using a correspondence between Reeb chords of $\Lambda$ and crossings and right cusps of the front of $\Lambda$. 

Here we will present the grading in terms of the $xy$ projection rather than the front projection, which is more natural geometrically. This may be of independent interest since to this author's knowledge, a combinatorial description of gradings in the multi-component case in the $xy$ projection has not previously appeared in the literature, even though it is surely well known to experts. Furthermore, we extend the construction of gradings to the case where there are multiple base points on single components of $\Lambda$; this is a setting that naturally arises in considering DGA maps induced by exact Lagrangian cobordisms, see e.g.\ \cite{Pan,CN}.

Let $\Lambda$ be a pointed Legendrian link with base points $\bullet_1,\ldots,\bullet_s$. In any case when $\Lambda$ has more than one base point, the grading depends on an auxiliary piece of information, a Maslov potential. The oriented unit tangent vector in $\R^2$ to $\Pi_{xy}(\Lambda)$ at each base point is an element of $S^1$, which we identify with $\R/\Z$ in the usual way: to be concrete, the vector obtained from the positive $x$ direction $\partial/\partial x$ by a counterclockwise $\theta$ rotation is associated with $\frac{\theta}{2\pi} \in \R/\Z$.

\begin{definition}
A \textit{Maslov potential} on $\Lambda$ is a map $m :\thinspace \{\bullet_1,\ldots,\bullet_s\} \to \R$ such that $m(\bullet_i)$ is a lift of the oriented unit tangent vector to $\Pi_{xy}(\Lambda)$ at $\Pi_{xy}(\bullet_i)$ from $\R/\Z$ to $\R$.
\end{definition}

\noindent
The gradings will be defined in terms of differences of Maslov potentials, and so the collection of gradings for $\Lambda$ will be an affine space based on $\Z^{s-1}$.

\begin{remark}
In practice, a convenient way to encode a Maslov potential in a link diagram is to perturb the $xy$ projection of $\Lambda$ by a planar isotopy in a neighborhood of each base point $\bullet_i$, so that the tangent vector at $\bullet_i$ rotates in $\R$ from $m(\bullet_i)$ at the beginning of the isotopy to $0$ at the end. The end result is a link diagram for $\Pi_{xy}(\Lambda)$ where the tangent direction at each base point is $\partial/\partial x$. In this new link diagram, there is an obvious Maslov potential given by $m(\bullet_i) = 0$ for all $i$. All of the gradings that we will define below are unchanged by the isotopy. Consequently, in all examples in this paper, we will choose link diagrams where the tangent direction at each base point is $\partial/\partial x$, where the Maslov potential is understood to be identically $0$. See Figure~\ref{fig:hopf-ex} for an illustration.
\label{rmk:Maslov}
\end{remark}

For an immersed oriented curve $\gamma$ in $\R^2$, define $r(\gamma) \in \R$ to be the number of counterclockwise revolutions in $S^1$ made by the unit tangent vector to $\gamma$ as we traverse $\gamma$; note that in general this will not be an integer. Recall that in Section~\ref{ssec:capping}, we constructed capping paths $\gamma_{t_i}$ associated to $\bullet_i$ and half capping paths $\gamma_j^\pm$ associated to the Reeb chord $a_j$; each of these, when composed with the projection $\Pi_{xy} :\thinspace \R^3\to\R^2$, yields an immersed oriented curve in $\R^2$. We will abuse notation and write $r(\gamma_{t_i})$ and $r(\gamma_j^{\pm})$ for $r(\Pi_{xy}\circ\gamma_{t_i})$ and $r(\Pi_{xy}\circ\gamma_j^{\pm})$.

\begin{definition}
For $v\in\SS$, we define:
\[
|v| = \left\lfloor 2\left(-r(\gamma_v) + m(\bullet_{e_+(v)}) - m(\bullet_{e_-(v)})\right) \right\rfloor.
\]
\label{def:gradings}
\end{definition}
Here if $v=p_j$ or $v=q_j$, then the full capping path $\gamma_v$ has a discontinuity, and by $r(\gamma_v)$ we mean the sum of $r$ over each of the two half capping paths:
\begin{align*}
r(\gamma_{p_j}) &= r(\gamma_j^-) - r(\gamma_j^+) \\
r(\gamma_{q_j}) &= r(\gamma_j^+) - r(\gamma_j^-).
\end{align*}

\begin{example}
For the running example from Section~\ref{ssec:hopf-ex}, we see by inspection of Figure~\ref{fig:capping-paths} that $r(\gamma_1^+) \approx \frac{7}{8}$ and $-r(\gamma_1^-)=r(-\gamma_1^-) \approx -\frac{1}{8}$, and so $|q_1| = -2$. The full grading in this example is:
\begin{align*}
|q_1|&=-2 & |p_1| &= 1 & |q_2| &= 0 & |p_2| &= -1 \\
|q_3| &= 1 & |p_3| &= -2 & |q_4| &= 1 & |p_4| &= -2 \\
|t_1| &=2 & |t_2| &= -2 & |t_3| &= 0.
\end{align*}

\end{example}

\begin{proposition}
\begin{enumerate}
\item
For any base point $\bullet_i$, $|t_i|$ is even and $|t_i^{-1}|=-|t_i|$. If $\Lambda_0$ is a component of $\Lambda$ and $\bullet_{i_1},\ldots,\bullet_{i_m}$ are the base points on component $\Lambda_0$, then
\[
|t_{i_1}|+\cdots+|t_{i_m}| = -2\rot(\Lambda_0).
\]
\item
For any Reeb chord $a_j$, we have 
\[
|q_j|+|p_j|=-1.
\]
Furthermore, $|q_j|$ is even (resp.\ odd), and $|p_j|$ is odd (resp.\ even), if the crossing corresponding to $a_j$ in the oriented link diagram given by the $xy$ projection of $\Lambda$ is positive (resp.\ negative).
\item
If we change the Maslov potential $m$ by adding $(n_1,\ldots,n_s) \in \Z^s$ to $m$, then the gradings change as follows:
\begin{align*}
\Delta |t_i| &= 2\left(n_{e_+(t_i)}-n_{e_-(t_i)}\right) \\
\Delta |q_j| &= -\Delta |p_j| = 2\left(n_{e_+(q_j)}-n_{e_-(q_j)}\right).
\end{align*}
\end{enumerate}
\end{proposition}

\begin{proof}
We first consider item (1).
As we traverse $\gamma_{t_i}$, the tangent direction in $\R^2$ changes in $\R/\Z=S^1$, though not necessarily in a way that lifts continuously to $\R$, from $m(\bullet_{e_-(t_i)})$ to $m(\bullet_{e_+(t_i)})$. Thus
\[
r(\gamma_{t_i}) \equiv m(\bullet_{e_+(t_i)}) - m(\bullet_{e_-(t_i)}) \pmod{\Z}
\]
and the expression inside the floor function in Definition~\ref{def:gradings} is an even integer; it follows that $|t_i|$ is an even integer. Since $\gamma_{t_i^{-1}}$ is the reverse of $\gamma_{t_i}$, we also deduce from Definition~\ref{def:gradings} that $|t_i^{-1}| = -|t_i|$.

For the statement about the gradings of the base points on a component $\Lambda_0$, suppose that we encounter $\bullet_{i_1},\ldots,\bullet_{i_m}$ in order as we traverse $\Lambda_0$. Then $e_+(t_{i_j}) = i_j$, $e_-(t_{i_j}) = i_{j-1}$ (where $j-1$ is taken mod $m$), and
\[
\sum_{j=1}^m |t_{i_j}| = \sum_j 2(-r(\gamma_{t_{i_j}})+m(\bullet_{i_j})-m(\bullet_{i_{j-1}}))
= -2 \sum_j r(\gamma_{t_{i_j}}).
\]
Since $\Lambda_0$ is the union of $\gamma_{t_{i_j}}$ over $j$, this last sum is $\rot(\Lambda_0)$, as desired.

We next consider item (2). For $v=p_j$ and $v=q_j$, the expressions inside the parentheses in Definition~\ref{def:gradings} satisfy
\[
-r(\gamma_{p_j}) + m(\bullet_{e_+(p_j)}) - m(\bullet_{e_-(p_j)})
= 
-\left(
-r(\gamma_{q_j}) + m(\bullet_{e_+(q_j)}) - m(\bullet_{e_-(q_j)})
\right)
\]
and the two sides of this equality are both equal in $\R/\Z$ to the difference between the tangent direction at $a_j^-$ and the tangent direction at $a_j^+$. Since we are assuming that all crossings of $\Pi_{xy}(\Lambda)$ are transverse, this quantity is neither an integer nor a half-integer. It follows that $|q_j|+|p_j| = -1$. To show that the parity of $|q_j|$ agrees with the sign of the crossing for $a_j$, note that the difference between the tangent directions at $a_j^-$ and $a_j^+$ is between $0$ and $1/2$ or between $1/2$ and $1$ mod $\Z$ depending on the sign of the crossing.

Finally, item (3) follows directly from the definition of the gradings.
\end{proof}

\begin{remark}
When $\Lambda$ has a single connected component with a single base point, the dependence on the Maslov potential $m$ disappears and our gradings agree with the corresponding definitions in \cite[\S 2.3]{SFT}. In particular, the path $\gamma_j$ in \cite{SFT} is our $-\Pi_{xy} \circ \gamma_{q_j}$, and so our definition gives $|q_j| = \lfloor 2r(\gamma_j)\rfloor$, in agreement with \cite{SFT}.
Also, in the examples considered in this paper, the Maslov potential will be understood to be identically $0$ and so the contributions from $m(\cdot)$ disappear from the above formulas; see Remark~\ref{rmk:Maslov}.
\end{remark}

\subsection{Broken closed strings}
\label{ssec:bcs}

In Section~\ref{ssec:capping}, we associated a path in $\Lambda$ to each of $t_i^{\pm},q_j,p_j$. This association can be extended to arbitrary words in these generators by concatenation, resulting in a geometric interpretation of generators of $\sft$ as a $\Z$-module. To make this precise, we next introduce the notion of a (based) broken closed string on $\Lambda$: the capping paths defined in Section~\ref{ssec:capping} are examples of based broken closed strings.

\begin{definition}
Let $\Lambda$ be a pointed Legendrian link. A \textit{based broken closed string} on $\Lambda$
is a path in the Legendrian link $\Lambda$ which begins and ends at (possibly distinct) base points on $\Lambda$, and which is continuous except for the following allowed types of discontinuities:
\begin{itemize}
\item
a discontinuity jumping from one endpoint of a Reeb chord of $\Lambda$ to the other;
\item
a discontinuity jumping from one of the base points $\bullet_1,\ldots,\bullet_s$ to another.
\end{itemize}
A \textit{broken closed string} is a map from $S^1$ to $\Lambda$ that is continuous except for the same types of discontinuities as a based broken closed string.
\end{definition}

Note that there is an obvious broken closed string associated to any based broken closed string, given by gluing together the two endpoints: if the two endpoints of the based broken closed string are distinct, then this produces one more discontinuity in the associated broken closed string.

There is a correspondence between based broken closed strings (resp.\ broken closed strings) on $\Lambda$ and words (resp.\ cyclic words) in the $q_j,p_j,t_i^{\pm}$. Define $\W$ to be the collection of words in $q_1,p_1,\ldots,q_n,p_n,t_1^{\pm},\ldots,t_s^{\pm}$, including the empty word. Also define $\W\tocyc$ to be the collection of cyclic words in the same generators; that is, words up to cyclic permutation. Note that $\W$ and $\W\tocyc$ generate $\sft$ and $\sft\tocyc$, respectively, as $\Z$-modules.

To any word $v_1\cdots v_m \in \W$, where $v_1,\ldots,v_m \in \SS$, we can associate the based broken closed string
\[
\gamma_{v_1} \cdot \gamma_{v_2} \cdot \cdots \cdot \gamma_{v_m}
\]
given by the concatenation of the capping paths of $v_1,\ldots,v_m$. This association is reversible. More precisely, there is a map
\[
w :\thinspace \{\text{based broken closed strings}\} \to \W
\]
defined as follows: traverse the string and successively record $t_i$ when the string crosses $\star_i$ following the orientation of $\Lambda$, $t_i^{-1}$ when the string crosses $\star_i$ against the orientation of $\Lambda$, $p_j$ when the string jumps from the bottom endpoint $a_j^-$ of the Reeb chord $a_j$ to the top endpoint $a_j^+$, and $q_j$ when the string jumps from $a_j^+$ to $a_j^-$. For example, if the based string passes through $\star_1$ positively, then jumps from $a_1^-$ to $a_1^+$, then passes through $\star_1$ negatively, then the associated word is $t_1p_1t_1^{-1}$. We now note that
\[
w(\gamma_{v_1} \cdot \gamma_{v_2} \cdot \cdots \cdot \gamma_{v_m}) = v_1v_2\cdots v_m.
\]

We claim that $w$ induces a bijection between $\W$ and homotopy classes of based broken closed strings. To define homotopy for broken closed strings, consider the quotient $\widetilde{\Lambda}$ of $\Lambda$ obtained by identifying all of the base points $\bullet_1,\ldots,\bullet_s$ with each other to produce a single base point $\bullet$ of $\widetilde{\Lambda}$. Then a broken closed string (respectively based broken closed string) is a map from $S^1$ (resp.\ $[0,1]$) to $\widetilde{\Lambda}$ (resp.\ $(\widetilde{\Lambda},\bullet)$)  whose only discontinuities are jumps from one endpoint of a Reeb chord to the other. We then define homotopy of (based) broken closed strings to mean a homotopy in $\widetilde{\Lambda}$ that preserves all discontinuities at Reeb chords.

Since $\widetilde{\Lambda}$ is a bouquet of circles and removing the marked points $\star_1,\ldots,\star_s$ results in a contractible space, the map $w$ is injective: if two based broken closed strings map under $w$ to the same word, then they are homotopic. We conclude the following.

\begin{proposition}
The map
\[
w :\thinspace \{\text{based broken closed strings up to homotopy}\} \to \W
\]
is a bijection, as is the induced map
\[
w :\thinspace \{\text{broken closed strings up to homotopy}\} \to \W\tocyc.
\]
\end{proposition}

\begin{remark}
\label{rmk:composable}
Consider a cyclic word $v_1\cdots v_m$ in $\W\tocyc$ with $v_i \in \SS$ for each $i$. We say that the cyclic word is \textit{cyclically composable} if
\[
e_+(v_i) = e_-(v_{i+1})
\]
for each $i=1,\ldots,m-1$, and $e_+(v_m) = e_-(v_1)$. Then a word in $\W\tocyc$ is cyclically composable if and only if the corresponding broken closed string has discontinuities only at Reeb chords and not at base points. One can form a graded $\Z$-module $\A^{\operatorname{cyc},\operatorname{comp}}$ generated by composable cyclic words, where we replace the empty word in the generating set by a collection of $r$ ``idempotent'' empty words. This composable module appears in the surgery formulas of \cite{BEE,BEE2}.
\end{remark}

\subsection{The Hamiltonian}
\label{ssec:h}

Given a Legendrian link $\Lambda$, rational Legendrian Symplectic Field Theory counts rigid holomorphic disks in the symplectization $\R\times\R^3$ with boundary on the Lagrangian $\R\times\Lambda$. In a familiar manner dating back to the work of Chekanov \cite{Che}, these holomorphic disks are in one-to-one correspondence with certain immersed disks-with-corners in $(\R^2,\Pi_{xy}(\Lambda))$. We review how to use this correspondence to construct a Hamiltonian in $\sft\tocyc$, following \cite{SFT}.

\begin{definition}[cf.\ {\cite[Definition 2.18]{SFT}}]
Given a diagram-generic Legendrian link $\Lambda$, let $\Delta(\Lambda)$ denote the set of all maps $\Delta$ from a disk $D^2$ equipped with any number of boundary marked points on the boundary to $\R^2$, up to domain reparametrization, such that:
\begin{itemize}
\item
$\Delta$ is continuous;
\item
$\Delta$ is an orientation-preserving immersion away from the boundary marked points;
\item
$\Delta(\partial D^2) \subset \Pi_{xy}(\Lambda)$;
\item
$\Delta$ sends each boundary marked point to a crossing of $\Pi_{xy}(\Lambda)$, and a neighborhood of each boundary marked point injectively to one of the four quadrants at that crossing.
\end{itemize}
\end{definition}

Observe that the oriented boundary of any $\Delta \in \Delta(\Lambda)$ is (the projection to $\R^2$ of) a broken closed string of $\Lambda$, with discontinuities precisely at the images of the boundary marked points. Thus, following Section~\ref{ssec:bcs}, to any $\Delta\in\Delta(\Lambda)$ we can associate a cyclic word
\[
w(\Delta) \in \W\tocyc.
\]
Each of these cyclic words $w(\Delta)$ is cyclically composable in the terminology from Section~\ref{ssec:bcs}, since the corresponding broken closed string $\partial\Delta$ has discontinuities only at Reeb chords and not at marked points.

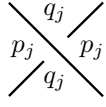
\begin{figure}
\labellist
\small\hair 2pt
\pinlabel $p_j$ at 5 15
\pinlabel $p_j$ at 27 15
\pinlabel $q_j$ at 15 5
\pinlabel $q_j$ at 15 27
\endlabellist
\centering
\includegraphics[height=0.5in]{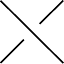}
\caption{Decorating the quadrants at a Reeb chord $a_j$.
}
\label{fig:Reebsigns}
\end{figure}

An easy way to read the cyclic word $w(\Delta)$ from $\Delta$ is as follows. Decorate the four quadrants at any Reeb chord $a_j$ by $p_j,q_j$ as shown in Figure~\ref{fig:Reebsigns}. Now traverse the boundary $\partial\Delta$, with the orientation induced by the orientation on $\Delta$, and successively read off: at the image $a_j$ of a boundary marked point, $p_j$ or $q_j$ depending on whether a neighborhood of this marked point is sent to a quadrant labeled by $p_j$ or $q_j$; and when $\partial\Delta$ passes through a marked point $\bullet_i$, $t_i$ or $t_i^{-1}$ depending on whether the orientations of $\partial\Delta$ and $\Lambda$ agree at this point. The concatenation of these letters is the cyclic word $w(\Delta)$.

Each $w(\Delta)$ inherits a grading from the gradings in Section~\ref{ssec:gradings}, and we have the following result, cf.\ \cite[Lemma~2.19]{SFT}.

\begin{proposition}
For any immersed disk $\Delta\in\Delta(\Lambda)$, we have
\label{prop:hamdeg}
\[
|w(\Delta)| = -2.
\]
\end{proposition}

\begin{proof}
In the case where $\Lambda$ has a single component and a single base point, this was proven in \cite[Lemma 2.19]{SFT}, and the proof here is very similar. Given $\Delta\in\Delta(\Lambda)$, write $w(\Delta) = v_1v_2\cdots v_k$, where each $v_\ell \in \SS$.  For each $v_\ell$ that is a $q_j$ or $p_j$, $\Delta$ has a corner at the Reeb chord $a_j$; let $\theta_\ell \in (0,\pi)$ denote the angle of the corner of $\Delta$ at $a_j$, and let $\eta_\ell = \pi-\theta_\ell \in (0,\pi)$ be the exterior angle of the curved polygon $\Delta$ at $a_j$. Then the angle formed by the tangent vectors to $\Pi_{xy}(\Lambda)$ at $a_j$ is
\[
r(\Pi_{xy}\circ\gamma_{v_\ell}) - m(\bullet_{e_+(v_\ell)}) + m(\bullet_{e_-(v_\ell)}) \equiv \frac{\theta_\ell}{2\pi} \equiv -\frac{\eta_\ell}{2\pi} \pmod{\frac{1}{2}\Z}
\]
and so
\begin{align*}
|v_\ell|+\frac{\eta_\ell}{\pi} & = \frac{\eta_\ell}{\pi}+\lfloor -2(r(\Pi_{xy}\circ\gamma_{v_\ell}) - m(\bullet_{e_+(v_\ell)}) + m(\bullet_{e_-(v_\ell)}) \rfloor \\
&= -2(r(\Pi_{xy}\circ\gamma_{v_\ell}) - m(\bullet_{e_+(v_\ell)}) + m(\bullet_{e_-(v_\ell)}).
\end{align*}
Note that this equation also holds if $v_\ell = t_i^{\pm 1}$ if we set $\eta_\ell = 0$ in this case.

Now sum over $\ell$: then the differences $-m(\bullet_{e_+(v_\ell)}) + m(\bullet_{e_-(v_\ell)})$ telescope since $v_1v_2\cdots v_k$ is cyclically composable, and we are left with
\[
\sum_{\ell=1}^k |v_\ell| + \frac{\sum_{\ell=1}^k \eta_\ell}{\pi} =
-2 \sum_{\ell=1}^k r(\Pi_{xy}\circ\gamma_{v_\ell}).
\]
Observe that the oriented boundary $\partial \Delta$ of $\Delta$ is homotopic as a broken closed string to $\cup_{\ell=1}^k \gamma_{v_\ell}$; thus $\sum_{\ell=1}^k r(\Pi_{xy}\circ\gamma_{v_\ell}) = r(\partial\Delta)$. But since $\Delta$ is an immersed disk, the sum of $2\pi r(\partial\Delta)$ and the external angles $\eta_\ell$ is equal to $2\pi$. Consequently, we have
\[
|w(\Delta)| = \sum_{\ell=1}^k |v_\ell| = -2r(\partial\Delta)-\sum_{\ell=1}^k \frac{\eta_\ell}{\pi} = -2,
\]
as desired.
\end{proof}

We construct the Hamiltonian for $\Lambda$ by summing the cyclic words $w(\Delta)$ over all immersed disks $\Delta\in\Delta(\Lambda)$ to obtain an element of $\sft\tocyc$. To do this precisely, we need to associate signs to each of these cyclic words. 

To this end, consider an immersed disk $\Delta\in\Delta(\Lambda)$. The disk $\Delta$ must have a corner at some Reeb chord of $\Lambda$ since otherwise it has area $0$ because each component of $\Pi_{xy}(\Lambda)$ bounds a region of signed area $0$; see also Definition/Proposition~\ref{defprop:h} below. Thus we can write the cyclic word $w(\Delta)$ as $v_1v_2\cdots v_k$ with $v_i \in \SS$ for all $i$ and $v_k \in \{q_1,p_1,\ldots,q_n,p_n\}$.
For $i=1,\ldots,k$, define
\[
\epsilon_i(\Delta) = \begin{cases}
+1 & v_i = t_j^{\pm 1} \text{ for some } j \\
+1 & v_i = q_j \text{ or } v_i=p_j \text{ and the corner of $\Delta$ at $v_i$ is unshaded} \\
-1 & v_i = q_j \text{ or } v_i=p_j \text{ and the corner of $\Delta$ at $v_i$ is shaded},
\end{cases}
\]
where ``shaded'' and ``unshaded'' refer to the shading of quadrants at the crossing $a_j$ shown in Figure~\ref{fig:signs}. We also define another sign $\epsilon'(\Delta;v_k)$ as follows: if the orientation of $\partial\Delta$ immediately after the corner $v_k$ agrees with the orientation of $\Lambda$, then $\epsilon'(\Delta;v_k) = +1$; if it disagrees, then $\epsilon'(\Delta;v_k) = -1$. Finally, we define
\[
\tilde{w}(\Delta) = \left(\epsilon_1(\Delta)\cdots \epsilon_k(\Delta)\epsilon'(\Delta;v_k)\right) v_1v_2\cdots v_k \in \sft\tocyc.
\]

\begin{figure}
\labellist
\small\hair 2pt
\pinlabel $p$ at 25 15
\pinlabel $p$ at 5 15
\pinlabel $q$ at 15 25
\pinlabel $q$ at 15 5
\pinlabel $p$ at 101 15
\pinlabel $p$ at 81 15
\pinlabel $q$ at 91 25
\pinlabel $q$ at 91 5
\endlabellist
\centering
\includegraphics[height=0.75in]{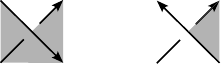}
\caption{Orientation signs. On the left, a positive crossing ($|q|$ even, $|p|$ odd); on the right, a negative crossing ($|q|$ odd, $|p|$ even). Each corner of an immersed disk contributes a $+$ sign if it occupies one of the unshaded quadrants or a $-$ sign if it occupies one of the shaded quadrants.
}
\label{fig:signs}
\end{figure}
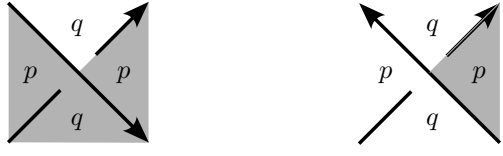

\begin{lemma}
For any $\Delta\in\Delta(\Lambda)$, the element $\tilde{w}(\Delta) \in \sft\tocyc$ is independent of the choice of the representative word $v_1v_2\cdots v_k$ for $w(\Delta)$.
\end{lemma}

\begin{proof}
We need to check that if we replace $v_1v_2\cdots v_k$ by a cyclic permutation $v_{i+1}\cdots v_kv_1\cdots v_i$, then the sign $(-1)^{(|v_1|+\cdots+|v_i|)(|v_{i+1}|+\cdots+|v_k|)}$ associated to this cyclic permutation is canceled by the product of the signs $\epsilon'(\Delta;v_k)\epsilon'(\Delta;v_i)$. This follows essentially immediately from the fact that $w(\Delta)$ has even degree from Proposition~\ref{prop:hamdeg}; see \cite[Lemma 3.10]{SFT} for details.
\end{proof}

\begin{defprop}
The \textit{Hamiltonian} $h \in \sft\tocyc$ is defined by
\[
h = \sum_{\Delta\in\Delta(\Lambda)} \tilde{w}(\Delta).
\]
This is a well-defined element and satisfies $|h|=-2$ and $h \in \F^1\sft\tocyc$.
\label{defprop:h}
\end{defprop}

\begin{proof}
The fact that $|h|=-2$ follows from Proposition~\ref{prop:hamdeg}. To see that $h\in\F^1\sft\tocyc$, we note that for any $\Delta\in\Delta(\Lambda)$, the standard Stokes' Theorem argument (cf.\ \cite[Lemma~2.22]{SFT}) shows that the area of $\Delta$ is equal to the sum of the heights of every Reeb chord appearing as a $p$ in $w(\Delta)$, minus the sum of the heights of every Reeb chord appearing as a $q$ in $w(\Delta)$. It follows that for each $\Delta\in\Delta(\Lambda)$, $w(\Delta)$ contains at least one $p$: $w(\Delta) \in \F^1\sft\tocyc$. Furthermore, for any $k \geq 1$, only finitely many $\Delta\in\Delta(\Lambda)$ are such that $w(\Delta)$ has at most $k$ $p$'s; it follows that $h \in \F^1\sft\tocyc$.
\end{proof}

\begin{example}
For the example from Section~\ref{ssec:hopf-ex}, we have \label{ex:h}
\[
h = p_3+t_1t_2p_3+t_1q_1t_3q_2p_3+p_4+t_3^{-1}p_4+q_2t_2^{-1}q_1p_4+t_2p_2t_3^{-1}p_1.
\]
\end{example}

When we define the $L_\infty$ structure in Section~\ref{sec:l-infty}, it will be helpful to subdivide $h$ into parts according to the number of positive punctures of the immersed disks. For $k \geq 1$, define
\[
\Delta_k(\Lambda) = \{\Delta \in \Delta(\Lambda) \,|\, w(\Delta) \text{ contains exactly } k~p\text{'s}\}
\]
and
\[
h_k = \sum_{\Delta\in\Delta_k(\Lambda)} \tilde{w}(\Delta).
\]
Then $h_k$ counts disks with $k$ positive punctures, and $h=\sum_{k\geq 1} h_k$.

\begin{remark}[sign conventions]
The sign convention that we use in this paper for the Hamiltonian $h$, determined by the shading of corners in Figure~\ref{fig:signs}, is different from the sign convention in \cite{SFT}, which is given in \cite[Figure 4]{SFT}. The difference comes from our desire in this paper to have $h_1$ induce the LCH differential $\partial$ on the Chekanov--Eliashberg DGA with its standard sign convention as described in Section~\ref{ssec:sft} below; see also Remark~\ref{rmk:d-partial}. By contrast, the sign convention for $h$ used in \cite{SFT} corresponds to an unusual choice of signs for the LCH differential: see \cite[Appendix~A]{SFT}.
\label{rmk:sign-conventions}
\end{remark}

\subsection{SFT bracket}
\label{ssec:sft}

We next introduce another ingredient necessary to define the $L_\infty$ structure, the SFT bracket
\[
\{\cdot,\cdot\} :\thinspace \sft\tocyc \otimes \sft\tocyc \to \sft\tocyc.
\]
Geometrically, the SFT bracket is defined as follows. Consider two cyclic words in $\sft\tocyc$, represented as broken closed strings. If one of the cyclic words involes $q_j$ and the other involves $p_j$ for some $j$, then the corresponding broken closed strings both have a discontinuity at the Reeb chord $a_j$, traversed in opposite directions. We can glue together the broken closed strings at $a_j$ to obtain another broken closed string that no longer has the discontinuity at $a_j$. Summing over all possible ways to glue the two broken closed strings at a Reeb chord yields the SFT bracket of the original cyclic words.

We now give the precise definition of the SFT bracket, with signs. First note that given two cyclic words $v_1\cdots v_k$ and $u_1\cdots u_\ell$ where $v_1,\ldots,v_k,u_1,\ldots,u_\ell \in \SS$, we have that for each $i=0,\ldots,k-1$ and each $j=1,\ldots,\ell$,
$v_1\cdots v_k = \sigma_{v,i}  v_{i+1}\cdots v_kv_1\cdots v_i$ and $u_1\cdots u_\ell = \sigma_{u,j-1} u_{j}\cdots u_\ell u_1\cdots u_{j-1}$, where
\begin{align*}
\sigma_{v,i} &:= (-1)^{(|v_1|+\cdots+|v_i|)(|v_{i+1}|+\cdots+|v_k|)} \\
\sigma_{u,j-1} &:= (-1)^{(|u_1|+\cdots+|u_{j-1}|)(|u_{j}|+\cdots+|u_\ell|)}.
\end{align*}
Also, for $v,u \in \SS$, define
\[
\{v,u\} := \begin{cases} 
1 & \text{if $v=p_j$ and $u=q_j$ for some $j$} \\
-1 & \text{if $v=q_j$ and $u=p_j$ for some $j$} \\
0 & \text{otherwise}.
\end{cases}
\]

\begin{definition}
The \textit{SFT bracket} $\{\cdot,\cdot\} :\thinspace \sft\tocyc \otimes \sft\tocyc \to \sft\tocyc$ is the $\Z$-bilinear map defined on generators $v_1\cdots v_k,u_1\cdots u_\ell$ of $\sft\tocyc$ as follows: \label{def:SFTbracket}
\begin{align*}
\{v_1\cdots v_k,&u_1\cdots u_\ell\} \\
&= \sum_{i=0}^{k-1} \sum_{j=1}^{\ell}
\sigma_{v,i} \sigma_{u,j-1} \{v_i,u_{j}\}
v_{i+1}\cdots v_kv_1\cdots v_{i-1}u_{j+1}\cdots u_\ell u_1 \cdots u_{j-1}.
\end{align*}
\end{definition}

\noindent
Note that, because of the presence of the signs $\sigma_{v,i}$ and $\sigma_{u,j-1}$, the SFT bracket is well-defined independent of the choice of representative for the cyclic words $v_1\cdots v_k$ and $u_1\cdots u_\ell$.

\begin{example}
For the example from Section~\ref{ssec:hopf-ex}, we have \label{ex:bracket}
\[
\{t_1q_1t_3q_2p_3,t_2p_2t_3^{-1}p_1\} =t_3q_2p_3t_1t_2p_2t_3^{-1}- p_3t_1q_1p_1t_2  = 
t_1t_2p_2 q_2p_3-t_1q_1p_1t_2p_3.
\]
Note that $t_1q_1t_3q_2p_3$ and $t_2p_2t_3^{-1}p_1$ are both terms in the Hamiltonian $h$ for this link, and the SFT bracket of these two terms involves the string cobracket $\delta h$ of $h$; cf.\ Example~\ref{ex:qme} below.
\end{example}

The SFT bracket that we have presented here is nearly identical to the construction in \cite[\S 3]{SFT}, with the small difference that we are allowing for more than one $t$ variable. As in \cite{SFT}, we have the following.

\begin{proposition}[{\cite[Proposition 3.4]{SFT}}]
The SFT bracket on $\sft\tocyc$ has degree $+1$ 
\label{prop:sft-bracket}
and satisfies the following properties: 
\begin{itemize}
\item
symmetry:
\[
\{y,x\} = (-1)^{|x||y|+|x|+|y|}\{x,y\};
\]
\item
Leibniz: 
\begin{align*}
\{x,yz\} &= \{x,y\}z+(-1)^{(|x|+1)|y|}y\{x,z\} \\
\{xy,z\} &= x\{y,z\}+(-1)^{|y|(|z|+1)}\{x,z\}y;
\end{align*}
\item
Jacobi:
\[
\{x,\{y,z\}\} + (-1)^{(|x|+1)(|y|+|z|)}\{y,\{z,x\}\} + (-1)^{(|z|+1)(|x|+|y|)}\{z,\{x,y\}\} = 0.
\]
\end{itemize}
\end{proposition}

The SFT bracket in Definition~\ref{def:SFTbracket}, which is defined on $\sft\tocyc$, descends to a map on $\sft\tocomm$, which we also call the SFT bracket. On $\sft\tocomm$, there is a simpler way to define the SFT bracket than on $\sft\tocyc$. For clarity, we separate out this ``definition'', even though in fact it follows from Definition~\ref{def:SFTbracket}.

\begin{definition}
The \textit{SFT bracket} $\{\cdot,\cdot\} :\thinspace \sft\tocomm \otimes \sft\tocomm \to \sft\tocomm$ is the Poisson bracket defined on generators by:
\begin{gather*}
\{p_i,q_j\} = -\{q_j,p_i\} = \delta_{ij} \\
\{p_i,p_j\} = \{q_i,q_j\} = \{t_i^{\pm 1},p_j\} = \{t_i^{\pm 1},q_j\} = \{t_i^{\pm 1},t_j^{\pm 1}\} = 0.
\end{gather*}
That is, $\{\cdot,\cdot\}$ is uniquely determined by its action on generators, along with symmetry and the Leibniz rule.
\end{definition}

We conclude this subsection by reviewing the differential $\partial$ in the (commutative) Chekanov--Eliashberg DGA $(\A\tocomm,\partial)$ and its relation to the Hamiltonian and SFT bracket. Recall that $\Delta_1(\Lambda)$ is the collection of immersed disks with boundary on $\Pi_{xy}(\Lambda)$ and convex corners, exactly one of which is labeled by a $p$ (a ``positive corner''). For a given $p_j \in \SS$, let $\Delta_1(\Lambda;p_j)$ denote the subset of $\Delta_1(\Lambda)$ consisting of disks whose positive corner is at $p_j$.

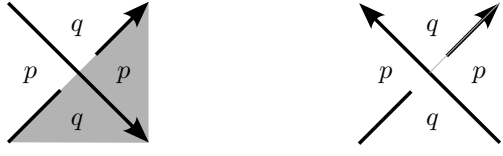
\begin{figure}
\labellist
\small\hair 2pt
\pinlabel $p$ at 25 15
\pinlabel $p$ at 5 15
\pinlabel $q$ at 15 25
\pinlabel $q$ at 15 5
\pinlabel $p$ at 101 15
\pinlabel $p$ at 81 15
\pinlabel $q$ at 91 25
\pinlabel $q$ at 91 5
\endlabellist
\centering
\includegraphics[height=0.75in]{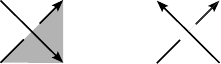}
\caption{Orientation signs for the differential in the Chekanov--Eliashberg DGA. On the left, a positive crossing ($|q|$ even, $|p|$ odd); on the right, a negative crossing ($|q|$ odd, $|p|$ even). Each corner of an immersed disk contributes a $+$ sign if it occupies one of the unshaded quadrants or a $-$ sign if it occupies one of the shaded quadrants.
}
\label{fig:signs-lch}
\end{figure}

Given $\Delta\in\Delta_1(\Lambda;p_j)$, we can associate a word $w'(\Delta)$ by taking $w(\Delta)$ and removing $p_j$: to be precise, $w'(\Delta) = v_1\cdots v_k$ where $v_1,\ldots,v_k \in \{q_1,\ldots,q_n,t_1^{\pm},\ldots,t_s^{\pm}\}$ are the labels of the corners and base points that are encountered when we traverse the boundary of $\Delta$ starting just after $p_j$ and ending just before $p_j$. We can also define a sign $\sgn(\Delta) \in \{\pm 1\}$ to be the product of the ``LCH orientation signs'' associated to the corners of $\Delta$, including $p_j$, as shown in Figure~\ref{fig:signs-lch}: that is, $\sgn(\Delta)$ is $(-1)$ raised to the number of corners of $\Delta$ that are shaded. 

The differential $\partial$ on $\A\tocomm$ is now defined as follows:
\[
\partial(q_j) = \sum_{\Delta\in\Delta_1(\Lambda;p_j)} \sgn(\Delta) w'(\Delta).
\]
Extend $\partial$ to all of $\A\tocomm$ by setting $\partial(t_i^{\pm 1})=\partial(1)=0$ and using the Leibniz rule.

Note that the orientation signs for LCH, as determined by the shadings in Figure~\ref{fig:signs-lch}, are different from the orientation signs for the Hamiltonian $h$, as shown in Figure~\ref{fig:signs}. However, they have been chosen to be compatible in the following sense; see also Remark~\ref{rmk:sign-conventions}.

\begin{proposition}
For any $x\in\A\tocomm$, we have \label{prop:sft-lch}
\[
\partial(x) = \{h_1,x\}.
\]
\end{proposition}

\begin{proof}
Since both $\partial$ and $\{h_1,\cdot\}$ satisfy the Leibniz rule, it suffices to check $\partial(x) = \{h_1,x\}$ for $x\in \{q_1,\ldots,q_n,t_1^{\pm},\ldots,t_s^{\pm}\}$. For $x=t_i^{\pm}$, both sides are $0$. Now suppose that $x=q_j$. Then for each $\Delta\in\Delta_1(\Lambda;p_j)$, we can write $\tilde{w}(\Delta) = \sgn'(\Delta) v_1\cdots v_{k-1}p_j$, where $\sgn'(\Delta)$ is the sign defined in Section~\ref{ssec:h}: in the notation from that subsection, we have $\sgn'(\Delta) = \epsilon_1(\Delta) \cdots \epsilon_k(\Delta) \epsilon'(\Delta;p_j)$ and $v_k=p_j$. Now note that the shadings in Figures~\ref{fig:signs} and~\ref{fig:signs-lch} differ in two quadrants, namely the $p$ quadrants where the orientation of the quadrant's boundary (induced from the standard orientation on the quadrant) disagrees with the orientation of $\Lambda$ after the corner. That is, the difference between the sign $\epsilon_k(\Delta)$ associated to the corner at $p_j$ from Figure~\ref{fig:signs} and the corresponding sign from Figure~\ref{fig:signs-lch} is precisely $\epsilon'(\Delta;p_j)$. It follows that $\sgn'(\Delta) = \sgn(\Delta)$ and
\[
\{\sgn'(\Delta) v_1\cdots v_{k-1}p_j,q_j\} = \sgn(\Delta) v_1\cdots v_{k-1}.
\]
Summing over all $\Delta\in\Delta_1(\Lambda;p_j)$ yields $\partial(q_j) = \{h_1,q_j\}$, as desired.
\end{proof}

\subsection{String coproduct}
\label{ssec:delta}

The final ingredient we need in order to define the $L_\infty$ structure is the string coproduct. This is a map
\[
\delta :\thinspace \sft \to \sft,
\]
descending to $\delta :\thinspace \sft\tocyc \to \sft\tocyc$ and $\delta :\thinspace \sft\tocomm \to \sft\tocomm$, that inserts a $(p,q)$ pair somewhere in the middle of any (based) broken closed string. Here we define $\delta$ and present some of its basic properties, following \cite[\S 3.2]{SFT}. The term ``string coproduct'' is borrowed from string topology \cite{CS}; see also \cite{CL} for a discussion of the relevance of this operation to Symplectic Field Theory. We note that our use of the term is somewhat imprecise: the $\delta$ operation can be viewed as dividing a broken closed string into itself along with a trivial broken string at a Reeb chord, and this resembles but is not identical to the usual coproduct in string topology as introduced by Chas--Sullivan \cite{CS}.

\begin{definition}
A (based) broken closed string $\gamma$ on $\Lambda$ is \textit{generic} if whenever $\gamma(t)$ is at the endpoint of some Reeb chord, $\gamma'(t) \neq 0$; in particular, if $\gamma$ has a Reeb chord discontinuity at $t$, then $\gamma'(t^-) := \lim_{\tau\to t^-}\gamma'(\tau)$ and $\gamma'(t^+) :=\lim_{\tau\to t^+} \gamma'(\tau)$ are nonzero.

A generic (based) broken closed string $\gamma$ \textit{has holomorphic corners} if wherever $\gamma$ has a discontinuity at a Reeb chord, the $xy$ projection of $\gamma$ turns left at the crossing corresponding to the Reeb chord.
\end{definition}

Any (based) broken closed string is homotopic to a generic (based) broken closed string that has holomorphic corners: see Figure~\ref{fig:hol-corners} for an illustration.

\begin{figure}
\labellist
\small\hair 2pt
\pinlabel{$a_1$} at 4 70
\pinlabel{$a_1$} at 271 54
\endlabellist
\centering
\includegraphics[width=\textwidth]{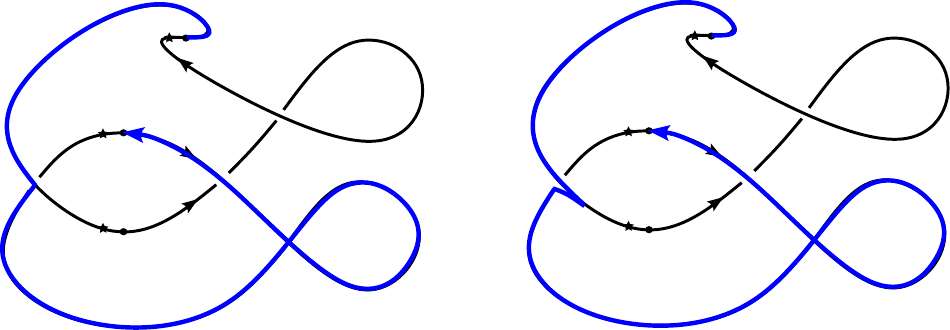}
\caption{
The capping path $\gamma_{q_1}$ from Figure~\ref{fig:capping-paths} (left) does not have a holomorphic corner at $a_1$. It can be perturbed in a neighborhood of $a_1$ to produce a generic based broken closed string that does have a holomorphic corner at $a_1$ (right).
}
\label{fig:hol-corners}
\end{figure}

Now suppose that $\gamma :\thinspace [0,1] \to \Lambda$ is a generic based broken closed string with holomorphic corners. We say that $\tau \in (0,1)$ is \textit{interior Reeb} for $\gamma$ if $\gamma$ does not have a Reeb chord discontinuity at $\tau$ but $\gamma(\tau)$ lies at the endpoint of some Reeb chord of $\Lambda$. In this case, we can define a new based broken closed string $\delta(\gamma;\tau)$ that coincides with $\gamma$ except that at time $\tau$ we insert a jump from $\gamma(\tau)$ to the other endpoint of the Reeb chord and back again. 

As in Section~\ref{ssec:bcs}, the based broken closed string $\gamma$ corresponds to a word $w(\gamma) = v_1\cdots v_m$. Suppose that $\tau$ is interior Reeb for $\gamma$ and $\gamma(\tau)$ lies on the segment of $\gamma$ between $v_i$ and $v_{i+1}$. Then
\[
w(\delta(\gamma;\tau)) = \begin{cases}
v_1\cdots v_i p_jq_j v_{i+1} \cdots v_m & \text{if } \gamma(\tau) = a_j^-, \text{ the negative endpoint of Reeb chord $a_j$} \\
v_1\cdots v_i q_jp_j v_{i+1}\cdots v_m &  \text{if } \gamma(\tau) = a_j^+, \text{ the positive endpoint of Reeb chord $a_j$}.
\end{cases}
\]
Furthermore, we can associate a sign to $\delta(\gamma;\tau)$ as follows. Let $\epsilon_1 = +1$ if $\gamma(\tau) = a_j^-$ and $\epsilon_1 = -1$ if $\gamma(\tau) = a_j^+$. Let $\epsilon_2 = \pm 1$ depending on whether the orientation of $\gamma$ at $\tau$ agrees or disagrees with the orientation of $\Lambda$. Finally, define
\[
\epsilon(\gamma;\tau) = (-1)^{|v_1|+\cdots+|v_i|}\epsilon_1 \epsilon_2.
\]

\begin{definition}
The \textit{string coproduct} $\delta :\thinspace \sft \to \sft$ is defined as follows. Suppose that $w\in\W$ is a word in $\SS$ and let $\gamma$ be any generic based broken closed string with holomorphic corners for which $w(\gamma) = w$. Then
\label{def:delta}
\[
\delta(w) = \sum_{\text{$\tau$ interior Reeb for $\gamma$}} \epsilon(\gamma;\tau) w(\delta(\gamma;\tau)).
\]
\end{definition}

The definition of $\delta(w)$ is independent of the choice of $\gamma$; see \cite[\S 3.2]{SFT}. 

\begin{example}
Here we calculate $\delta(t_2)$, $\delta(q_1)$, and $\delta(p_1)$ for the example from Section~\ref{ssec:hopf-ex}. The path $\gamma_{t_2}$, as shown in Figure~\ref{fig:capping-paths}, passes through the positive endpoint $a_1^+$ of the Reeb chord $a_1$; it follows that $\delta(t_2) = -q_1p_1t_2$. To compute $\delta(q_1)$, note that the path $\gamma_{q_1}$ in Figure~\ref{fig:capping-paths} does not have holomorphic corners. Once we perturb it to have holomorphic corners as shown in the right hand diagram in Figure~\ref{fig:hol-corners}, we see that it passes through $a_1^+$, $a_4^-$, $a_4^+$, and $a_2^+$ in succession. This yields the four terms in
\[
\delta(q_1) = -q_1p_1q_1-q_1p_4q_4+q_1q_4p_4+q_1q_2p_2.
\]
By contrast, to compute $\delta(p_1)$, we note that the path $\gamma_{p_1}$ given by the reverse of $\gamma_{q_1}$ from Figure~\ref{fig:capping-paths} does have holomorphic corners and does not need to be perturbed. This leads to one fewer term in $\delta(p_1)$ than in $\delta(q_1)$:
\[
\delta(p_1) = -q_2p_2p_1-q_4p_4p_1+p_4q_4p_1.
\]
\end{example}

\begin{proposition}[{\cite[Proposition 3.8]{SFT}}]
The map $\delta :\thinspace \sft \to \sft$ has degree $-1$ and satisfies the following properties: \label{prop:delta-properties}
\begin{itemize}
\item
$\delta(xy) = (\delta x)y + (-1)^{|x|}x(\delta y)$;
\item
$\delta^2=0$.
\end{itemize}
\end{proposition}

The map $\delta :\thinspace \sft \to \sft$ descends to maps $\sft\tocyc \to \sft\tocyc$ and $\sft\tocomm \to \sft\tocomm$, which we will also write as $\delta$. Note that $\delta^2=0$ also holds in these quotients, and that $\delta$ is also a derivation in $\sft\tocomm$ with respect to associative multiplication.

\subsection{Derivation property for $\delta$}
\label{ssec:derivation}

When we construct the $L_\infty$ structure on $\A\tocomm$ in Section~\ref{sec:l-infty}, a key ingredient in the proof of the $L_\infty$ relations is that the string coproduct $\delta$ behaves nicely with respect to the SFT bracket. Specifically, when $\Lambda$ has a single base point, $\delta$ is a derivation with respect to $\{\cdot,\cdot\}$:

\begin{proposition}[{\cite[Proposition 3.8]{SFT}}]
If $\Lambda$ is a Legendrian knot with a single base point, then for any $x,y\in\sft\tocyc$,
\label{prop:delta-knot}
\[
\delta\{x,y\} = \{\delta x,y\}-(-1)^{|x|}\{x,\delta y\}.
\]
\end{proposition}

In the general case where $\Lambda$ has multiple components and base points, $\delta$ is not a derivation with respect to $\{\cdot,\cdot\}$, even if we pass to the commutative quotient $\sft\tocomm$. However, $\delta$ does still satisfy a modified version of the derivation property on $\sft\tocomm$. This subsection is devoted to stating and proving this property. To state it, we introduce some more notation that will also play a key role in our construction of the $L_\infty$ structure.

We first note that while $\sft$ is generated as a $\Z$ module by elements of $\W$, which are words in the alphabet $\SS=\{q_1,\ldots,q_n,p_1,\ldots,p_n,t_1^{\pm},\ldots,t_s^{\pm}\}$, $\sft\tocomm$ is instead generated as a $\Z$-module by ``commutative words'', words in the alphabet $\SS$ up to arbitrary permutation. (Note that two words that are related by permutation are only equal up to the usual sign, e.g., $x_2x_1 = (-1)^{|x_1||x_2|}x_1x_2$.)
Write $\W\tocomm$ as the collection of commutative words in $\SS$, including the empty word.

\begin{definition}
Given a pointed Legendrian link $\Lambda$ with $s$ base points, define maps \label{def:vw}
\[
\vv,\ww \co \W\tocomm \to \Z^s
\]
as follows. Let $\mathbf{e}_1,\ldots,\mathbf{e}_s$ be the standard basis of $\Z^s$. For any $v \in \SS = \{q_1,\ldots,q_n,p_1,\ldots,p_n,t_1^{\pm},\ldots,t_s^{\pm}\}$, set
\[
\vv(v) = \ee_{e_+(v)}-\ee_{e_-(v)}
\qquad 
\ww(v) = \begin{cases} \ee_i & v = t_i \\
-\ee_i & v = t_i^{-1} \\
0 & v\in\{q_1,\ldots,q_n,p_1,\ldots,p_n\},
\end{cases}
\]
where $e_\pm(v)$ are defined in Section~\ref{ssec:capping}.
Now extend $\vv,\ww$ multiplicatively to $\W\tocomm$:
\[
\vv(xy) = \vv(x)+\vv(y) \qquad \ww(xy) = \ww(x)+\ww(y).
\]
\end{definition}

The maps $\vv,\ww$ can be interpreted geometrically as follows. Represent a commutative word $x\in\W\tocomm$ by a word $\tilde{x}\in\W$ (in any way) and let $\gamma$ be the based broken closed string associated to $\tilde{x}$ as in Section~\ref{ssec:bcs}. Then the $j$-th entry in $\vv(x)$ is the number of end points of $\gamma$ at $\bullet_j$ minus the number of beginning points of $\gamma$ at $\bullet_j$, and the $j$-th entry in $\ww(x)$ is the number of times $\gamma$ passes through $\star_j$, counted with sign.

\begin{definition}
Define a map $\beta \co \W\tocomm \times \W\tocomm \to \textstyle{\frac{1}{2}}\Z$
by
\[
\beta(x,y) = \left( -\textstyle{\frac{1}{2}}\vv(x)+\ww(x)\right) \cdot \vv(y),
\]
and define a map $\tb \co \sft\tocomm \otimes \sft\tocomm \to \sft\tocomm$ to be the bilinear map determined by \label{def:beta}
\[
\tb(x,y) = \beta(x,y) xy
\]
for $x,y\in\W\tocomm$.
\end{definition}

We now have the following derivation property for the string cobracket $\delta$.

\begin{proposition}
For any $x,y\in\sft\tocomm$,
\label{prop:delta}
\[
\delta \{x,y\} = \{\delta x,y\} - (-1)^{|x|}\{x,\delta y\} - (-1)^{|x|}\tb(x,y) - (-1)^{|x|(|y|+1)}\tb(y,x).
\]
\end{proposition}

Note that the correction to the derivation property for $\delta$, $- (-1)^{|x|}\tb(x,y) - (-1)^{|x|(|y|+1)}\tb(y,x)$, is a half-integer multiple of $xy$. In the special case where $\Lambda$ is a knot with a single base point, $\vv \equiv 0$ and thus this correction disappears, in agreement with Proposition~\ref{prop:delta-knot} above.

\begin{proof}[Proof of Proposition~\ref{prop:delta}]
For $x,y\in\W\tocomm$, define
\begin{align*}
f(x,y) &:= \delta\{x,y\} - \{\delta x,y\}+(-1)^{|x|}\{x,\delta y\} \\
g(x,y) &:= - (-1)^{|x|}\tb(x,y) - (-1)^{|x|(|y|+1)}\tb(y,x).
\end{align*}
We wish to show that $f(x,y) = g(x,y)$ for all $x,y\in\sft\tocomm$.
Note that $g(y,x) = (-1)^{|x||y|+|x|+|y|}g(x,y)$, while by the definition of $\tb$, $g(x,yz) = g(x,y)z+(-1)^{|x||y|}yg(x,z)$ for any $x,y,z\in\W\tocomm$. But $f$ satisfies the same properties:
$f(y,x) = (-1)^{|x||y|+|x|+|y|}f(x,y)$ from the symmetry property of $\{\cdot,\cdot\}$ in Proposition~\ref{prop:sft-bracket}, while $f(x,yz) = f(x,y)z+(-1)^{|x||y|}yf(x,z)$ by Propositions~\ref{prop:sft-bracket} and~\ref{prop:delta-properties}. To show $f(x,y)= g(x,y)$ for all $x,y\in\sft\tocomm$, it thus suffices to check that $f(x,y) = g(x,y)$ for all $x,y\in\SS$; the general result will follow for $x,y\in\W\tocomm$ by induction on wordlength, and then for $x,y\in\sft\tocomm$ by linearity.

We show that $f(x,y) = g(x,y)$ for $x,y\in\SS$ by breaking into cases. Note that for $x,y\in\SS$, $\delta\{x,y\} = 0$ and so $f(x,y) =  - \{\delta x,y\}+(-1)^{|x|}\{x,\delta y\}$.

\vspace{11pt}
\noindent
\textsc{Case 1:} $x\in\{q_i,p_i\}$ and $y\in\{q_j,p_j\}$, where $i\neq j$. 

In this case we have $\beta(x,y) = \beta(y,x) = -\frac{1}{2} \vv(x) \cdot \vv(y)$ and thus $g(x,y) = (-1)^{|x|} \vv(x)\cdot \vv(y) \,xy$.

Write $x^\pm$ for the endpoints of the Reeb chord $a_i$, where $x$ travels along the Reeb chord from $x^-$ to $x^+$: that is, $x^{\pm} = a_i^\pm$ if $x=p_i$ and $x^{\pm} = a_i^{\mp}$ if $x=q_i$. Similarly define $y^{\pm}$. Write the capping path $\gamma_x$ as $\gamma_x^- \cdot \gamma_x^+$, where $\gamma_x^-$ is the portion of $\gamma_x$ up to $x^-$ (i.e., $\gamma_i^-$ if $x=p_i$, $\gamma_i^+$ if $x=q_i$) and $\gamma_x^+$ is the portion after $x^+$ ($-\gamma_i^+$ if $x=p_i$, $-\gamma_i^-$ if $x=q_i$). Similarly write $\gamma_y = \gamma_y^- \cdot \gamma_y^+$. Finally, write $x^*,y^*$ for the ``signed duals'' to $x,y$: $x^*=-p_i$ if $x=q_i$ and $x^*=q_i$ if $x=p_i$, and similarly for $y$. Note that $\{x,x^*\} = -\{x^*,x\} = \{y,y^*\} = -\{y^*,y\}=1$.

Now $\{\delta x,y\} = 0$ unless at least one of $\gamma_x^\pm$ passes through an endpoint $a_j^\pm$ of $y$. By the definition of $\delta$, if $\gamma_x^-$ passes through $a_j^-$ (resp.\ $a_j^+$), then $\delta x$ contains the term $+p_jq_jx$ (resp.\ $-q_jp_jx$). Stated another way, if $\gamma_x^-$ passes through $y^+$ (resp.\ $y^-$), then $\delta x$ contains the term $-y^*yx$ (resp.\ $yy^*x$). The resulting contribution to $-\{\delta x,y\}$ is
$-\{-y^*yx,y\} = (-1)^{|x|}\{xyy^*,y\}=-(-1)^{|x|}xy$ (resp.\
$-\{yy^*x,y\} = -(-1)^{|x|}\{xyy^*,y\} = (-1)^{|x|}xy$). But in both cases this exactly equals the contribution to $(-1)^{|x|}\vv(x)\cdot\vv(y)\,xy$ arising from the fact that $\gamma_x^-$ shares an endpoint with $\gamma_y^\pm$.

In a similar way, we can show that in each instance where $\gamma_x$ shares an endpoint with $\gamma_y$ at a base point, we get equal contributions to $f(x,y)$ and $g(x,y)$:
\[
\begin{array}{|c||c|c|c|}
\hline
& \multicolumn{2}{c|}{f(x,y)} 
& g(x,y) \\ \cline{2-4}
& -\{\delta x,y\} & (-1)^{|x|}\{x,\delta y\} & (-1)^{|x|} \vv(x)\cdot \vv(y)\,xy \\ \hline\hline
\gamma_x^- \text{ contains } y^+
& -\{-y^*yx,y\} & & - (-1)^{|x|} xy \\ \hline
\gamma_x^- \text{ contains } y^-
& -\{yy^*x,y\} & & (-1)^{|x|} xy \\ \hline 
\gamma_x^+ \text{ contains } y^+
& -\{-(-1)^{|x|}xy^*y,y\} & & (-1)^{|x|} xy \\ \hline 
\gamma_x^+ \text{ contains } y^-
& -\{(-1)^{|x|}xyy^*,y\} & & - (-1)^{|x|} xy \\ \hline
\gamma_y^- \text{ contains } x^+
& & (-1)^{|x|}\{x,-x^*xy\} &  - (-1)^{|x|} xy \\ \hline
\gamma_y^- \text{ contains } x^-
& & (-1)^{|x|}\{x,xx^*y\} &  (-1)^{|x|} xy \\ \hline
\gamma_y^+ \text{ contains } x^+
& & (-1)^{|x|}\{x,-(-1)^{|y|}yx^*x\} &  (-1)^{|x|} xy \\ \hline
\gamma_y^+ \text{ contains } x^-
& & (-1)^{|x|}\{x,(-1)^{|y|}yxx^*\} &  - (-1)^{|x|} xy \\ \hline
\end{array}
\]
In each case the expression in the $f(x,y)$ columns equals the expression in the $g(x,y)$ column. This shows that $f(x,y)=g(x,y)$.

\vspace{11pt}
\noindent
\textsc{Case 2:} $x,y\in\{q_i,p_i\}$. 

This is similar to Case 1 but with additional contributions to $f(x,y)$ and $g(x,y)$. Specifically, if $|q_i|$ is even, then the capping path $\gamma_{q_i}$ does not have holomorphic corners and must be perturbed, leading to an additional term $-q_ip_iq_i$ in $\delta q_i$. This adds $2q_i^2,-2q_ip_i,-2p_iq_i,0$ to $f(q_i,q_i),f(q_i,p_i),f(p_i,q_i),f(p_i,p_i)$ respectively, in each case exactly matching the contribution to $g(x,y) = (-1)^{|x|}\vv(x)\cdot \vv(y)\,xy$ coming from the fact that $\gamma_x$ and $\gamma_y$ have the same endpoints, either in the same order if $x=y$ or in opposite order otherwise. (Note for $x=y=p_i$ that $p_i^2=0$.) Similarly, if $|q_i|$ is odd, then $\gamma_{p_i}$ does not have holomorphic corners, leading to an additional term $+p_iq_ip_i$ in $\delta p_i$. This adds $0,2q_ip_i,-2p_iq_i,2p_i^2$ to $f(q_i,q_i),f(q_i,p_i),f(p_i,q_i),f(p_i,p_i)$ respectively, again matching the corresponding contribution to $g(x,y)$ in each case.

\vspace{11pt}
\noindent
\textsc{Case 3:} $x\in\{t_i,t_i^{-1}\},y\in\{q_j,p_j\}$. By symmetry this also covers the case when $x$ and $y$ are swapped.

For $x = t_i^{\pm 1}$, we have
$f(x,y) = -\{\delta t_i^{\pm 1},y\}$. This is nonzero if either endpoint $y^\pm$ of $y$ lies on the capping path $\gamma_{t_i}$ between $t_{e_-(t_i)}$ and $t_i$. If $y^+$ (resp.\ $y^-$) lies on $\gamma_{t_i}$, then $\delta t_i$ contains a term $-y^*yt_i$ (resp.\ $yy^*t_i$), and this contributes $-t_iy$ (resp.\ $t_iy$) to $f(t_i,y)$ and $t_i^{-1}y$ (resp.\ $-t_i^{-1}y$) to $f(t_i^{-1},y)$. On the other hand, we compute that
\[
g(t_i^{\pm 1},y) = \left((\vv(t_i^{\pm 1})-\ww(t_i^{\pm 1}))\cdot \vv(y)\right)\,t_i^{\pm 1}y = \mp (\ee_{e_-(t_i)} \cdot \vv(y))\, yt_i^{\pm 1}
\]
and the presence of $y^\pm$ on $\gamma_{t_i}$ contributes $\pm 1$ to $\ee_{e_-(t)} \cdot \vv(y)$. We conclude that $f(x,y)=g(x,y)$.

\vspace{11pt}
\noindent
\textsc{Case 4:} $x\in\{t_i,t_i^{-1}\},y\in\{t_j,t_j^{-1}\}$.

In this case we have $f(x,y)=0$ and $g(x,y) = -(\beta(x,y)+\beta(y,x))xy$. Define signs $\pm_1,\pm_2 \in\{\pm 1\}$ by $x = t_i^{\pm_1}$, $y=t_j^{\pm_2}$. Then
\begin{align*}
\beta(x,y)+\beta(y,x) &= \left(\ww(t_i^{\pm_1})-\vv(t_i^{\pm_1})\right)\cdot \left(\ww(t_j^{\pm_2})-\vv(t_j^{\pm_2})\right) - \ww(t_i^{\pm_1})\cdot \ww(t_j^{\pm_2}) \\
&= \pm_1 \pm_2 \left( \ee_{e_-(t_i)}\cdot\ee_{e_-(t_j)} - \ee_i \cdot \ee_j \right).
\end{align*}
But both $\ee_i \cdot \ee_j$ and $\ee_{e_-(t_i)}\cdot\ee_{e_-(t_j)}$ are $0$ unless $i=j$, in which case they are both $1$. Thus $g(x,y)=0$ and we are done.
\end{proof}

To conclude this subsection, we collect some easy properties of $\tb$ that will be useful later. To set this up, we extend the definitions of $\vv,\ww,\beta$ slightly, as follows. If $x\in\sft\tocomm$ is a linear combination of commutative words $w_i$ such that $\vv(w_i) \in \Z^s$ is the same for all $i$, then we define $\vv(x)$ to be equal to this same element of $\Z^s$. Similarly we can define $\ww(x) \in \Z^s$ if $\ww(w_i)$ is the same for all $i$, and $\beta(x,y) \in \frac{1}{2}\Z$ is defined as in Definition~\ref{def:beta} whenever $\vv(x),\ww(x),\vv(y)$ are well-defined.

\begin{proposition}
\begin{enumerate}
\item
The Hamiltonian $h$ satisfies $\vv(h) = 0$. \label{it:beta4}
\item 
If $y\in\W\tocomm$ satisfies $\vv(y) = 0$, then
$\tb(x,y) = 0$ for any $x\in\sft\tocomm$. It follows that the Hamiltonian $h$ satisfies
$\tb(x,h) = 0$ for any $x\in\sft\tocomm$. \label{it:beta1}
\item
For any $x_1,x_2 \in \sft\tocomm$ for which $\vv(x_1),\vv(x_2)$ are defined, $\vv(\{x_1,x_2\}) = \vv(x_1x_2) = \vv(x_1)+\vv(x_2)$, and similarly for $\ww$. \label{it:beta5}
\item
If $\vv(x)$ is defined for $x\in\sft\tocomm$, then $\vv(\delta x) = \vv(x)$. \label{it:beta6}
\item
For any $x_1,x_2,x_3 \in \W\tocomm$, we have \label{it:beta3}
\begin{align*}
\beta(\{x_1,x_2\},x_3) &= \beta(x_1x_2,x_3) = \beta(x_1,x_3) + \beta(x_2,x_3) \\
\beta(x_1,\{x_2,x_3\}) &= \beta(x_1,x_2x_3) = \beta(x_1,x_2)+\beta(x_1,x_3).
\end{align*}
\item
For any $x_1,x_2 \in \A\tocomm$ and $y\in\sft\tocomm$, we have  \label{it:beta2}
\[
\{y,\tb(x_2,x_1)\}+(-1)^{|x_1|(|x_2|+1)}\{\tb(y,x_1),x_2\} = 
\tb(\{y,x_2\},x_1) + (-1)^{|x_2|(|y|+1)} \beta(x_2,x_1) x_2 \{y,x_1\}.
\]
\end{enumerate}
\label{prop:beta}
\end{proposition}

\begin{proof}
\eqref{it:beta4} holds because $h$ is a linear combination of cyclically composable words, each of which satisfies $\vv=0$; \eqref{it:beta1} is an immediate consequence.
\eqref{it:beta5} is clear by inspection of the definitions of $\vv,\ww$ in Definition~\ref{def:vw} and the SFT bracket in Definition~\ref{def:SFTbracket}, and \eqref{it:beta6} similarly follows from the definition of $\delta$ in Definition~\ref{def:delta}; \eqref{it:beta3} follows from \eqref{it:beta5}.

To establish \eqref{it:beta2}, we may assume that $x_1,x_2,y$ are all words in $\W\tocomm$. Then by applying Proposition~\ref{prop:sft-bracket} and using the fact that $\{x_1,x_2\}=0$ since $x_1,x_2\in\A\tocomm$, we compute that
\begin{align*}
\{y,\tb(x_2,x_1)\} &= \beta(x_2,x_1) \{y,x_2x_1\} = \beta(x_2,x_1)(\{y,x_2\}x_1+(-1)^{(|y|+1)|x_2|}x_2\{y,x_1\}) \\
\{\tb(y,x_1),x_2\} &= \beta(y,x_1)\{yx_1,x_2\} = \beta(y,x_1)(-1)^{|x_1|(|x_2|+1)}\{y,x_2\}x_1 \\
\{\tb(y,x_2),x_1\} &= \beta(y,x_2)\{yx_2,x_1\} = \beta(y,x_2)(-1)^{|x_2|(|x_1|+1)}\{y,x_1\}x_2;
\end{align*}
combining these and using \eqref{it:beta3} yields \eqref{it:beta2}.
\end{proof}

\subsection{The quantum master equation and differential}
\label{ssec:qme}

We conclude this section by discussing the relation between the Hamiltonian $h$ and the SFT bracket and string cobracket. This allows us to define a ``differential'' $d$ on $\sft\tocomm$ that extends the standard differential $\partial$ in the (commutative) Chekanov--Eliashberg DGA.

The key relation is a formula called the ``quantum master equation'' in \cite{SFT}. We omit the proof here since it is identical to the proof given there: at heart, a standard Floer-theoretic argument decomposing the boundary of a $1$-dimensional moduli space into products of two $0$-dimensional moduli spaces.

\begin{proposition}[{\cite[Proposition 3.13]{SFT}}]
If we view the Hamiltonian $h$ as an element of either $\sft\tocyc$ or $\sft\tocomm$, then we have:
 \label{prop:qme}
\[
\delta h = \frac{1}{2} \{h,h\}.
\]
\end{proposition}

\begin{remark}
In \cite[Proposition 3.13]{SFT}, the quantum master equation is presented as $\delta h + \frac{1}{2}\{h,h\}=0$. However, our choice of sign convention for $h$ in Section~\ref{ssec:h} differs from the one in \cite{SFT}, and consequently the sign in Proposition~\ref{prop:qme} differs as well. See \cite[Appendix~A]{SFT} for a full discussion of this sign difference.
\end{remark}

\begin{example}
Using our calculation of $h$ in Example~\ref{ex:h}, we can verify Proposition~\ref{prop:qme} for the example from Section~\ref{ssec:hopf-ex}. For instance, the bracket $\{t_1q_1t_3q_2p_3,t_2p_2t_3^{-1}p_1\}  =t_1t_2p_2 q_2p_3-t_1q_1p_1t_2p_3$
computed in Example~\ref{ex:bracket} contributes to $\frac{1}{2}\{h,h\}$ and corresponds to the terms 
$t_1t_2(p_2q_2p_3)$ and $t_1(-q_1p_1t_2)p_3$ from $t_1t_2(\delta p_3)$ and $t_1(\delta t_2)p_3$, which appear in $\delta h$. \label{ex:qme}
\end{example}

The quantum master equation in Proposition~\ref{prop:qme} resembles the Maurer--Cartan equation, and indeed the Hamiltonian $h$ plays the role of a Maurer--Cartan element in SFT. Specifically, we can deform the map $\delta$ using $h$ to construct a differential.

\begin{definition}
The \textit{SFT differential} is the map 
$d \co \sft\tocomm \to \sft\tocomm$ defined by: \label{def:sftdiff}
\[
d(x) = \{h,x\} - \delta x.
\]
\end{definition}

The map $d$ on $\sft\tocomm$ is induced from a map $d \co \sft\tocyc \to \sft\tocyc$ with the same definition. In the single-component case, $d$ is a differential:

\begin{proposition}[{\cite[Proposition 3.15]{SFT}}]
If $\Lambda$ is a Legendrian knot with a single base point, then the map $d \co \sft\tocyc \to \sft\tocyc$ satisfies $d^2=0$.
\end{proposition}

In general, however, $d$ does not satisfy $d^2=0$, even if we pass from $\sft\tocyc$ to the quotient $\sft\tocomm$. (We  abuse language in referring to $d$ as the ``SFT differential''.)  Instead, we have the following result, where $\tb$ was defined in Definition~\ref{def:beta}:

\begin{proposition}
Let $\Lambda$ be a pointed Legendrian link. Then the following equality holds for all $x\in\sft\tocomm$:
\label{prop:differential}
\[
d^2 x = \tb(h,x).
\]
\end{proposition}

\begin{proof}
By Propositions~\ref{prop:sft-bracket} and~\ref{prop:qme}, we have
\[
2 \{h,\{h,x\}\} = - \{x,\{h,h\}\} = 2 \{\delta h,x\}.
\]
It follows that
\begin{align*}
d^2 x &= \{h,\{h,x\}\}-\{h,\delta x\}-\delta\{h,x\}+\delta^2 x \\
&= -\delta\{h,x\} + \{\delta h,x\} - \{h,\delta x\} \\
&= \tb(h,x)+\tb(x,h) \\
&= \tb(h,x)
\end{align*}
where we have used Propositions~\ref{prop:delta-properties}, \ref{prop:delta}, and \ref{prop:beta} \eqref{it:beta1} in the second, third, and fourth lines, respectively.
\end{proof}

For future use, we note that $d$ satisfies essentially the same modified derivation property with respect to $\{\cdot,\cdot\}$ as does $\delta$:

\begin{proposition}
For any $x,y\in\sft\tocomm$,
\label{prop:d}
\[
d \{x,y\} = \{d x,y\} - (-1)^{|x|}\{x,d y\} + (-1)^{|x|}\tb(x,y) + (-1)^{|x|(|y|+1)}\tb(y,x).
\]
\end{proposition}

\begin{proof}
Immediate from Proposition~\ref{prop:delta} and the following consequence of the Jacobi identity:
\[
\{h,\{x,y\}\} = \{\{h,x\},y\}-(-1)^{|x|}\{x,\{h,y\}\}.
\]
\end{proof}

\begin{remark}
It is clear from the construction of the differential $d$ on $\sft\tocomm$ that it extends the differential $\partial$ on the commutative Chekanov--Eliashberg DGA $(\A\tocomm,\d)$. 
\label{rmk:d-partial}
To be precise,
$\A\tocomm$ is precisely the quotient $\sft\tocomm/\F^1\sft\tocomm$. Since the Hamiltonian $h$ lies in $\F^1\sft\tocomm$ by Definition/Proposition~\ref{defprop:h}, it follows from Proposition~\ref{prop:differential} that the image of $d^2$ in $\sft\tocomm$ lies in $\F^1\sft\tocomm$. Thus $d$ induces an honest differential on $\A\tocomm \cong \sft\tocomm/\F^1\sft\tocomm$. The only part of $dx = \{h,x\}+\delta x$ that does not raise filtration level is $\{h_1,x\}$, and we conclude from Proposition~\ref{prop:sft-lch} that $d$ induces the map $\partial$ on $\A\tocomm$.
\end{remark}

\section{The $L_\infty$ structure}
\label{sec:l-infty}

In this section, we construct the $L_\infty$ structure associated to a pointed Legendrian link and present some examples. The proof that the $L_\infty$ relations are satisfied, and thus that our construction actually constitutes an $L_\infty$ algebra, is deferred until Section~\ref{sec:l-infty-proof}.

Throughout this section and for the rest of the paper, we fix a field $\kk$ of characteristic $0$. Rather than using the $\Z$-algebra $\A\tocomm$, we will tensor with $\kk$ and use the $\kk$-algebra $\A\tocomm\otimes\kk$. However, to simplify notation, we will abuse notation and henceforth use $\A\tocomm$ to mean the $\kk$-algebra $\A\tocomm\otimes\kk$. We note that to define the $L_\infty$ structure on $\A\tocomm$, as we will do, $\kk$ could also be any field of characteristic $\neq 2$; however, the invariance result in Section~\ref{sec:invariance} relies on $\kk$ having characteristic $0$.

\subsection{$L_\infty$ algebras}
\label{ssec:l-infty}

We begin by reviewing the definition of an $L_\infty$ algebra, partly to state the sign conventions that we use in this paper. 

\begin{definition}
An \textit{$L_\infty$ algebra} (or \textit{strong homotopy Lie algebra}) is a $\Z$-graded vector space $A$ over the field $\kk$ equipped with operations $\ell_k :\thinspace A^{\otimes k} \to A$ of degree $k-2$ for each $k \geq 1$, satisfying:
\begin{itemize}
\item
graded symmetry: for any permutation $\sigma \in S_k$,
\[
\ell_k(x_{\sigma(1)},\ldots,x_{\sigma(k)}) = \chi(\sigma,x_1,\ldots,x_k) \,\ell_k(x_1,\ldots,x_k);
\]
\item
strong homotopy Jacobi identities: for any $k \geq 1$,
\[
\sum_{i=1}^k (-1)^{i(k-i)} \sum_{\sigma \in \Unshuff(i,k-i)} \chi(\sigma,x_1,\ldots,x_k)  \, \ell_{k-i+1}\left(
\ell_i(x_{\sigma(1)},\ldots,x_{\sigma(i)}),x_{\sigma(i+1)},\ldots,x_{\sigma(k)} \right) = 0.
\]
\end{itemize}
Here $\chi(\sigma,x_1,\ldots,x_k)$ is the alternating Kozsul sign, defined to be $-(-1)^{|x_i||x_{i+1}|}$ for $\sigma = (i~~i+1)$ and extended multiplicatively on $S_k$, and 
$\Unshuff(i,k-i)$ is the subset of $S_k$ consisting of permutations $\sigma$ with $\sigma(1)<\sigma(2)<\cdots<\sigma(i)$ and $\sigma(i+1)<\cdots<\sigma(k)$.
\label{def:Linfalg}
\end{definition}

For concreteness, the first three Jacobi identities are:
\begin{align*}
0 &= \ell_1(\ell_1(x_1)) \\
0 &= \ell_1(\ell_2(x_1,x_2)) - \ell_2(\ell_1(x_1),x_2) - (-1)^{|x_1|} \ell_2(x_1,\ell_1(x_2)) \\
0 &= \ell_2(\ell_2(x_1,x_2),x_3)) + (-1)^{|x_1|(|x_2|+|x_3|)} \ell_2(\ell_2(x_2,x_3),x_1)
+(-1)^{|x_3|(|x_1|+|x_2|)} \ell_2(\ell_2(x_3,x_1),x_2) \\
& \quad
+ \ell_3(\ell_1(x_1),x_2,x_3) + (-1)^{|x_1|} \ell_3(x_1,\ell_1(x_2),x_3)+
(-1)^{|x_1|+|x_2|}\ell_3(x_1,x_2,\ell_1(x_3)) \\
&\quad +\ell_1(\ell_3(x_1,x_2,x_3)).
\end{align*}

If $(A,\{\ell_k\})$ is an $L_\infty$ algebra, then $(A,\ell_1)$ is a complex, with $\ell_1$ being a differential of degree $-1$,  and the homology $H_*(A,\ell_1)$ is a Lie algebra with Lie bracket induced by $\ell_2$.

We say that an $L_\infty$ algebra $(A,\{\ell_k\})$ is \textit{strict} if $\ell_k = 0$ for $k \geq 3$. In this case, the usual Jacobi identity for $\ell_2$ holds on the nose, and $(A,\ell_1,\ell_2)$ is a DG Lie algebra.

\begin{definition}
A \textit{homotopy Poisson algebra} is an $L_\infty$ algebra $A$ that is also an associative algebra $A$, such that the $L_\infty$ operations $\ell_k$ satisfy the Leibniz rule for all $k \geq 1$:
\label{def:htpyPoisson}
\[
\ell_k(x_1,\ldots,x_kx_k') = \ell_k(x_1,\ldots,x_k)x_k' + (-1)^{(|x_1|+\cdots+|x_{k-1}|+k)|x_k|} x_k \ell_k(x_1,\ldots,x_k').
\]
\end{definition}

If $A$ is a homotopy Poisson algebra, then $(A,\ell_1)$ is a DG algebra, and the homology $H_*(A,\ell_1)$ is an associative algebra. Furthermore, $\ell_2$ induces a Poisson bracket on $H_*(A,\ell_1)$, and so $(H_*(A,\ell_1),\ell_2)$ is a Poisson algebra. If in addition $A$ is strict, then $A$ itself is also a Poisson algebra with Poisson bracket $\ell_2$.

\subsection{Definition of the $L_\infty$ structure for Legendrian links}
\label{ssec:linfty-def}

We now construct the $L_\infty$ algebra associated to a Legendrian link. Let $\Lambda$ be a pointed Legendrian link equipped with a Maslov potential, and let $\kk$ be a field of characteristic $0$. The vector space underlying the $L_\infty$ algebra ($A$, in the notation of Section~\ref{ssec:l-infty}) is $\alg\tocomm$, which we recall from the introduction to Section~\ref{sec:l-infty} is now the graded polynomial ring over $\kk$ generated by $q_1,\ldots,q_n,t_1^{\pm 1},\ldots,t_s^{\pm 1}$ (but not $p_1,\ldots,p_n$) in the notation of Section~\ref{ssec:algebras}. 

To define the $L_\infty$ operations on $\alg\tocomm$, we use the machinery built up in Section~\ref{sec:sft}. Let $h\in\sft\tocyc$ be the Hamiltonian for $\Lambda$. Recall from Definition/Proposition~\ref{defprop:h} that $h \in \F^1\sft\tocyc$: that is, every term in $h$ contains at least one $p$. We can further write
\[
h = h_1 + h_2 + h_3 + \cdots
\]
where $h_k$ consists of all terms in $h$ containing exactly $k$ $p$'s; in the language of SFT, $h_k$ counts holomorphic disks with exactly $k$ positive punctures.

\begin{definition}
For $k \geq 1$, define $\ell_k :\thinspace (\alg\tocomm)^{\otimes k} \to \alg\tocomm$ as follows: \label{def:main}
\begin{align*}
\ell_1(x_1) &= \{h_1,x_1\} \\
\ell_2(x_1,x_2) &= (-1)^{|x_1|} \{\{h_2,x_1\},x_2\} - \frac{1}{2} (-1)^{|x_1|}\{\delta x_1,x_2\} - \frac{1}{2} \{x_1,\delta x_2\} \\
&\quad - \frac{1}{2} \tb(x_1,x_2) + \frac{1}{2}(-1)^{|x_1||x_2|} \tb(x_2,x_1)
\end{align*}
and for $k \geq 3$,
\[
\ell_k(x_1,\ldots,x_k) = (-1)^{|x_{k-1}|+|x_{k-3}|+\cdots} \{ \cdots \{\{\{h_k,x_1\},x_2\},x_3\},\ldots,x_k\}.
\]
\end{definition}

\noindent
See Definition~\ref{def:beta} for the definition of $\tb$ appearing in $\ell_2$; for computational purposes, it may be helpful to note that when $x_1,x_2 \in \W\tocomm$, 
\[
 -\frac{1}{2} \tb(x_1,x_2)+\frac{1}{2}(-1)^{|x_1||x_2|} \tb(x_2,x_1) = \frac{1}{2} (\vv(x_1)\cdot \ww(x_2)-\ww(x_1)\cdot\vv(x_2)) x_1x_2.
\]
Also note that each $\ell_k$ does indeed map into $\alg\tocomm$ and not just into the larger $\sft\tocomm$: this is because the SFT bracket reduces the number of $p$'s by $1$, $\delta$ increases the number of $p$'s by $1$, and each $x_i$ contains no $p$'s.

The first operation $\ell_1$ is precisely the Chekanov--Eliashberg differential $\partial$ on $\alg\tocomm$; see Section~\ref{ssec:qme}. Thus the first $L_\infty$ relation simply states the well-known fact that $\partial^2=0$.

\begin{proposition}
The operations $\ell_k$ constructed in Definition~\ref{def:main} satisfy graded symmetry and the Leibniz rule.
\label{prop:skewsymmetry}
\end{proposition}

\begin{proof}
Note that if $x \in \sft\tocomm$ and $y,z \in \alg\tocomm$, then by properties of $\{\cdot,\cdot\}$ (Proposition~\ref{prop:sft-bracket}),
\begin{align*}
\{\{x,y\},z\} &= (-1)^{(|x|+|y|)(|z|+1)+1}\{z,\{x,y\}\} \\
&= \{x,\{y,z\}\} + (-1)^{(|x|+1)(|y|+|z|)}\{y,\{z,x\}\} \\
&= (-1)^{(|x|+1)(|y|+|z|)}\{y,\{z,x\}\} \\
&= (-1)^{(|x|+1)(|y|+1)}\{\{x,z\},y\},
\end{align*}
where the third equality follows from the fact that $\{y,z\}=0$ since neither $y$ nor $z$ contains any $p$'s. Thus we have
\[
(-1)^{|x_2|}\{\{h_2,x_2\},x_1\} = (-1)^{|x_1||x_2|+|x_1|+1} \{\{h_2,x_1\},x_2\}.
\]
The fact that $\ell_2(x_2,x_1) = (-1)^{|x_1||x_2|+1}\ell_2(x_1,x_2)$ is now easy to check from this.

The proof that $\ell_k$ is graded-symmetric for $k\geq 3$ is similar. For $j=1,\ldots,k-1$, transposing $x_j$ and $x_{j+1}$ in the definition of $\ell_k$ has the effect of replacing $\{\{y,x_j\},x_{j+1}\}$ by $(-1)^{|x_j|+|x_{j+1}|} \{\{y,x_{j+1}\},x_j\}$, where $y = \{\cdots \{h_k,x_1\},\ldots,x_{j-1}\}$. Now
$ \{\{y,x_{j+1}\},x_j\} = (-1)^{(|x_j|+1)(|x_{j+1}|+1)} \{\{y,x_j\},x_{j+1}\}$ and thus transposing $x_j$ and $x_{j+1}$ changes $\ell_k(x_1,\ldots,x_k)$ by multiplication by
$(-1)^{|x_j||x_{j+1}|+1}$, as desired.

It remains to show that for $k\geq 1$, $\ell_k$ satisfies the Leibniz rule as stated in Definition~\ref{def:htpyPoisson}. For $k \neq 2$, this is immediate from the fact that $\{\cdot,\cdot\}$ satisfies the Leibniz rule (Proposition~\ref{prop:sft-bracket}). For $k=2$ we additionally compute that
\[
\{x_1,\delta(x_2x_2')\} = \{x_1,(\delta x_2)x_2'\} + (-1)^{|x_2|} \{x_1,x_2(\delta x_2')\} = \{x_1,\delta x_2\}x_2' + (-1)^{|x_1||x_2|}x_2\{x_1,\delta x_2'\}
\]
by Propositions~\ref{prop:sft-bracket} and \ref{prop:delta-properties} and the fact that $\{x_1,x_2\} = \{x_1,x_2'\} = 0$, and that
\begin{align*}
\tb(x_1,x_2x_2') &= \tb(x_1,x_2)x_2' + (-1)^{|x_1||x_2|} x_2 \tb(x_1,x_2') \\
\tb(x_2x_2',x_1) &= (-1)^{|x_1||x_2'|}\tb(x_2,x_1)x_2' + x_2 \tb(x_2',x_1)
\end{align*}
from Proposition~\ref{prop:beta} \eqref{it:beta3}. Combining these, we conclude that $\ell_2$ satisfies the Leibniz rule.
\end{proof}

\begin{remark}
The need for $\ell_2$ to be graded-symmetric is what forces the presence of the $\frac{1}{2}$ fractions in Definition~\ref{def:main}. Alternatively, one could multiply every $\ell_k$ by $2$ to obtain an $L_\infty$ structure whose operations have integer coefficients, for instance if one wished to work over a field of characteristic $2$. 
However, $\ell_1$ would no longer coincide with the usual Chekanov--Eliashberg differential $\partial$ but would instead equal $2\partial$. 

As an additional reason not to work over a field of characteristic $2$, we note that the quasi-isomorphism type of $(\A\tocomm,\partial)$ is only guaranteed to be invariant under Legendrian isotopy if we work over a field of characteristic $0$: see \cite[section~3.4]{ENS}.
\end{remark}

We can now state our main result.

\begin{proposition}
For any pointed Legendrian link, $(\A\tocomm,\{\ell_k\}_{k=1}^\infty)$ is a homotopy Poisson algebra.
\label{prop:htpyPoisson}
\end{proposition}

To prove Proposition~\ref{prop:htpyPoisson}, it remains to check that the $\ell_k$ operations satisfy the $L_\infty$ relations. We defer this to Section~\ref{sec:l-infty-proof}.

An immediate corollary of Proposition~\ref{prop:htpyPoisson} is that the commutative Legendrian contact homology of a Legendrian link inherits a Poisson bracket from $\ell_2$.

\begin{corollary}
For any pointed Legendrian link, $(H_*(\A\tocomm,\partial),\ell_2)$ is a Poisson algebra.
\label{cor:Poisson}
\end{corollary}

\subsection{A formula for the string portion of $\ell_2$}
\label{ssec:l2str}

The definition of the $\ell_k$ operations in Definition~\ref{def:main} is relatively straightforward except for $\ell_2$. Here we take a brief digression and present an equivalent formula for $\ell_2$ that is convenient for computations. Write
\[
\ell_2(x_1,x_2) = (-1)^{|x_1|}\{\{h_2,x_1\},x_2\} + \ell_2^\str(x_1,x_2)
\]
where we define
\[
\ell_2^\str(x_1,x_2) := \frac{1}{2}\left((-1)^{|x_1|+1}\{\delta x_1,x_2\} - \{x_1,\delta x_2\} -\tb(x_1,x_2) +(-1)^{|x_1||x_2|} \tb(x_2,x_1)\right).
\]
Note that $\ell_2^\str(x_1,x_2)$ is the part of $\ell_2(x_1,x_2)$ that does not involve any disks in the Hamiltonian but does involve the string coproduct $\delta$; we call $\ell_2^\str$ the \textit{string portion} of $\ell_2$. For any $x_1,x_2$, $\ell_2^\str(x_1,x_2)$ is some half-integer multiple of $x_1x_2$. Here we present a simple formula for this half-integer.

 To that end, we set some notation. Given two distinct points $z_1,z_2$ on a pointed Lagrangian link $\Lambda$, a \textit{segment} between $z_1$ and $z_2$ is an embedded path in $\Lambda$ whose endpoints are $z_1$ and $z_2$; there are exactly $2$ or $0$ segments between $z_1$ and $z_2$ depending on whether the points lie on the same component of $\Lambda$ or not. A segment between $z_1$ and $z_2$ is \textit{unbroken} if its interior does not contain any base points of $\Lambda$.

Write $\SS_0 = \{q_1,\ldots,q_n,t_1,\ldots,t_s\}$ for the collection of $q$-variables associated to the Reeb chords of $\Lambda$, along with the base points of $\Lambda$, and note that the elements of $\SS_0$ along with the inverses $t_1^{-1},\ldots,t_s^{-1}$ generate $\A\tocomm$. We associate \textit{endpoints} to each $x\in\SS_0$ as follows: for $x=q_j$, the endpoints of $x$ are $a_j^+$ and $a_j^-$, the top and bottom endpoints of the Reeb chord $a_j$; for $x=t_i$, we say that the point $\bullet_i$ itself is the unique endpoint of $t_i$.

For distinct $x_1,x_2 \in \SS_0$, write $\mathcal{P}(x_1,x_2)$ for the collection of unbroken segments between an endpoint of $x_1$ and an endpoint of $x_2$; there are at most $4$ elements in $\mathcal{P}(x_1,x_2)$. Given a path $\gamma\in\mathcal{P}(x_1,x_2)$, orient $\gamma$ following the orientation of $\Lambda$, and define $\sigma(\gamma)$ to be $+1$ if $\gamma$ is oriented from $x_1$ to $x_2$ and $-1$ if it is oriented from $x_2$ to $x_1$. Now define $n(\gamma)\in\{\pm 1\}$ as follows:
\[
n(\gamma) = \begin{cases}
-\sigma_1\sigma_2\sigma(\gamma) & \text{if } (x_1,x_2) = (q_{j_1},q_{j_2}) \text{ and the endpoints of } \gamma \text{ are } a_{j_1}^{\sigma_1},a_{j_2}^{\sigma_2} \\
-\sigma & \text{if } (x_1,x_2) = (q_j,t_i) \text{ and the endpoints of } \gamma \text{ are } a_j^{\sigma},\bullet_i \\
\sigma & \text{if } (x_1,x_2) = (t_i,q_j) \text{ and the endpoints of } \gamma \text{ are } \bullet_i,a_j^{\sigma} \\
\sigma(\gamma) & \text{if } (x_1,x_2)=(t_{i_1},t_{i_2}).
\end{cases}
\]

We now have the following result, which completely determines $\ell_2^\str$ since it is easy to check that $\ell_2^\str(t_i^{-1},\cdot) = -t_i^{-2} \ell_2^\str(t_i,\cdot)$.

\begin{proposition}
Let $x_1,x_2\in\SS_0$. If $x_1=x_2$, then $\ell_2^\str(x_1,x_2) = 0$; if $x_1\neq x_2$, then
\label{prop:l2str}
\[
\ell_2^\str(x_1,x_2) = \frac{1}{2} \left(\sum_{\gamma\in\mathcal{P}(x_1,x_2)} n(\gamma)\right) x_1x_2.
\]
\end{proposition}

\begin{proof}
From the definition of $\ell_2^\str$, we have $\ell_2^\str(x_2,x_1) = (-1)^{|x_1||x_2|+1}\ell_2^\str(x_1,x_2)$. Thus if $x_1=x_2$, then $\ell_2^\str(x_1,x_1) = 0$: either $|x_1|$ is even, in which case $\ell_2^\str(x_1,x_1)=0$ from the graded symmetry of $\ell_2^\str$, or $|x_1|$ is odd, in which case $\ell_2^\str(x_1,x_1)$ is a multiple of $x_1^2$, which is $0$ in $\A\tocomm$.

Now suppose $x_1\neq x_2$. By the symmetry of $\ell_2^\str$, it suffices to consider $3$ cases: $(x_1,x_2) = (q_{j_1},q_{j_2})$; $(x_1,x_2) = (q_j,t_i)$; $(x_1,x_2) = (t_{i_1},t_{i_2})$. 

First suppose $(x_1,x_2) = (q_{j_1},q_{j_2})$. Then
\begin{equation}
2\ell_2^\str(q_{j_1},q_{j_2}) = (-1)^{|q_{j_1}|+1}\{\delta q_{j_1},q_{j_2}\} - \{q_{j_1},\delta q_{j_2}\}.
\label{eq:l2str}
\end{equation}
Contributions to the right hand side of \eqref{eq:l2str} correspond exactly to instances in which the capping path for one of $q_{j_1},q_{j_2}$ passes through an endpoint of the other Reeb chord. By the construction of our capping paths, these instances are in precise correspondence to unbroken segments between $q_{j_1}^{\pm}$ and $q_{j_2}^{\pm}$. 

It now suffices to check that for such a segment $\gamma$, the contribution to the right hand side of \eqref{eq:l2str} is exactly $n(\gamma) x_1x_2$. 
This is a straightforward case-by-case check, which we will illustrate in two cases. If $\gamma$ is a path from $a_{j_1}^+$ to $a_{j_2}^+$, then the front half $\gamma_{j_2}^+$ of the capping path $\gamma_{q_{j_2}} = \gamma_{j_2}^+ \cdot (-\gamma_{j_2}^-)$ passes through $a_{j_1}^+$, yielding a term $-q_{j_1}p_{j_1}q_{j_2}$ in $\delta q_{j_2}$, and the contribution to \eqref{eq:l2str} is $-\{q_{j_1},-q_{j_1}p_{j_1}q_{j_2}\} = -q_{j_1}q_{j_2} = n(\gamma)q_{j_1}q_{j_2}$. If $\gamma$ is a path from $a_{j_1}^+$ to $a_{j_2}^-$, then the back half $-\gamma_{j_2}^-$ of $\gamma_{q_{j_2}}$ passes through $a_{j_1}^+$, yielding a term $(-1)^{|q_{j_2}|}q_{j_2}q_{j_1}p_{j_1}=q_{j_1}p_{j_1}q_{j_2}$ in $\delta q_{j_2}$, and the contribution to \eqref{eq:l2str} is $+q_{j_1}q_{j_2} = n(\gamma)q_{j_1}q_{j_2}$. We leave the other $6$ cases for the possible endpoints of $\gamma$ to the reader.

Next suppose $(x_1,x_2) = (q_j,t_i)$, in which case
\[
2\ell_2^\str(q_j,t_i) = -\{q_j,\delta t_i\} + (\vv(q_j)\cdot \ww(t_i)) q_jt_i.
\]
Now unbroken segments between $a_j^\pm$ and $\bullet_i$ come in four types. A segment $\gamma$ from $a_j^+$ to $\bullet_i$ corresponds to $\gamma_{t_i}$ passing through $a_j^+$; this gives a term $-q_jp_jt_i$ in $\delta t_i$ and thus contributes $-q_jt_i = n(\gamma)q_jt_i$ to $2\ell_2^\str(q_j,t_i)$. A segment $\gamma$ from $\bullet_i$ to $a_j^+$, on the other hand, does not contribute to $-\{q_j,\delta t_i\}$ but does contribute $-1$ to $\vv(q_j) \cdot \ww(t_i)$, resulting in a contribution of $-q_jt_i = n(\gamma)q_jt_i$. We leave the other $2$ cases to the reader.

Finally, suppose $(x_1,x_2) = (t_{i_1},t_{i_2})$, in which case
\[
2\ell_2^\str(t_{i_1},t_{i_2}) = (\vv(t_{i_1})\cdot \ww(t_{i_2})-\ww(t_{i_1})\cdot\vv(t_{i_2}))t_{i_1}t_{i_2}.
\]
An unbroken segment $\gamma$ from $\bullet_{i_1}$ to $\bullet_{i_2}$ contributes $-1$ to $\ww(t_{i_1})\cdot\vv(t_{i_2})$ and thus $+t_{i_1}t_{i_2} = n(\gamma)t_{i_1}t_{i_2}$ to $2\ell_2^\str(t_{i_1},t_{i_2})$, while a segment $\gamma$ from $\bullet_{i_2}$ to $\bullet_{i_1}$ contributes $-1$ to $\vv(t_{i_1})\cdot\ww(t_{i_2})$ and thus $-t_{i_1}t_{i_2} = n(\gamma)t_{i_1}t_{i_2}$ to $2\ell_2^\str(t_{i_1},t_{i_2})$.
\end{proof}

\subsection{$(-1)$-closures of admissible positive braids}
\label{ssec:closures}

To illustrate the construction of the $L_\infty$ algebra from Section~\ref{ssec:linfty-def}, here we consider a family of Legendrian links, $(-1)$-closures of admissible positive braids. These lie in a neighborhood of a standard Legendrian unknot in $\R^3$ and are braided arond this unknot. This family includes ``rainbow closures'' of positive braids.

We first review some notation from \cite[\S 2.2]{CN}. A positive braid is \textit{admissible} if drawing the braid from left to right results in a diagram in $[0,1] \times \R$ that represents the $xy$ projection of a Legendrian link in $(J^1(S^1) = S^1\times\R\times\R,~\ker(dz-y\,dx))$, where we view $S^1 = [0,1]/(0\sim 1)$. Not all positive braids are admissible; for our purposes, we note that by \cite[Proposition~2.7]{CN}, any positive braid containing a positive half-twist $\Delta$ is admissible.

It is shown in \cite[Proposition~2.6]{CN} that given any admissible positive braid $\beta$, there is a well-defined Legendrian link $\Lambda(\beta)$ in $\R^3$ which is topologically isotopic to the closure of $\beta\Delta^{-2}$, where $\Delta$ represents a half-twist in the braid group. This link $\Lambda(\beta)$ is called the \textit{$(-1)$-closure} of $\beta$ and is illustrated in Figure~\ref{fig:closure}. In particular, the $(-1)$-closure of a braid of the form $\beta\Delta^2$, where $\beta$ is any positive braid, is a Legendrian link often called the rainbow closure of $\beta$.

Suppose that $\beta$ is an admissible positive braid with $N$ strands, and that the length of $\beta$ as a braid word is $k$. Then the $(-1)$-closure $\Lambda(\beta)$ as shown in Figure~\ref{fig:closure} has $k+N^2$ Reeb chords: the $k$ crossings in $\beta$, which we label $a_1,\ldots,a_n$, and the $N^2$ crossings where the link loops around itself, which we label $\tilde{a}_{ij}$ for $1\leq i,j\leq N$. We note that if we view $\Lambda(\beta)$ as a satellite of the standard Legendrian unknot, then the $\tilde{a}_{ij}$ crossings of $\Lambda(\beta)$ correspond to the single Reeb chord of the standard unknot.

If we place one base point on each strand of $\beta$ as shown in Figure~\ref{fig:closure}, then we can write a generating set for the commutative algebra $\A\tocomm$ as
\[
\SS = \{q_1,\ldots,q_k,\tilde{q}_{11},\ldots,\tilde{q}_{NN},p_1,\ldots,p_k,\tilde{p}_{11},\ldots,\tilde{p}_{NN},t_1^{\pm 1},\ldots,t_N^{\pm 1}\}.
\]
If we give $\Lambda(\beta)$ the Maslov potential determined by the given placement of base points, then it is straightforward to check that we have the following gradings:
\begin{align*}
|q_i| &= 0 & |p_i| &= -1 \\
|\tilde{q}_{ij}| &= 1 & |\tilde{p}_{ij}| &= -2 \\
|t_i^{\pm 1}| &= 0.
\end{align*}

\begin{figure}
\labellist
\small\hair 2pt
\pinlabel {\color{blue} $\beta$} at 52 36
\endlabellist
\centering
\includegraphics[width=\textwidth]{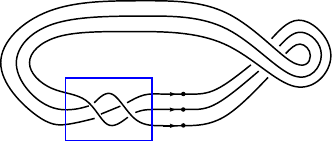}
\caption{The $(-1)$-closure $\Lambda(\beta)$ of an admissible positive braid $\beta$, in the $xy$ projection.
}
\label{fig:closure}
\end{figure}

\begin{proposition}
For $\Lambda(\beta)$, the homotopy Poisson algebra $(\A\tocomm,\{\ell_k\})$ is strict: $\ell_k = 0$ for $k \geq 3$.
\label{prop:closure-strict}
\end{proposition}

\begin{proof}
We claim that any immersed disk in $\Delta(\Lambda(\beta))$ has at most two positive corners, whence $h_k = 0$ for $k \geq 3$. Note that $\Lambda(\beta)$ lies in a tubular neighborhood of a Legendrian unknot $\Lambda_0$ whose $xy$ projection has a single crossing, and $\Pi_{xy}(\Lambda(\beta))$ lies in a tubular neighborhood of $\Pi_{xy}(\Lambda_0)$. Any disk $\Delta\in\Delta(\Lambda(\beta))$ either maps into this tubular neighborhood, in which case we call $\Delta$ ``thin'', or it does not, in which case we call it ``thick''. If $\Delta$ is thin, then by inspection, $\Delta$ has exactly two positive corners: these are the two points where $\partial\Delta$ changes direction along $\Pi_{xy}(\Lambda_0)$. If $\Delta$ is thick, then in fact $\Delta$ can only have one positive corner. Indeed, in the limit where we shrink the tubular neighborhood so that $\Pi_{xy}(\Lambda(\beta))$ approaches a multiple cover of $\Pi_{xy}(\Lambda_0)$, $\Delta$ approaches an (immersed) disk in $\Delta(\Lambda_0)$. But there are exactly two disks in $\Delta(\Lambda_0)$ and both have exactly one positive corner; thus $\Delta$ must have one positive corner as well.
\end{proof}

It follows from Proposition~\ref{prop:closure-strict} that for the Legendrian link $\Lambda(\beta)$, $\A\tocomm$ is a DG Poisson algebra with Poisson bracket $\ell_2$. In particular, the degree-$0$ subalgebra of $\A\tocomm$ generated by $q_1,\ldots,q_k,t_1^{\pm 1},\ldots,t_N^{\pm 1}$ is a Poisson algebra. We thus have the following corollary of Proposition~\ref{prop:closure-strict}.

\begin{corollary}
Given any admissible positive braid $\beta$ with $k$ crossings and $N$ strands, the $L_\infty$ structure for the Legendrian link $\Lambda(\beta)$ induces a Poisson bracket on \label{cor:closure}
\[
\kk[q_1,\ldots,q_k,t_1^{\pm 1},\ldots,t_N^{\pm 1}],
\]
the coordinate ring of $\kk^k \times (\kk^*)^N$.
\end{corollary}

As an example, we compute the Poisson bracket from Corollary~\ref{cor:closure} for the $2$-strand braid $\beta=\sigma_1^{2n+2}$, $n \geq 1$. The $(-1)$-closure of this braid is a $2$-component Legendrian link which is topologically a $(2,2n)$ torus link. What we will find is that this Poisson bracket is essentially the same as the Flaschka--Newell Poisson bracket \cite{FN} on $\kk[s_1,\ldots,s_{2n+2},\lambda^{\pm 1}]$, which is defined by
\begin{align*}
\{s_i,s_j\}_{\text{FN}} &= \delta_{i,j-1}-\frac{\delta_{i,1}\delta_{j,2n+2}}{\lambda^2}+(-1)^{i-j+1}s_is_j, \qquad i<j \\
\{s_i,\lambda\}_{\text{FN}} &= (-1)^i s_i \lambda.
\end{align*}

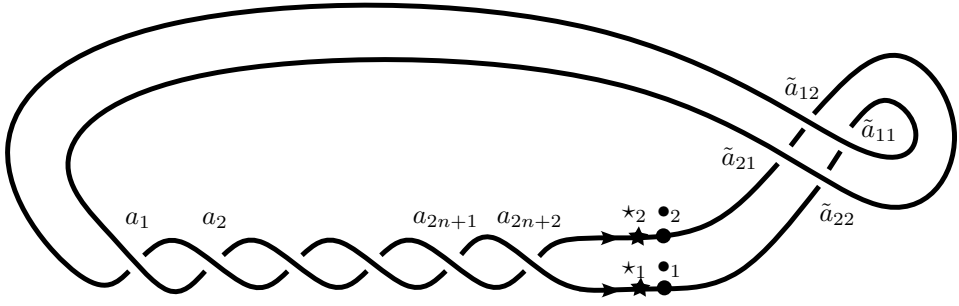
\begin{figure}
\labellist
\small\hair 2pt
\pinlabel $a_1$ at 19 11
\pinlabel $a_2$ at 30 11
\pinlabel $a_{2n+1}$ at 63 11
\pinlabel $a_{2n+2}$ at 75 11
\pinlabel $\star_1$ at 90 4
\pinlabel $\bullet_1$ at 95 4
\pinlabel $\star_2$ at 90 12
\pinlabel $\bullet_2$ at 95 12
\pinlabel $\tilde{a}_{11}$ at 125 24
\pinlabel $\tilde{a}_{12}$ at 114 30
\pinlabel $\tilde{a}_{22}$ at 105 20
\pinlabel $\tilde{a}_{21}$ at 119 12
\endlabellist
\centering
\includegraphics[width=\textwidth]{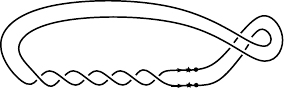}
\caption{The $(2,2n)$ torus link $\Lambda(\sigma_1^{2n+2})$, in the $xy$ projection.
}
\label{fig:torus}
\end{figure}

\begin{proposition}
The Poisson bracket on
\[
\kk[q_1,\ldots,q_{2n+2},t_1^{\pm 1},t_2^{\pm 1}]
\]
induced by the $L_\infty$ structure on the $(2,2n)$ torus link $\Lambda(\sigma_1^{2n+2})$ is: \label{prop:FN}
\begin{align*}
\{q_i,q_j\} &= (-1)^{i+j+1}q_iq_j && j \geq i+1 \\
\{q_i,q_{i+1}\} &= q_iq_{i+1}+1 && 1 \leq i \leq 2n+1 \\
\{q_1,q_{2n+2}\} &= q_1q_{2n+2}-t_1t_2^{-1} & \\
\{q_i,t_1\} &= (-1)^{i+1} q_i t_1 && 1\leq i\leq 2n+2 \\
\{q_i,t_2\} &= (-1)^{i} q_i t_2 && 1\leq i\leq 2n+2 \\
\{t_1,t_2\} &= 0.
\end{align*}
It follows that the map $q_i \mapsto s_i$, $t_1 \mapsto \lambda^{-1}$, $t_2 \mapsto \lambda$ sends our $\{\cdot,\cdot\}$ to the Flaschka--Newell bracket $\{\cdot,\cdot\}_{\text{FN}}$.
\end{proposition}

\begin{proof}
We need to compute $\ell_2$ on $q_1,\ldots,q_{2n+2},t_1,t_2$. To do this, we first compute the relevant portion of the Hamiltonian $h_2$, 
namely the terms that involve only $q_j,p_j,t_i$ and not $\tilde{q}_{ij},\tilde{p}_{ij}$. By inspection of Figure~\ref{fig:torus}, we have:
\[
h_2 = -p_1p_2-p_2p_3-\cdots-p_{2n+1}p_{2n+2}-t_2^{-1}p_{2n+2}t_1p_1+\cdots
\]
where the terms correspond to embedded bigons with corners at $a_j$ and $a_{j+1}$ for $j=1,\ldots,2n+1$, except for the final term, which corresponds to an immersed bigon with corners at $a_1$ and $a_{2n+2}$ and wraps around the loop region. The result now follows by using this formula for $h_2$ along with Proposition~\ref{prop:l2str} to compute $\ell_2(x,y) = \{\{h_2,x\},y\}+\ell_2^\str(x,y)$ for $x,y\in\{q_1,\ldots,q_{2n+2},t_1,t_2\}$.
\end{proof}

\subsection{The $L_\infty$ structure for Legendrian knots, and a variant}
\label{ssec:knot}

Here we consider the case where $\Lambda$ is a single-component Legendrian knot with a single base point, rather than an arbitrary pointed Legendrian link. In this case the definition of the $L_\infty$ structure from Section~\ref{ssec:linfty-def} simplifies considerably.

\begin{definition}
Let $\Lambda$ be a Legendrian knot with a single base point. For $k \geq 1$, define $\ell_k :\thinspace (\alg\tocomm)^{\otimes k} \to \alg\tocomm$ as follows: \label{def:knot}
\begin{align*}
\ell_1(x_1) &= \{h_1,x_1\} \\
\ell_2(x_1,x_2) &= (-1)^{|x_1|} \{\{h_2,x_1\},x_2\} - (-1)^{|x_1|}\{\delta x_1,x_2\}
\end{align*}
and for $k \geq 3$,
\[
\ell_k(x_1,\ldots,x_k) = (-1)^{|x_{k-1}|+|x_{k-3}|+\cdots} \{ \cdots \{\{\{h_k,x_1\},x_2\},x_3\},\ldots,x_k\}.
\]
Equivalently, if we define $d :\thinspace \sft\tocomm \to \sft\tocomm$ by $d = \{h,\cdot\}-\delta$ as in Definition~\ref{def:sftdiff}, and write $\Pi :\thinspace \sft\tocomm \to \alg\tocomm$ for the projection map annihilating all words containing $p$'s, then we can define $\ell_k$ for all $k \geq 1$ by:
\[
\ell_k(x_1,\ldots,x_k) = (-1)^{|x_{k-1}|+|x_{k-3}|+\cdots} \Pi \{ \cdots \{\{dx_1,x_2\},x_3\},\ldots,x_k\}.
\]
\end{definition}

For Legendrian knots, this definition of the $L_\infty$ structure agrees with the more general Definition~\ref{def:main}; the simplification of $\ell_2$ in this case is due to the facts that the $\tilde{\beta}$ terms disappear when there is a single base point, and the string coproduct $\delta$ is now an honest derivation with respect to the SFT bracket $\{\cdot,\cdot\}$.

The unified definition of $\ell_k$ involving $d$ in Definition~\ref{def:knot} implies the following result.

\begin{proposition}
Let $\Lambda$ be a Legendrian knot with a single base point. The $L_\infty$ structure on $\alg\tocomm$ determines the differential $d :\thinspace \alg\tocomm \to \alg\tocomm$. \label{prop:commsft}
\end{proposition}

\begin{proof}
Write the generators of $\alg\tocomm$ as $q_1,\ldots,q_n,t_1^\pm,\ldots,t_s^\pm$ as usual. Since $d$ satisfies the Leibniz rule, and $d(t_i)=0$ for all $i$, it suffices to show that the $L_\infty$ structure determines $d(q_j)$ for all $j$. We can write $d(q_j)$ as a polynomial (or power series) in $p_1,\ldots,p_n$ with coefficients in $q_1,\ldots,q_n,t_1^\pm,\ldots,t_s^\pm$:
\[
d(q_j) = \sum_{k\geq 0} \sum_{j_1\leq\cdots\leq j_k} f_{j_1,\ldots,j_k} p_{j_1} \cdots p_{j_k}
\]
with $f_{j_1,\ldots,j_k} \in \alg\tocomm$. Then for any $j_1,\ldots,j_k$, $f_{j_1,\ldots,j_k}$ is determined by $\{\cdots\{dq_j,q_{j_1}\},\ldots,q_{j_k}\}$ and thus by $\ell_{k+1}(q_j,q_{j_1},\ldots,q_{j_k})$.
\end{proof}

In \cite{SFT}, the complex $(\alg\tocomm,d)$ is called the \textit{commutative complex} associated to the LSFT algebra of the Legendrian knot $\Lambda$; this is the commutative version of the rational SFT invariant that is the subject of that paper. Proposition~\ref{prop:commsft} states that the $L_\infty$ structure developed in the present paper encodes the commutative rational SFT invariant from \cite{SFT}.

\begin{remark}
It is shown in \cite{SFT} that the commutative complex, up to filtered chain homotopy equivalence, is an invariant of the Legendrian knot under Legendrian isotopy. In Section~\ref{sec:invariance}, we will prove a restricted version of an invariance result for the $L_\infty$ structure, and we conjecture that the $L_\infty$ structure is in fact invariant up to $L_\infty$ equivalence. However, it is not clear that this invariance conjecture would directly imply invariance for the commutative complex: the procedure described above to pass from the $L_\infty$ structure to the commutative complex involves information about the $L_\infty$ structure besides its equivalence class, namely the distinguished generators $q_1,\ldots,q_n$ of $\alg\tocomm$.
\end{remark}

We next note that when $\Lambda$ is a Legendrian knot, we can lift the above $L_\infty$ structure from $\alg\tocomm$ to $\alg\tocyc$: the definition of $\ell_k :\thinspace (\alg\tocyc)^{\otimes k} \to \alg\tocyc$ is exactly the same as in Definition~\ref{def:knot} above.
Note that $\A\tocyc$ is not an associative algebra, and so the $L_\infty$ structure in this case will not be a homotopy Poisson structure.

\begin{proposition}
For any Legendrian knot with a single base point, $(\A\tocyc,\{\ell_k\}_{k=1}^\infty)$ is an $L_\infty$ algebra.
\label{prop:knot}
\end{proposition}

Proposition~\ref{prop:knot} will be proved in Section~\ref{sec:l-infty-proof}. 

The usual Chekanov--Eliashberg differential $\partial$ on the noncommutative algebra $\A$ descends to a differential on $\A\tocyc$. Proposition~\ref{prop:knot} immediately implies that the homology of $\A\tocyc$ with this differential inherits a Lie bracket from $\ell_2$.

\begin{corollary}
For any Legendrian knot with a single base point, $(H_*(\A\tocyc,\partial),\ell_2)$ is a Lie algebra.
\end{corollary}

\begin{remark}
For a Legendrian knot $\Lambda$, the homology $H_*(\A\tocyc,\partial)$ is the cylindrical contact homology of the contact manifold $\R^3(\Lambda)$ obtained from $\R^3$ by Legendrian surgery on $\Lambda$ \cite{BEE}. In this context, the $L_\infty$ operations appear to count rational holomorphic curves in the symplectization of $\R^3(\Lambda)$ with multiple positive punctures and a single negative puncture. In particular, the Lie bracket induced by $\ell_2$ counts curves with two positive ends and one negative end.

There is a version of this story for Legendrian links as well. In this case, one can consider $\A^{\operatorname{cyc},\operatorname{comp}}$, the module of composable cyclic words; see Remark~\ref{rmk:composable}. Then the homology $(\A^{\operatorname{cyc},\operatorname{comp}},d)$ is the cylindrical contact homology of the contact manifold $\R^3(\Lambda)$ obtained by Legendrian surgery on each component of $\Lambda$ \cite{BEE}. It appears to be the case that $(\A^{\operatorname{cyc},\operatorname{comp}},\{\ell_k\}_{k=1}^\infty)$ is again an $L_\infty$ algebra, and the $L_\infty$ operations again count rational curves with multiple positive and one negative puncture. 

Whether $\Lambda$ is a knot or link, let $X$ denote the Weinstein domain obtained by attaching a $2$-handle to $D^4$ along $\Lambda \subset \partial D^4$. Then the aforementioned cylindrical contact homology, associated to the contact $3$-manifold $\partial X$, can be interpreted as the positive $S^1$-equivariant symplectic homology of $X$. It is an interesting question to relate the invariants considered in this paper to known algebraic structures on symplectic invariants associated to Weinstein domains. 
However, we do not pursue this direction further in this paper.
\end{remark}

\section{Proof of the $L_\infty$ relations}
\label{sec:l-infty-proof}

This section is primarily devoted to proving Proposition~\ref{prop:htpyPoisson} by verifying that the $L_\infty$ operations that we defined for Legendrian links in Section~\ref{ssec:linfty-def} do indeed satisfy the $L_\infty$ relations. We begin by laying the groundwork for the proof of Proposition~\ref{prop:htpyPoisson}, and then present the proof itself. We end the section with the proof of Proposition~\ref{prop:knot}, which states that the variant $L_\infty$ structure for Legendrian knots as constructed in Section~\ref{ssec:knot} also satisfies the $L_\infty$ relations.

For $k \geq 1$, we extend the maps $\ell_k$ to maps $m_k :\thinspace (\A\tocomm)^{\otimes k} \to \sft\tocomm$ defined as follows. Recall that $d :\thinspace \sft\tocomm \to \sft\tocomm$ is defined by $d(x) = \{h,x\} + \delta x$. For any $k\geq 1$ with $k \neq 2$, define
\[
m_k(x_1,\ldots,x_k) = (-1)^{|x_{k-1}|+|x_{k-3}|+\cdots}\{\cdots \{\{\{dx_1,x_2\},x_3\},\ldots,x_k\};
\]
note in particular that $m_1(x_1) = dx_1 = \{h,x_1\}-\delta x_1$. For $k=2$, define
\[
m_2(x_1,x_2) = \frac{1}{2} \left((-1)^{|x_1|}\{dx_1,x_2\} + \{x_1,dx_2\}
- \tb(x_1,x_2) + (-1)^{|x_1||x_2|}\tb(x_2,x_1)\right).
\]
To shorten the equations in this section, we will henceforth suppress the $k$ index in $m_k$.

Let
\[
\Pi :\thinspace \sft\tocomm \to \A\tocomm
\]
denote the projection map, which annihilates all words that contain $p$'s and preserves all words that do not. Then $\ell_k$ and $m$ are related by this projection:
\[
\ell_k(x_1,\ldots,x_k) = \Pi m(x_1,\ldots,x_k)
\]
for all $k \geq 1$.

\begin{lemma}
The operation $m$ satisfies the following properties: \label{lem:m-properties}
\begin{enumerate}
\item \label{it:m1}
For any $k\geq 1$ and any permutation $\sigma\in S_k$, 
\[
m(x_{\sigma(1)},\ldots,x_{\sigma(k)}) = \chi(\sigma,x_1,\ldots,x_k) m(x_1,\ldots,x_k),
\]
where $\chi(\sigma,x_1,\ldots,x_k)$ is the Koszul sign from Definition~\ref{def:Linfalg}.
\item \label{it:m2}
We have
\[
m(x_1,x_2) = (-1)^{|x_1|}\{m(x_1),x_2\}+(-1)^{|x_1||x_2|}\tb(x_2,x_1)
= \{x_1,m(x_2)\}-\tb(x_1,x_2),
\]
while for any $k \geq 3$,
\[
m(x_1,\ldots,x_{k-1},x_k) = (-1)^{|x_1|+\cdots+|x_{k-1}|} \{m(x_1,\ldots,x_{k-1}),x_k\}.
\]
\item \label{it:m3}
For any $k \geq 1$, if $x_1,\ldots,x_{k+1}$ are words in $\W\tocomm$, then
\[
\tb(x_{k+1},m(x_1,\ldots,x_k)) = \sum_{i=1}^k \beta(x_{k+1},x_i)x_{k+1}m(x_1,\ldots,x_k).
\]
\end{enumerate}
\end{lemma}

\begin{proof}
\eqref{it:m1} is proved in exactly the same way as the corresponding result for $\ell_k$ in Proposition~\ref{prop:skewsymmetry}. For \eqref{it:m2}, note that since $\{x_1,x_2\}=0$, it follows from Proposition~\ref{prop:d} that
\[
0 = d\{x_1,x_2\} = \{mx_1,x_2\}-(-1)^{|x_1|}\{x_1,mx_2\} + (-1)^{|x_1|}\tb(x_1,x_2) + (-1)^{|x_1|(|x_2|+1)}\tb(x_2,x_1).
\]
Combining this with the definition of $m(x_1,x_2)$ immediately gives \eqref{it:m2}.

For \eqref{it:m3}, we use Proposition~\ref{prop:beta}(\ref{it:beta4},\ref{it:beta5},\ref{it:beta6}) to see that $\vv(\{h,x_1\}) = \vv(\delta x_1) = \vv(x_1)$ and thus $\vv(dx_1) = \vv(x_1)$. By Proposition~\ref{prop:beta}\eqref{it:beta5} again, it follows that $\vv(m(x_1,\ldots,x_k)) = \sum_{i=1}^k \vv(x_i)$. Thus $\beta(x_{k+1},m(x_1,\ldots,x_k)) = \sum_{i=1}^k \beta(x_{k+1},x_i)$, and \eqref{it:m3} follows.
\end{proof}

To set up the next lemma, we introduce some notation. Fix $k \geq 2$ and $x_1,\ldots,x_k \in \A\tocomm$. For a nonempty subset $S \subset [k] = \{1,\ldots,k\}$, write $S = \{i_1,\ldots,i_\ell\}$ with $i_1<\cdots<i_{\ell}$, define
\begin{align*}
m(x_S) &= m(x_{i_1},\ldots,x_{i_\ell}) \\
|x_S| &= \sum_{j=1}^\ell |x_{i_j}| \\
v_S &= \sum_{j=1}^\ell v(x_{i_j}).
\end{align*}
For a partition of $[k] = \{1,\ldots,k\}$ into a union of two disjoint nonempty subsets $[k] = A \sqcup B$, let $\sigma_A$ denote the permutation of $[k]$ that sends $1,\ldots,k$ successively to the elements of $A$ in increasing order, followed by the elements of $B$ in increasing order; analogously let $\sigma_B$ denote the permutation sending $1,\ldots,k$ to the elements of $B$ in increasing order, followed by the elements of $A$ in increasing order.

\begin{lemma}
For any $k \geq 2$, we have:
\begin{align*}
(-1)^k dm(x_1,\ldots,x_k) &= \sum_{[k] = A \sqcup B} \chi(\sigma_A) (-1)^{(|x_A|+1)(k-|A|+1)}\{m(x_A),m(x_B)\} \\
&\qquad -\sum_{i=1}^k \chi(\sigma_{[k]\setminus\{i\}}) \tb(m(x_{[k]\setminus\{i\}}),x_i),
\end{align*}
where the first sum is over all unordered ways to partition $[k]$ into two disjoint subsets. \label{lem:dm}
\end{lemma}

\noindent
We note that the summand in the first sum is well-defined because a straightforward computation shows that it is symmetric in $A$ and $B$:
\[
\chi(\sigma_A) (-1)^{(|x_A|+1)(k-|A|+1)}\{m(x_A),m(x_B)\} = \chi(\sigma_B) (-1)^{(|x_B|+1)(k-|B|+1)}\{m(x_B),m(x_A)\}.
\]

\begin{proof}[Proof of Lemma~\ref{lem:dm}]
We induct on $k$. For $k=2$, the desired identity is:
\begin{equation}
dm(x_1,x_2) = \{m(x_1),m(x_2)\} - \tb(m(x_1),x_2)+(-1)^{|x_1||x_2|} \tb(m(x_2),x_1).
\label{eq:dm2}
\end{equation}
Now applying $d$ to both sides of the equality $m(x_1,x_2) = \{x_1,m(x_2)\}-\tb(x_1,x_2)$
from Lemma~\ref{lem:m-properties} and using Propositions~\ref{prop:d} and~\ref{prop:differential} yields
\begin{align*}
dm(x_1,x_2) &= \{m(x_1),m(x_2)\}-(-1)^{|x_1|}\{x_1,d^2x_2\}+(-1)^{|x_1|}\tb(x_1,dx_2) \\
&\quad +(-1)^{|x_1||x_2|}\tb(dx_2,x_1)-d\tb(x_1,x_2);
\end{align*}
thus to prove \eqref{eq:dm2}, it suffices to show that
\begin{equation}
-(-1)^{|x_1|}\{x_1,d^2x_2\}+(-1)^{|x_1|}\tb(x_1,dx_2)-d\tb(x_1,x_2)=-\tb(dx_1,x_2).
\label{eq:dm22}
\end{equation}
But by Propositions~\ref{prop:differential} and~\ref{prop:beta}\eqref{it:beta2}, we have
\begin{align*}
-(-1)^{|x_1|}\{x_1,d^2x_2\} &= -(-1)^{|x_1|}\{x_1,\tb(h,x_2)\} \\
&= -(-1)^{|x_2|(|x_1|+1)} \{\tb(h,x_2),x_1\} \\
&= -\tb(\{h,x_1\},x_2) + \{h,\tb(x_1,x_2)\} - (-1)^{|x_1|}\beta(x_1,x_2)x_1\{h,x_2\} \\
&= -\tb(dx_1,x_2)+d\tb(x_1,x_2)-(-1)^{|x_1|}\beta(x_1,x_2)x_1dx_2 \\
&\quad +\left(\delta\tb(x_1,x_2)-\tb(\delta x_1,x_2)+
(-1)^{|x_1|}\beta(x_1,x_2)x_1\delta x_2\right) \\
&= -\tb(dx_1,x_2)+d\tb(x_1,x_2)- (-1)^{|x_1|}\tb(x_1,dx_2) \\
&\quad +\left(\delta\tb(x_1,x_2)-\tb(\delta x_1,x_2)+
(-1)^{|x_1|}\beta(x_1,x_2)x_1\delta x_2\right)
\end{align*}
where we have used Proposition~\ref{prop:differential} in the first equality, Proposition~\ref{prop:beta}\eqref{it:beta2} in the third equality, and Lemma~\ref{lem:m-properties}\eqref{it:m3} in the fifth equality.
Finally, since $\vv(\delta x_1) = \vv(x_1)$ and $\ww(\delta x_1)=\ww(x_1)$ by Proposition~\ref{prop:beta}\eqref{it:beta5}, we have $\beta(\delta x_1,x_2) = \beta(x_1,x_2)$, and so
\begin{align*}
\delta\tb(x_1,x_2)&-\tb(\delta x_1,x_2)-
(-1)^{|x_1|}\beta(x_1,x_2)x_1\delta x_2 \\
& = \beta(x_1,x_2) \left(\delta(x_1x_2)-(\delta x_1)x_2-(-1)^{|x_1|}x_1(\delta x_2)\right)\\
&= 0
\end{align*}
This proves \eqref{eq:dm22} and completes the proof of the lemma for the base case $k=2$.

For the inductive step, assume that the desired equality holds for $k \geq 2$; we establish it for $k+1$. Indeed, by Proposition~\ref{prop:d}, Lemma~\ref{lem:m-properties}, and the induction assumption, we have
\begin{align*}
(-1)^{k+1} &dm(x_1,\ldots,x_{k+1}) \\
&= (-1)^{k+1+|x_{[k]}|} d\{m(x_1,\ldots,x_k),x_{k+1}\} \\
&= 
(-1)^{k+1+|x_{[k]}|}\{dm(x_1,\ldots,x_k),x_{k+1}\} +\{m(x_{[k]}),dx_{k+1}\} \\
&\quad  - \tb(m(x_{[k]}),x_{k+1})- (-1)^{(k+|x_{[k]}|)|x_{k+1}|} \tb(x_{k+1},m(x_{[k]})) \\
&= C+D+\{m(x_{[k]}),dx_{k+1}\} \\
&\quad - \tb(m(x_{[k]}),x_{k+1})- (-1)^{(k+|x_{[k]}|)|x_{k+1}|} \tb(x_{k+1},m(x_{[k]}))
\end{align*}
where
\begin{align*}
C &=  \sum_{[k]=A\sqcup B} \chi(\sigma_A) (-1)^{1+|x_{[k]}|+(|x_A|+1)(k-|A|+1)}\{\{m(x_A),m(x_B)\},x_{k+1}\} \\
D &= (-1)^{|x_{[k]}|} \sum_{i=1}^k \chi(\sigma_{[k]\setminus\{i\}})  \{\tb(m(x_{[k]\setminus\{i\}}),x_i),x_{k+1}\}.
\end{align*}
Now by the Jacobi identity, we compute that the summand in $C$ is equal to
\begin{align*}
&\chi(\sigma_A) (-1)^{1+|x_{[k]}|+(|x_A|+1)(k-|A|+1)}\{m(x_A),\{m(x_B),x_{k+1}\}\}\\
&\qquad + 
 \chi(\sigma_B) (-1)^{1+|x_{[k]}|+(|x_B|+1)(k-|B|+1)}\{m(x_B),\{m(x_A),x_{k+1}\}\}.
\end{align*}
By Lemma~\ref{lem:m-properties}, these two terms are equal to $\chi(\sigma_A) (-1)^{(|x_A|+1)(k-|A|)}\{m(x_A),m(x_{B\cup\{k+1\}})\}$ and $\chi(\sigma_B) (-1)^{(|x_B|+1)(k-|B|)}\{m(x_B),m(x_{A\cup\{k+1\}})\}$ respectively, unless either of $A$ or $B$ (say $B$) is a singleton $\{i\}$, in which case we need to add
$\chi(\sigma_{[k]\setminus\{i\}})(-1)^{|x_{[k]\setminus\{i\}}|+|x_i||x_{k+1}|} 
\{m(x_{[k]\setminus\{i\}}),\tb(x_{k+1},x_i)\}$
to the first term. It follows that
\begin{align*}
C &=
E+\sum_{[k]=A\sqcup B} \chi(\sigma_A) (-1)^{(|x_A|+1)(k-|A|)}\{m(x_A),m(x_{B\cup\{k+1\}})\} \\
&\quad + \sum_{[k]=A\sqcup B} \chi(\sigma_B) (-1)^{(|x_B|+1)(k-|B|)}\{m(x_B),m(x_{A\cup\{k+1\}})\} \\
&= E + \sum_{[k+1]=A\sqcup B}  \chi(\sigma_A) (-1)^{(|x_A|+1)(k-|A|)}\{m(x_A),m(x_B)\} - \{m(x_{[k]}),dx_{k+1}\} \\
\end{align*}
where
\[
E =\sum_{i=1}^k  \chi(\sigma_{[k]\setminus\{i\}})(-1)^{|x_{[k]\setminus\{i\}}|+|x_i||x_{k+1}|} 
\{m(x_{[k]\setminus\{i\}}),\tb(x_{k+1},x_i)\}.
\]

From the above expression for $C$, to complete the induction step it suffices to prove:
\begin{equation}
D+E = (-1)^{(k+|x_{[k]}|)|x_{k+1}|} \tb(x_{k+1},m(x_{[k]}))-\sum_{i=1}^k 
\chi(\sigma_{[k+1]\setminus\{i\}}) \tb(m(x_{[k+1]\setminus\{i\}}),x_i).
\label{eq:DE}
\end{equation}

We will establish \eqref{eq:DE} first for $k \geq 3$, and then for $k=2$; the $k=2$ case involves some additional correction terms that we address at that point. For now, assume $k \geq 3$. By Proposition~\ref{prop:beta} \eqref{it:beta2}, we have
\begin{equation}
\begin{split}
\{&m(x_{[k]\setminus\{i\}}), \tb(x_{k+1},x_i)\} + (-1)^{|x_i|(|x_{k+1}|+1)}\{\tb(m(x_{[k]\setminus\{i\}}),x_i),x_{k+1}\} \\
& =  \tb(\{m(x_{[k]\setminus\{i\}}),x_{k+1}\},x_i) + (-1)^{|x_{k+1}|(|m(x_{[k]\setminus\{i\}})|+1)} \beta(x_{k+1},x_i) x_{k+1} \{m(x_{[k]\setminus\{i\}}),x_i\} \\
&= (-1)^{|x_{[k]\setminus\{i\}}|} \tb(m(x_{[k+1]\setminus\{i\}}),x_i) + \chi(\sigma_{[k]\setminus\{i\}})  
(-1)^{|x_{k+1}|(|x_{[k]\setminus\{i\}}|+k)+|x_{[k]\setminus\{i\}}|} \beta(x_{k+1},x_i) x_{k+1} m(x_{[k]});
\end{split}
\label{eq:k3}
\end{equation}
multiplying both sides by $\chi(\sigma_{[k]\setminus\{i\}})(-1)^{|x_{[k]\setminus\{i\}}|+|x_i||x_{k+1}|}$ and summing over $i$ gives
\begin{align*}
D+E &= \sum_{i=1}^k \left( \chi(\sigma_{[k]\setminus\{i\}})(-1)^{|x_i||x_{k+1}|} \tb(m(x_{[k+1]\setminus\{i\}}),x_i)
+ \beta(x_{k+1},x_i)m(x_{[k]})x_{k+1} \right) \\
&= (-1)^{(k+|x_{[k]}|)|x_{k+1}|} \tb(x_{k+1},m(x_{[k]}))
 - \sum_{i=1}^k \chi(\sigma_{[k+1]\setminus\{i\}})\tb(m(x_{[k+1]\setminus\{i\}}),x_i) \\
\end{align*}
where in the second step we have used Lemma~\ref{lem:m-properties} \eqref{it:m3}. This completes the proof of \eqref{eq:DE} when $k\geq 3$.

When $k=2$, since $\{m(x_1),x_2\}$ is equal not to $(-1)^{|x_1|}m(x_1,x_2)$ but rather to $(-1)^{|x_1|}m(x_1,x_2)$ plus the correction term $-(-1)^{|x_1|(|x_2|+1)} \tb(x_2,x_1)$, we need to add correction terms to the final expression in \eqref{eq:k3}. The total contribution of these correction terms summed over $i=1,2$ is:
\begin{align*}
F &:= -(-1)^{|x_1||x_2|+|x_1||x_3|+|x_2||x_3|}\tb(\tb(x_3,x_2),x_1)-(-1)^{(|x_1|+|x_2|)|x_3|}\beta(x_3,x_1)x_3\tb(x_1,x_2) \\
&\quad - (-1)^{(|x_1|+|x_2|)|x_3|+1}\tb(\tb(x_3,x_1),x_2) - (-1)^{|x_1||x_2|+|x_1||x_3|+|x_2||x_3|+1}\beta(x_3,x_2)x_3\tb(x_2,x_1).
\end{align*}
To establish \eqref{eq:DE} when $k=2$, it suffices to show $F=0$. 
Now since
\begin{align*}
\tb(\tb(x_3,x_2),x_1) &= \beta(\tb(x_3,x_2),x_1) \tb(x_3,x_2)x_1 \\
&= (\beta(x_3,x_1)+\beta(x_2,x_1))\beta(x_3,x_2)x_3x_2x_1 \\
&= (-1)^{|x_1||x_2|+|x_1||x_3|+|x_2||x_3|} (\beta(x_3,x_1)+\beta(x_2,x_1))\beta(x_3,x_2)x_1x_2x_3
\end{align*}
with a similar expression for $\tb(\tb(x_3,x_1),x_2)$, we find that
\begin{align*}
F &= 
-(\beta(x_3,x_1)+\beta(x_2,x_1))\beta(x_3,x_2)x_1x_2x_3-\beta(x_3,x_1)\beta(x_1,x_2)x_1x_2x_3 \\
&\quad + (\beta(x_3,x_2)+\beta(x_1,x_2))\beta(x_3,x_1)x_1x_2x_3 + \beta(x_3,x_2)\beta(x_2,x_1)x_1x_2x_3 \\
&=0,
\end{align*}
and we are done.
\end{proof}

\begin{proof}[Proof of Proposition~\ref{prop:htpyPoisson}]
Recall that $\ell_k$ and $m$ are related by $\ell_k = \Pi \circ m$ where $\Pi$ is the projection map $\sft\tocomm \to \alg\tocomm$. We note that $\Pi$ satisfies the following properties:
\begin{itemize}
\item
for any $x\in \sft\tocomm$, $\Pi d \Pi(x) = \Pi d(x)$;
\item
for any $x_1,x_2 \in \sft\tocomm$, $\Pi\{x_1,x_2\} = \Pi\{\Pi x_1,x_2\}+\Pi\{x_1,\Pi x_2\}$.
\end{itemize}

We now turn to the proof of the Jacobi identities. The first Jacob identity is $\ell_1^2 = 0$, which follows directly from $\partial^2=0$ in the usual Chekanov--Eliashberg DGA for $\Lambda$. The second Jacobi identity is
\[
\ell_1(\ell_2(x_1,x_2)) = \ell_2(\ell_1(x_1),x_2)+(-1)^{|x_1|}\ell_2(x_1,\ell_1(x_2)).
\]
To prove this, we apply $\Pi$ to both sides of the $k=2$ case of Lemma~\ref{lem:dm}, as presented in equation~\eqref{eq:dm2}:
\begin{align*}
\ell_1(\ell_2(x_1,x_2)) &= \Pi d \Pi m(x_1,x_2) \\
&=
\Pi dm(x_1,x_2) \\
&= \Pi \{\Pi m(x_1),m(x_2)\} + \Pi \{m(x_1),\Pi m(x_2)\} \\
& \quad - \Pi \tb(m(x_1),x_2)+(-1)^{|x_1||x_2|} \Pi \tb(m(x_2),x_1) \\
&= \Pi\left(\{\ell_1(x_1),m(x_2)\}- \tb(\ell_1(x_1),x_2) \right) \\
&\quad + \Pi\left(\{m(x_1),\ell_1(x_2)\}+(-1)^{|x_1||x_2|} \tb(\ell_1(x_2),x_1)\right) \\
&= \Pi m(\ell_1(x_1),x_2) + (-1)^{|x_1|} \Pi m(x_1,\ell_1(x_2)) \\
&= \ell_2(\ell_1(x_1),x_2) + (-1)^{|x_1|} \ell_2(x_1,\ell_1(x_2)).
\end{align*}

It remains to prove the $k$-th Jacobi identity for $k \geq 3$, which we rewrite as:
\begin{equation}
0 = \ell_1(\ell_k(x_1,\ldots,x_k)) + \sum_{[k]=A\sqcup B} (-1)^{|A||B|} \left( \chi(\sigma_A) \ell_{|B|+1}(\ell_{|A|}(x_A),x_B)
+ \chi(\sigma_B) \ell_{|A|+1}(\ell_{|B|}(x_B),x_A) \right),
\label{eq:kJacobi}
\end{equation}
with the sum being over unordered nonempty sets $A,B$ as before. Here we use analogous notation to above: if $A = \{i_1,\ldots,i_p\}$ and $B = \{j_1,\ldots,j_{k-p}\}$ where the elements of $A$ and $B$ are in increasing order, then $\ell_{p}(x_A) := \ell_p(x_{i_1},\ldots,x_{i_p})$, 
$\ell_{k-p+1}(\ell_p(x_A),x_B) := \ell_{k-p+1}(\ell_p(x_A),x_{j_1},\ldots,x_{j_{k-p}})$, and so forth.

Since $k\geq 3$, we can assume that in the unordered partition $[k]=A\sqcup B$, we always have $|B|>1$. Then we have
\begin{align*}
\Pi\{\Pi m(x_A),m(x_B)\} &= (-1)^{(|x_A|+|A|+1)(|x_B|+|B|+1)+1} \Pi\{m(x_B),\ell_{|A|}(x_A)\} \\
&= (-1)^{(|x_A|+|A|+1)(|x_B|+|B|+1)+1+|x_B|}\Pi m(x_B,\ell_{|A|}(x_A)) \\
&= (-1)^{(|x_A|+|A|)(|B|+1)} \ell_{|B|+1}(\ell_{|A|}(x_A),x_B).
\end{align*}
Similarly, if $|A|>1$ then
\begin{align*}
\Pi\{m(x_A),\Pi m(x_B)\} &= (-1)^{|x_A|} \Pi m(x_A,\ell_{|B|}(x_B)) \\
&= (-1)^{|x_A|(|x_B|+|B|+1)+|A|} \ell_{|A|+1}(\ell_{|B|}(x_B),x_A) \\
&= \chi(\sigma_A)\chi(\sigma_B)(-1)^{(|x_A|+|A|)(|B|+1)}\ell_{|A|+1}(\ell_{|B|}(x_B),x_A),
\end{align*}
while if $|A|=1$, say $A = \{i\}$, then there is a correction term:
\begin{align*}
\Pi\{m(x_A),\Pi m(x_B)\} &= \Pi\{m(x_i),\ell_{k-1}(x_B)\} \\
&= (-1)^{|x_i|}\Pi \left(m(x_i,\ell_{k-1}(x_B)) - (-1)^{|x_i||\ell_{k-1}(x_B)|} \tb(\ell_{k-1}(x_B),x_i) \right) \\
&= \chi(\sigma_A)\chi(\sigma_B)(-1)^{(|x_A|+|A|)(|B|+1)}\ell_{|A|+1}(\ell_{|B|}(x_B),x_A) \\
&\quad + \chi(\sigma_{\{i\}})\chi(\sigma_{[k]\setminus\{i\}})(-1)^{k|x_i|+k} \tb(\ell_{k-1}(x_{[k]\setminus\{i\}}),x_i).
\end{align*}

Now apply $\Pi$ to both sides of the equation in the statement of Lemma~\ref{lem:dm}. The left hand side is $(-1)^k \Pi dm(x_1,\ldots,x_k) = (-1)^k \ell_1(\ell_k(x_1,\ldots,x_k))$. Using the fact that $\Pi\{m(x_A),m(x_B)\} = \Pi\{\Pi m(x_A),x_B\}+\Pi\{m(x_A),\Pi x_B\}$, we find that the first sum on the right hand side is:
\begin{align*}
\sum_{[k]=A\sqcup B} &(-1)^{|A||B|+k+1}\left(\chi(\sigma_A) \ell_{|B|+1}(\ell_{|A|}(x_A),x_B) +
\chi(\sigma_B) \ell_{|A|+1}(\ell_{|B|}(x_B),x_A)\right) \\
& + \sum_{i=1}^k \chi(\sigma_{[k]\setminus\{i\}}) \tb(\ell_{k-1}(x_{[k]\setminus\{i\}}),x_i).
\end{align*}
This last sum exactly cancels the image under $\Pi$ of the second sum on the right hand side of the statement of Lemma~\ref{lem:dm}, and \eqref{eq:kJacobi} follows for $k\geq 3$.
\end{proof}

\begin{proof}[Proof of Proposition~\ref{prop:knot}]
A key to the above proof of the $L_\infty$ relations for links is the statement (Proposition~\ref{prop:delta}) that on $\sft\tocomm$, $\delta$ is a derivation with respect to $\{\cdot,\cdot\}$ when we include some correction terms involving $\tb$. Now suppose that $\Lambda$ is a knot with a single base point. Recall in this case that $\delta$ is actually a derivation with respect to $\{\cdot,\cdot\}$ on the module of cyclic words $\sft\tocyc$, not just the commutative quotient $\sft\tocomm$, and further that there are no $\tb$ correction terms: see Proposition~\ref{prop:delta-knot}.

This allows us to lift the argument in the above proof of Proposition~\ref{prop:htpyPoisson} from $\sft\tocomm$ to $\sft\tocyc$. We can define $m_k :\thinspace (\A\tocyc)^{\otimes k} \to \sft\tocyc$ for $k \geq 1$ as in the previous proof:
\[
m_k(x_1,\ldots,x_k) = (-1)^{|x_{k-1}|+|x_{k-3}|+\cdots} \{\cdots \{\{\{dx_1,x_2\},x_3\},\ldots,x_k\}.
\]
Now however we do not need a separate definition for $m_2$ since the above definition for $m_2(x_1,x_2)$ is already graded-symmetric in $x_1$ and $x_2$. We now follow the proof of Proposition~\ref{prop:htpyPoisson} word for word, but with the significant simplification that we omit all terms involving $\beta$ or $\tb$. The resulting proof verifies that the $L_\infty$ operations on $\A\tocyc$ satisfy the $L_\infty$ relations, as desired.
\end{proof}

\section{Invariance}
\label{sec:invariance}

In this section, we prove a version of invariance for the $L_\infty$ structure constructed in Section~\ref{sec:l-infty} under Legendrian isotopy. More precisely, we show that the $L_\infty$ algebra $(\A\tocomm,\{\ell_k\}_{k=1}^\infty)$ is invariant up through the $\ell_2$ operation. For $\ell_1$, this is contained in the usual invariance statement for the Chekanov--Eliashberg DGA (involving ``stable tame isomorphism''); the new content is the invariance of $\ell_2$ as well.

\subsection{Statement of invariance}
\label{ssec:invariance-statement}

Given two $L_\infty$ algebras $(A^+,\{\ell_k^+\})$ and $(A^-,\{\ell_k^-\})$, there is a standard notion of an $L_\infty$ morphism from $A^+$ to $A^-$. This is a collection of graded-symmetric maps $\{f_k\}_{k=1}^\infty$ with $f_k :\thinspace (A^+)^{\otimes k} \to A^-$, satisfying a sequence of relations of the form:
\begin{equation}
\begin{split}
\sum_{i=1}^k &\sum_{\sigma\in\text{Unshuff}(i,k-i)} f_{k-i+1}(\ell_i^+(x_{\sigma(1)},\ldots,x_{\sigma(i)}),x_{\sigma(i+1)},\ldots,x_{\sigma(k)}) \\
& = \sum_{k_1+\cdots+k_j=k} \sum_{\sigma\in\text{Unshuff}(k_1,\ldots,k_j)}\ell_j^-(f_{k_1}(x_{\sigma(1)},\ldots,x_{\sigma(k_1)}),\ldots,
f_{k_j}(x_{\sigma(k-k_j+1)},\ldots,x_{\sigma(k)}))
\end{split}
\label{eq:morphism}
\end{equation}
for $k \geq 1$, where for readability we have omitted signs in the relations \eqref{eq:morphism}. (In our setting, where the $L_\infty$ algebras also have associative multiplication and are homotopy Poisson algebras, we would also require that each $f_k$ satisfies the Leibniz rule.) In particular, for $k=1$ we have $f_1\circ\ell_1^+ = \ell_1^-\circ f_1$, and so $f_1$ is a chain map $(A^+,\ell_1^+) \to (A^-,\ell_1^-)$.
An $L_\infty$ equivalence between $A^+$ and $A^-$ is then an $L_\infty$ morphism $\{f_k\}$ such that $f_1$ is a quasi-isomorphism, and two $L_\infty$ algebras are said to be equivalent if there is an $L_\infty$ equivalence between them.

We believe that the $L_\infty$ algebra defined in this paper is invariant under this notion of equivalence, but do not prove this full result here.

\begin{conjecture}
The homotopy Poisson algebra $(\A\tocomm,\{\ell_k\})$ associated to a pointed Legendrian link is an invariant of the Legendrian isotopy type of the pointed link, up to $L_\infty$ equivalence.
\label{conj:invariance}
\end{conjecture}

Instead, in this paper we establish a more limited statement of invariance, up through the $\ell_2$ operation. That is, we construct maps $f_1,f_2$ and check the relation \eqref{eq:morphism} for $k=1,2$. More precisely, we define the following.

\begin{definition}
Let $(A^+,\{\ell_k^+\})$ and $(A^-,\{\ell_k^-\})$ be two homotopy Poisson algebras. An \textit{$L_2$ morphism} from $(A^+,\{\ell_k^+\})$ to $(A^-,\{\ell_k^-\})$ is a pair of maps $f_1 :\thinspace A^+ \to A^-$ and $f_2 :\thinspace A^+ \otimes A^+ \to A^-$ such that $f_1$ is an algebra map with $f_1 \circ \ell_1^+ = \ell_1^- \circ f_1$, and $f_2$ satisfies the following properties:
\begin{itemize}
\item
graded symmetry: $f_2(x_2,x_1) = (-1)^{|x_1||x_2|+1} f_2(x_1,x_2)$;
\item
Leibniz rule:
\[
f_2(x_1,x_2x_2') = f_2(x_1,x_2) f_1(x_2') + (-1)^{(|x_1|+1)|x_2|} f_1(x_2) f_2(x_1,x_2');
\]
\item
$L_2$ morphism property:
\[
f_1(\ell_2^+(x_1,x_2)) = \ell_2^-(f_1(x_1),f_1(x_2)) + f_2(\ell_1^+(x_1),x_2)+(-1)^{|x_1|}f_2(x_1,\ell_1^+(x_2))+\ell_1^-(f_2(x_1,x_2)).
\]
\end{itemize}
An $L_2$ morphism such that $f_1$ is a quasi-isomorphism is an \textit{$L_2$ equivalence}; 
\label{def:L2}
two homotopy Poisson algebras are \textit{$L_2$ equivalent} if there is an $L_2$ equivalence between them.
\end{definition}

\noindent
We note for future reference that graded symmetry and the Leibniz rule imply the following alternate version of Leibniz:
\[
f_2(x_1x_1',x_2) = (-1)^{|x_1|}f_1(x_1)f_2(x_1',x_2) + (-1)^{|x_1'||x_2|} f_2(x_1,x_2)f_1(x_1').
\]

Recall that the homology $H_*(A,\ell_1)$ of a homotopy Poisson algebra is a Poisson algebra, with multiplication induced from $A$ and Poisson bracket induced by $\ell_2$. The following result is immediate from Definition~\ref{def:L2}.

\begin{proposition}
If the homotopy Poisson algebras $(A^+,\{\ell_k^+\})$ and $(A^-,\{\ell_k^-\})$ are $L_2$ equivalent, then the Poisson algebras $(H_*(A^+,\ell_1^+),\ell_2^+)$ and $(H_*(A^-,\ell_1^-),\ell_2^-)$ are isomorphic.
\end{proposition}

We can now state our invariance result; note in the following that ``Legendrian isotopic'' refers to an isotopy that sends base points to base points and Maslov potential to Maslov potential. 

\begin{proposition}
Suppose that $\Lambda^+$ and $\Lambda^-$ are pointed Legendrian links in $\R^3$ that are Legendrian isotopic.
Then the respective homotopy Poisson algebras $((\A^+)\tocomm,\{\ell_k^+\})$ and $((\A^-)\tocomm,\{\ell_k^-\})$ are $L_2$ equivalent. 
\label{prop:invariance}
In particular, the commutative Legendrian contact homologies $(LCH\tocomm_*(\Lambda^+),\ell_2^+)$ and $(LCH\tocomm_*(\Lambda^-),\ell_2^-)$ are isomorphic as Poisson algebras.
\end{proposition}

The remainder of Section~\ref{sec:invariance} is devoted to the proof of Proposition~\ref{prop:invariance}. In order to do this, we follow Chekanov's combinatorial proof that Legendrian contact homology is an isotopy invariant. In the $xy$ projection, any Legendrian isotopy can be decomposed into planar isotopies along with a succession of elementary moves:
\begin{itemize}
\item
moving a base point through the endpoint of a Reeb chord;
\item
a Reidemeister III move;
\item
a Reidemeister II move.
\end{itemize}
Recall that the $xy$ projection of a Legendrian isotopy is a regular homotopy, and thus we do not need to consider a Reidemeister I move.

In the following subsections, we will construct an $L_2$ equivalence associated to each of the elementary moves. For each move, Chekanov \cite{Che} (see also \cite{ENS}) constructs a chain map between Chekanov--Eliashberg DGAs (``stable tame isomorphism'') that induces an isomorphism on the homology of the noncommutative DGA. This chain map is our $f_1$; when $\A\tocomm$ is the commutative DGA over a field $\kk$ of characteristic $0$, $f_1$ induces an isomorphism on the homology of $(\A\tocomm,\d)$ as well, see \cite[Theorem~3.14]{ENS}. We then construct $f_2$ for each elementary move in such a way that $(f_1,f_2)$ is an $L_2$ morphism (and hence an $L_2$ equivalence) in the sense of Definition~\ref{def:L2}. Proposition~\ref{prop:invariance} then follows.

\subsection{Invariance under base point move}
\label{ssec:move}

Suppose that the pointed Legendrian links $\Lambda^+$ and $\Lambda^-$ are identical except that the base point $\bullet_i$ is moved through an endpoint of the Reeb chord $a_j$: see Figure~\ref{fig:move}.

\begin{figure}
\labellist
\small\hair 2pt
\pinlabel $\bullet_i$ at 4 26
\pinlabel $\bullet_i$ at 63 12
\pinlabel $a_j$ at 16 31
\pinlabel $a_j$ at 61 31
\pinlabel $\bullet_i$ at 109 19
\pinlabel $\bullet_i$ at 168 34
\pinlabel $a_j$ at 121 15
\pinlabel $a_j$ at 166 15
\pinlabel $\Lambda^+$ at 16 2
\pinlabel $\Lambda^-$ at 61 2
\pinlabel $\Lambda^+$ at 121 2
\pinlabel $\Lambda^-$ at 166 2
\endlabellist
\centering
\includegraphics[width=0.8\textwidth]{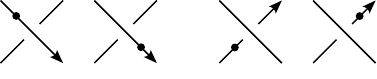}
\caption{
Moving a base point $\bullet_i$ through a Reeb chord $a_j$.
}
\label{fig:move}
\end{figure}
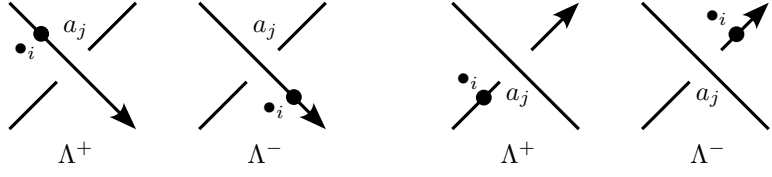

We first consider the case where $\bullet_i$ is moved through the positive endpoint $a_j^+$ of $a_j$, as shown in the left hand diagram in Figure~\ref{fig:move}. 
Define an algebra isomorphism
\[
f_1 :\thinspace (\A^+)\tocomm \to (\A^-)\tocomm
\]
by sending
\[
f_1(q_j) = t_i^{-1} q_j
\]
and letting $f_1$ act as the identity on all other generators of $(\A^+)\tocomm$ (including $t_i$). 

To set notation, let $h^\pm$ denote the Hamiltonians for $\Lambda^\pm$, and write $\SS_0 = \{q_1,\ldots,q_n,t_1,\ldots,t_s\}$ for the
the set of generators for both $(\A^+)\sfttocomm$ and $(\A^-)\sfttocomm$.
We now observe the following.

\begin{lemma}
\begin{enumerate}
\item \label{it:move2}
For any $x_1,x_2\in(\A^+)\sfttocomm$, we have $\{f_1(x_1),f_1(x_2)\} = f_1(\{x_1,x_2\})$.
\item \label{it:move3}
$f_1(h^+) = h^-$.
\item \label{it:move4}
For any $x_1,x_2\in\SS$, we have $(\ell_2^-)^\str(f_1(x_1),f_1(x_2)) = f_1((\ell_2^+)^\str(x_1,x_2))$, where $(\ell_2^\pm)^\str$ are the versions for $\Lambda^\pm$ of the operations $\ell_2^\str$ defined in Section~\ref{ssec:l2str}.
\end{enumerate}
\label{lem:move}
\end{lemma}

\begin{proof}
For \eqref{it:move2}, it suffices to check this when $x_1,x_2\in\SS$. In this case both sides are $0$ unless $\{x_1,x_2\} = \{p_\ell,q_\ell\}$ for some $\ell$. If $\ell\neq j$ then both sides are $\{x_1,x_2\}$, while if $\ell=j$ then the result follows from the fact that $\{f_1(p_j),f_1(q_j)\} = \{t_ip_j,t_i^{-1}q_j\} = \{p_j,q_j\}$.

To see \eqref{it:move3}, note that there is an obvious one-to-one correspondence between disks contributing to the Hamiltonians $h^\pm$. An inspection of the four quadrants at the Reeb chord $a_j$ shows that the boundaries of the disks that have a corner at $a_j$ change between $\Lambda^+$ and $\Lambda^-$ as follows: $t_i q_j \to q_j$; $q_j \to t_i^{-1} q_j$; $p_j t_i^{-1} \to p_j$; $p_j \to t_i p_j$. It follows immediately that $f_1(h^+) = h^-$.

For \eqref{it:move4}, recall from Proposition~\ref{prop:l2str} that $\ell_2^\str(x_1,x_2)$ is $x_1x_2$ times a signed count of the unbroken segments between $x_1$ and $x_2$. If neither $x_1$ nor $x_2$ is $q_j$, then this collection of unbroken segments is the same for $\Lambda^+$ and $\Lambda^-$, and so $(\ell_2^-)^\str(f_1(x_1),f_1(x_2)) = (\ell_2^-)^\str(x_1,x_2)= (\ell_2^+)^\str(x_1,x_2) = f_1((\ell_2^+)^\str(x_1,x_2))$. Since $(\ell_2^\pm)^\str(q_j,q_j) = 0$, it only remains to check \eqref{it:move4} when $x_1=q_j$ and $x_2\neq q_j$. 
In this case we have
$t_i f_1((\ell_2^+)^\str(q_j,x_2)) = (\ell_2^+)^\str(q_j,x_2)$ while
\[
t_i (\ell_2^-)^\str(f_1(q_j),f_1(x_2)) = t_i (\ell_2^-)^\str(t_i^{-1}q_j,x_2) = (\ell_2^-)^\str(q_j,x_2) -q_j t_i^{-1}(\ell_2^-)^\str(t_i,x_2).
\]
Thus we need to show that
\begin{equation}
(\ell_2^-)^\str(q_j,x_2) - (\ell_2^+)^\str(q_j,x_2) = q_j t_i^{-1}(\ell_2^-)^\str(t_i,x_2).
\label{eq:l2str-move}
\end{equation}
Now there is a one-to-one correspondence between unbroken segments between endpoints of $a_j$ and of $x_2$ in $\Lambda^+$ and in $\Lambda^-$, with two exceptions: 
\begin{itemize}
\item
(oriented) unbroken segments in $\Lambda^+$ from $a_j^+$ to an endpoint of $x_2$, and 
\item
unbroken segments in $\Lambda^-$ from an endpoint of $x_2$ to $a_j^+$;
\end{itemize}
these have no counterpart on the other side. Thus $(\ell_2^-)^\str(q_j,x_2)-(\ell_2^+)^\str(q_j,x_2)$ is given by a signed count of these exceptional segments. But in these exceptional cases there are also companion unbroken segments between $\bullet_i$ and the endpoint of $x_2$: an unbroken segment in $\Lambda^+$ from $\bullet_i$ to the endpoint of $x_2$ in the first case, and in $\Lambda^-$ from the endpoint of $x_2$ to $\bullet_i$ the latter. Thus each term on the left hand side of \eqref{eq:l2str-move} has a corresponding term on the right hand side. Completing the proof now reduces to a case-by-case verification that the terms agree with signs, using Proposition~\ref{prop:l2str}. Suppose for example that $x_2 = q_{j_2}$ and there is an unbroken segment in $\Lambda^+$ from $a_j^+$ to $a_{j_2}^{\sigma_2}$: this contributes $\frac{1}{2}\sigma_2q_jq_{j_2}$ to $-(\ell_2^+)^\str(q_j,x_2)$, while the companion segment in $\Lambda^+$ from $\bullet_i$ to $a_{j_2}^{\sigma_2}$ contributes $\frac{1}{2}\sigma_2q_jx_2$ to $q_j t_i^{-1}(\ell_2^-)^\str(t_i,x_2)$. We leave the other cases (a segment in $\Lambda^-$ from an endpoint of $x_2=q_{j_2}$ to $a_j^+$, and the cases where $x_2 = t_{i'}$ for some $i'$) to the reader.
\end{proof}

\begin{proposition}
Suppose that $\Lambda^\pm$ are related by moving a base point through the top endpoint of a Reeb chord. 
\label{prop:move}
For any $k\geq 1$ and $x_1,\ldots,x_k\in(\A^+)\tocomm$, we have
\[
f_1(\ell_k^+(x_1,\ldots,x_k)) = \ell_k^-(f_1(x_1),\ldots,f_1(x_k)).
\]
Thus $f_1$, as a map between the $L_\infty$ algebras $((\A^+)\tocomm,\{\ell_k^+\})$ and $((\A^-)\tocomm,\{\ell_k^-\})$, is an $L_\infty$ equivalence and in particular an $L_2$ equivalence.
\end{proposition}

\begin{proof}
For $k\neq 2$, we have
\begin{align*}
f_1(\ell_k^+(x_1,\ldots,x_k)) &= (-1)^{|x_{k-1}|+|x_{k-3}|+\cdots} f_1(\{\cdots \{\{h_k^+,x_1\},x_2\},\ldots,x_k\}) \\
&= (-1)^{|x_{k-1}|+|x_{k-3}|+\cdots} \{\cdots \{\{f_1(h_k^+),f_1(x_1)\},f_1(x_2)\},\ldots,f_1(x_k)\} \\
&= (-1)^{|x_{k-1}|+|x_{k-3}|+\cdots} \{\cdots \{\{h_k^-,f_1(x_1)\},f_1(x_2)\},\ldots,f_1(x_k)\} \\
&= \ell_k^-(f_1(x_1),\ldots,f_1(x_k))
\end{align*}
by Lemma~\ref{lem:move}(\ref{it:move2},\ref{it:move3}). For $k=2$, since $\ell_2(x_1,x_2) = (-1)^{|x_1|}\{\{h_2,x_1\},x_2\}+\ell_2^\str(x_1,x_2)$, we further need that
$ f_1((\ell_2^+)^\str(x_1,x_2))=(\ell_2^-)^\str(f_1(x_1),f_1(x_2))$; but this is precisely Lemma~\ref{lem:move} \eqref{it:move4}.
\end{proof}

This completes the proof that the $L_\infty$ structure is invariant under the move in the left hand diagram of Figure~\ref{fig:move}. For the move in the right hand diagram of Figure~\ref{fig:move}, an entirely analogous argument shows that $f_1 :\thinspace (\A^+)\tocomm \to (\A^-)\tocomm$ defined by $f_1(q_j) = t_iq_j$ and $f_1(x)=x$ for all other generators is an $L_\infty$ equivalence between the $L_\infty$ algebras $((\A^+)\tocomm,\{\ell_k^+\})$ and $((\A^-)\tocomm,\{\ell_k^-\})$. We conclude that the $L_\infty$ structure is invariant under any base point move.

\subsection{Invariance under Reidemeister III move}
\label{ssec:RIII}

Next we suppose that the pointed Legendrian links $\Lambda^+$ and $\Lambda^-$ are related by a Reidemeister III move, as shown in Figure~\ref{fig:RIII}. By moving base points if necessary and applying the invariance result from Section~\ref{ssec:move}, we can assume that the local Reidemeister III move does not involve any base points.

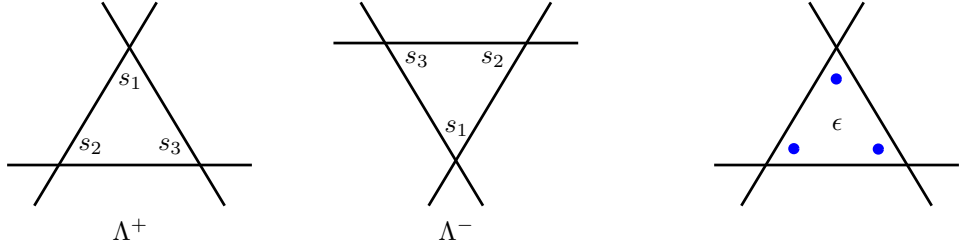
\begin{figure}
\labellist
\small\hair 2pt
\pinlabel $s_1$ at 34 42
\pinlabel $s_2$ at 23 24
\pinlabel $s_3$ at 45 24
\pinlabel $s_1$ at 124 29
\pinlabel $s_2$ at 134 48
\pinlabel $s_3$ at 113 48
\pinlabel $\Lambda^+$ at 34 1
\pinlabel $\Lambda^-$ at 124 1
\pinlabel $\epsilon$ at 229 30
\endlabellist
\centering
\includegraphics[width=\textwidth]{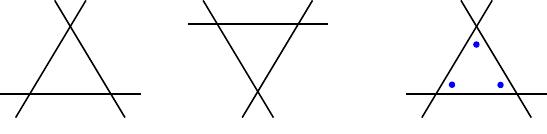}
\caption{
Reidemeister III move in the $xy$ projection, with any possible choice of crossing information at the three crossings. In each of $\Lambda^+$ and $\Lambda^-$, the three depicted quadrants are labeled $s_1,s_2,s_3$, where each $s_j$ is either a $p$ or $q$ as in Figure~\ref{fig:Reebsigns}. On the right, $\epsilon \in \{\pm 1\}$ is defined to be the product of the orientation signs of the three marked quadrants in $\Lambda^+$.
}
\label{fig:RIII}
\end{figure}

As usual, we append $\pm$ superscripts to label constructions in $\Lambda^\pm$: so for instance $\ell_k^\pm$ are the $L_\infty$ operations on $(\A^\pm)\tocomm$, and $h^\pm,\delta^\pm,d^\pm,\tb^\pm$ are the Hamiltonian, string coproduct, differential, and $\tb$ map (Definition~\ref{def:beta}) on $(\A^\pm)\sfttocomm$. There is an obvious one-to-one correspondence between base points and Reeb chords for $\Lambda^+$ and $\Lambda^-$, and we use the same labels for each, so that the $\pm$ algebras are the same: we can write
\begin{align*}
(\A^+)\tocomm = (\A^-)\tocomm &= \Z[q_1,\ldots,q_n,t_1^{\pm 1},\ldots,t_s^{\pm 1}] \\
(\A^+)\sfttocomm = (\A^-)\sfttocomm &= \Z[q_1,\ldots,q_n,p_1,\ldots,p_n,t_1^{\pm 1},\ldots,t_s^{\pm 1}].
\end{align*}
In particular, we label the Reeb chords for $\Lambda^\pm$ so that the three crossings involved in the Reidemeister III move correspond to Reeb chords $a_1,a_2,a_3$; we then define $s_1,s_2,s_3$ to be the associated $p$ or $q$ variables labeling the quadrants as shown in Figure~\ref{fig:RIII}, with $s_j \in \{p_j,q_j\}$.

To prove invariance for Reidemeister III, we will use the corresponding invariance proof for Legendrian SFT in \cite[\S 4.1]{SFT}. First we introduce a bit of notation.We use asterisks to denote signed duals: $p_j^* = q_j$ and $q_j^* = -p_j$. This follows the notation from the proof of Proposition~\ref{prop:delta} above; please note that this notation is slightly different from the notation used in \cite{SFT}, where $q_j^*$ is $+p_j$ rather than $-p_j$. Also, let $\epsilon,\epsilon_1 \in \{\pm 1\}$ denote the following signs: $\epsilon$ is the product of the orientation signs of the three shaded quadrants in the diagram for $\Lambda^+$ shown in the right-hand diagram in Figure~\ref{fig:RIII}, and $\epsilon_1$ records whether the orientation of $\Lambda^+$ agrees or disagrees with the depicted arrow.

We now define an algebra map $\phi :\thinspace (\A^+)\sfttocomm \to (\A^-)\sfttocomm$ by:
\begin{align*}
\phi(s_1) &= s_1 - \epsilon s_3^* s_2^* \\
\phi(s_2) &= s_2 - \epsilon s_1^* s_3^* \\
\phi(s_3) &= s_3 - \epsilon s_2^* s_1^*
\end{align*}
and $\phi$ acts as the identity on all other generators of $(\A^+)\sfttocomm = (\A^-)\sfttocomm$. Note that $\phi$ is a filtered map with respect to the $p$ filtrations on $(\A^\pm)\sfttocomm$ since not all of $s_1,s_2,s_3$ can be $p$'s (otherwise the move is not topologically allowed). The map $\phi$ was previously defined in \cite[\S 4.1]{SFT}, with slightly different signs because of the above-mentioned convention change for duals, and satisfies the following key properties.

\begin{lemma}
\begin{enumerate}
\item \label{it:phiRIII1}
For any $x_1,x_2\in(\A^+)\sfttocomm$, we have $\phi(\{x_1,x_2\}) = \{\phi(x_1),\phi(x_2)\}$.
\item \label{it:phiRIII2}
We have $\phi \circ d^+ = d^- \circ \phi$.
\item \label{it:phiRIII3}
For any $x_1,x_2\in(\A^+)\sfttocomm$, we have $\phi(\tb^+(x_1,x_2)) = \tb^-(\phi(x_1),\phi(x_2))$.
\end{enumerate}
\label{lem:phiRIII}
\end{lemma}

\begin{proof}
Properties \eqref{it:phiRIII1} and \eqref{it:phiRIII2} are proven in \cite[Lemma 4.1]{SFT} and \cite[Proposition 4.5]{SFT}, respectively; note as before that the sign conventions in \cite{SFT} are different from the ones here, but the proofs carry over to the new conventions. For \eqref{it:phiRIII3}, note that for any generator $x$ of $(\A^+)\sfttocomm$, we have $\vv^+(x) = \vv^-(\phi(x))$ and $\ww^+(x) = \ww^-(\phi(x))$: this is clear unless $x=s_j$ for some $j=1,2,3$, while for $x=s_j$, say $x=s_1$, it follows from the fact that $\vv^-(s_3^*s_2^*) = \vv^-(s_1)$ and $\ww^-(s_3^*s_2^*) = \ww^-(s_1)$. (The latter is trivial since both sides are $0$; for the former, note that the capping paths $\gamma_{s_1}$ and $\gamma_{s_3^*}\cdot\gamma_{s_2^*}$ agree outside of the triangle in the Reidemeister move.) Property \eqref{it:phiRIII3} now follows directly from the definition of $\tb$ in Definition~\ref{def:beta}.
\end{proof}

We now construct the maps $f_1,f_2$ that will constitute an $L_2$ equivalence between $((\A^+)\tocomm,\{\ell_k^+\})$ and $((\A^-)\tocomm,\{\ell_k^-\})$. Define $f_1 :\thinspace (\A^+)\tocomm \to (\A^-)\tocomm$ and $f_2 :\thinspace (\A^+)\tocomm \otimes (\A^+)\tocomm \to (\A^-)\tocomm$ as follows:
\begin{align*}
f_1(x_1) &= \Pi \phi(x_1) \\
f_2(x_1,x_2) &= (-1)^{|x_1|+1}\Pi\{\Pi \phi(x_1),\phi(x_2)\},
\end{align*}
where we recall that $\Pi :\thinspace (\A^\pm)\sfttocomm \to (\A^\pm)\tocomm$ is the projection map that sends any word containing $p$'s to $0$.

\begin{proposition}
Suppose that $\Lambda^\pm$ are related by a Reidemeister III move. Then the maps $(f_1,f_2)$ are an $L_2$ equivalence between $((\A^+)\tocomm,\{\ell_k^+\})$ and $((\A^-)\tocomm,\{\ell_k^-\})$.
\label{prop:RIII}
\end{proposition}

\begin{proof}
First note that by Lemma~\ref{lem:phiRIII} \eqref{it:phiRIII2} and the fact that $\phi$ and $d^-$ are filtered maps, we have
\[
f_1 \circ \ell_1^+ = \Pi \phi \Pi d^+ = \Pi \phi d^+ = \Pi d^- \phi = \Pi d^- \Pi \phi = \ell_1^- \circ f_1.
\]
It thus suffices to show that $f_2$ satisfies the properties listed in Definition~\ref{def:L2}.

It will be useful to write
\[
(\A^\pm)\sfttocomm = \A_0 \oplus \A_1 \oplus \A_2 \oplus \cdots
\]
where $\A_k$ is the $\Z$-module generated by words in $q_1,\ldots,q_n,p_1,\ldots,p_n,t_1^{\pm 1},\ldots,t_s^{\pm 1}$ containing exactly $k$ $p$'s; then $\A_0 = (\A^\pm)\tocomm$, the map $\Pi$ is simply projection to $\A_0$, and the $p$ filtration on $(\A^\pm)\sfttocomm$ is given by $\F^k (\A^\pm)\sfttocomm = \oplus_{m\geq k} \A_m$ for $k\geq 0$. Now decompose
\begin{align*}
\phi &= \phi_0 + \phi_1 + \cdots \\
d^+ &= d_0^+ + d_1^+ + \cdots \\
d^- &= d_0^- + d_1^- + \cdots
\end{align*}
where $\phi_m,d^+_m,d^-_m$ are the components of $\phi,d^+,d^-$ raising filtration level by $m$, i.e., mapping $\A_k$ to $\A_{m+k}$ for each $k$. Then we can write $\phi_0 = f_1$, $d^+_0 = \{h_1^+,\cdot\}$, $d^-_0 = \{h_1^-,\cdot\}$, and
\[
f_2(x_1,x_2) = (-1)^{|x_1|+1}\{\phi_0 x_1,\phi_1 x_2\}.
\]

Now suppose $x_1,x_2\in(\A^+)\tocomm$. Then $\{x_1,x_2\} = 0$, and thus by Lemma~\ref{lem:phiRIII} \eqref{it:phiRIII1} we have $\{\phi(x_1),\phi(x_2)\} = \phi\{x_1,x_2\} = 0$. In particular we have
\begin{align*}
0 &= \Pi \{\phi(x_1),\phi(x_2)\} = \{\phi_0 x_1,\phi_1 x_2\} + \{\phi_1 x_1,\phi_0 x_2\} \\
&=  \{\phi_0 x_1,\phi_1 x_2\} +(-1)^{|x_1||x_2|+|x_1|+|x_2|}\{\phi_0 x_2,\phi_1 x_1\} \\
&= (-1)^{|x_1|+1} f_2(x_1,x_2) + (-1)^{|x_1||x_2|+|x_1|+1}f_2(x_2,x_1),
\end{align*}
and we conclude that $f_2$ is graded-symmetric. The fact that $f_2$ satisfies the Leibniz rule from Definition~\ref{def:L2} follows from graded symmetry and the fact that $\phi_0 = f_1$ is an algebra map:
\begin{align*}
f_2(x_1,x_2x_2') &= (-1)^{|x_2x_2'|(|x_1|+1)}\{\phi_0(x_2)\phi_0(x_2'),\phi_1(x_1)\} \\
&= (-1)^{|x_2x_2'|(|x_1|+1)}\phi_0(x_2)\{\phi_0(x_2'),\phi_1(x_1)\}+(-1)^{|x_2|(|x_1|+1)}\{\phi_0(x_2),\phi_1(x_1)\}\phi_0(x_2') \\
&= (-1)^{|x_2|(|x_1|+1)}f_1(x_2)f_2(x_1,x_2')+f_2(x_1,x_2)f_1(x_2').
\end{align*}

Finally we establish the $L_2$ relation. For $x_1,x_2\in(\A^+)\tocomm$, we calculate
\begin{align*}
f_1(\ell_2^+(x_1,x_2)) &= \Pi \phi\Pi m_2^+(x_1,x_2) \\
&= \Pi \phi \{x_1,d^+x_2\} - \Pi \phi \tb^+(x_1,x_2) \\
&= \Pi \{\phi x_1,d^-\phi x_2\}-\Pi \tb^-(\phi x_1,\phi x_2) \\
&= \{\phi_0 x_1,(d^-_1 \phi_0+d_0^-\phi_1)x_2\}+\{\phi_1x_1,d_0^-\phi_0x_2\}-\tb^-(\Pi\phi x_1,\Pi\Phi x_2) \\
&= \Pi(\{\phi_0x_1,d^-\phi_0x_2\}-\tb^-(\phi_0x_1,\phi_0x_2))+\{\phi_0x_1,d_0^-\phi_1x_2\}+\{\phi_1x_1,d_0^-\phi_0x_2\} \\
&= \Pi m_2^-(f_1(x_1),f_1(x_2))+\{\phi_0x_1,d_0^-\phi_1x_2\}+\{\phi_1x_1,\phi_0d_0^+x_2\} \\
&= \ell_2^-(f_1(x_1),f_1(x_2))+\{\phi_0x_1,d_0^-\phi_1x_2\}+(-1)^{|x_1|}f_2(x_1,\ell_1^+(x_2)).
\end{align*}
Furthermore, by Proposition~\ref{prop:d} and the fact that $\Pi$ annihilates $\tb(\phi_0x_1,\phi_1x_2)$ and $\tb(\phi_1x_2,\phi_0x_1)$ since both of these are multiples of $(\phi_0x_1)(\phi_1x_2) \in \F^1$, we have
\begin{align*}
\ell_1^-(f_2(x_1,x_2)) &= (-1)^{|x_1|+1}\Pi d^-\{\phi_0x_1,\phi_1x_2\} \\
&= (-1)^{|x_1|+1}\Pi \{d^-\phi_0x_1,\phi_1x_2\} + \Pi\{\phi_0x_1,d^-\phi_1x_2\} \\
&= (-1)^{|x_1|+1}\{d^-_0\phi_0x_1,\phi_1x_2\} + \{\phi_0x_1,d^-_0\phi_1x_2\}\\
&= (-1)^{|x_1|+1}\{\phi_0d^+_0x_1,\phi_1x_2\} + \{\phi_0x_1,d^-_0\phi_1x_2\}\\
&= -f_2(\ell_1^+(x_1),x_2) + \{\phi_0x_1,d^-_0\phi_1x_2\}.
\end{align*}
Subtracting this from the formula for $f_1(\ell_2^+(x_1,x_2))$ gives the $L_2$ relation, as desired.
\end{proof}

\begin{remark}
The proof of Proposition~\ref{prop:RIII} works equally well to produce an $L_2$ morphism between the $L_\infty$ algebras $(\A^\pm)\tocomm$ for any Legendrian links $\Lambda_\pm$, given the existence of a map $\phi :\thinspace (\A^+)\sfttocomm \to (\A^-)\sfttocomm$ satisfying the properties in Lemma~\ref{lem:phiRIII}. 
\label{rmk:L2}
For the base point move from Section~\ref{ssec:move}, such a map also exists and is defined for the base point move in the left hand diagram of Figure~\ref{fig:move} by
$\phi(q_j) = t_i^{-1}q_j$, $\phi(p_j) = t_i p_j$, and $\phi$ acts as the identity on all other generators. (The same map with $t_i$ replaced by $t_i^{-1}$ works for the right hand diagram of Figure~\ref{fig:move}.) This gives an alternate proof of $L_2$ equivalence for the base point move. However, note that Proposition~\ref{prop:move} proves the stronger result that a base point move gives an $L_\infty$ equivalence. 
\end{remark}

\subsection{Invariance under Reidemeister II move}
\label{ssec:RII}

We now suppose that $\Lambda^+$ and $\Lambda^-$ are related by a Reidemeister II move, with the $xy$ projection of $\Lambda^+$ having two more crossings $a_1,a_2$ as shown in Figure~\ref{fig:RII}. By moving base points as before, we can assume that the local Reidemeister II move does not involve any base points. As usual we use $\pm$ superscripts to denote operations on $\Lambda^\pm$.

\begin{figure}
\labellist
\small\hair 2pt
\pinlabel $p_1$ at 3 23
\pinlabel $q_1$ at 11 31
\pinlabel $p_1$ at 19 23
\pinlabel $q_1$ at 11 16
\pinlabel $q_2$ at 42 23
\pinlabel $p_2$ at 49 31
\pinlabel $q_2$ at 57 23
\pinlabel $p_2$ at 49 15
\pinlabel $\epsilon$ at 31 45
\pinlabel $\Lambda^+$ at 31 1
\pinlabel $\Lambda^-$ at 120 1
\endlabellist
\centering
\includegraphics[width=4in]{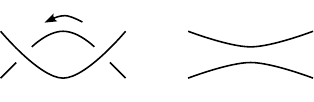}
\caption{
Reidemeister II move in the $xy$ projection. The sign $\epsilon$ is $\pm 1$ depending on whether the depicted arrow agrees or disagrees with the orientation of $\Lambda^+$.
}
\label{fig:RII}
\end{figure}
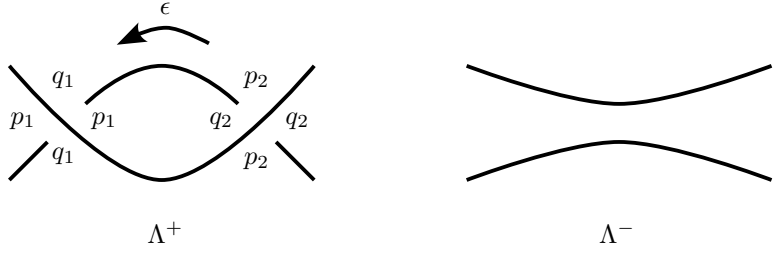

In \cite[\S 4.2]{SFT}, it is proved that the curved DGA associated to rational Legendrian SFT is invariant under a Reidemeister II move. In the language of the present paper, part of this proof involves the construction of an algebra map $\phi :\thinspace (\A^+)\sfttocomm \to (\A^-)\sfttocomm$ such that $\phi \circ d^+ = d^- \circ \phi$: in the language of \cite{SFT}, this is $\lim_{n\to\infty} \psi^n$, where $\psi^n$ is the $n$-th iterate of $\psi$. If this map satisfied $\phi\{x_1,x_2\} = \{\phi x_1,\phi x_2\}$ for $x_1,x_2\in(\A^+)\sfttocomm$, then we could use the same proof as for Reidemeister III moves in Section~\ref{ssec:RIII} to construct an $L_2$ equivalence between $(\A^+)\tocomm$ and $(\A^-)\tocomm$; see Remark~\ref{rmk:L2}. Unfortunately the $\phi$ map for Reidemeister II does not usually commute with $\{\cdot,\cdot\}$. Instead we will use a different strategy to prove $L_2$ equivalence under Reidemeister II, constructing $f_1$ and $f_2$ by hand.

To define these maps, note that the bigon in $\Lambda^+$ with corners at $p_1$ and $q_2$ contributes $\epsilon q_2p_1 = -(-1)^{|q_1|}\epsilon p_1q_2$ to the Hamiltonian $h^+$: see Figure~\ref{fig:RII-ori}. Then $\ell_1^+(q_1) = \{h_1^+,q_1\} = \epsilon(q_2-v)$ where we define
\[
v := -\epsilon\{h_1^+-q_2p_1,q_1\};
\]
geometrically $v$ counts disks with a single positive corner at the Reeb chord $a_1$ besides the bigon. A standard argument implies that $v$ does not involve either $q_1$ or $q_2$: by Stokes, the (positive, bounded below) area of any disk contributing to $v$ is the difference between the height of the Reeb chord $a_1$ and the sum of the heights of the Reeb chords for all negative corners of the disk, and the difference between the heights of the Reeb chords $a_1$ and $a_2$ can be arbitrarily small. Thus we have $v \in (\A^-)\tocomm$.

\begin{figure}
\labellist
\small\hair 2pt
\pinlabel $=-(-1)^{|q_1|}\epsilon$ at 93 13
\pinlabel $=-\epsilon$ at 250 13
\endlabellist
\centering
\includegraphics[width=\textwidth]{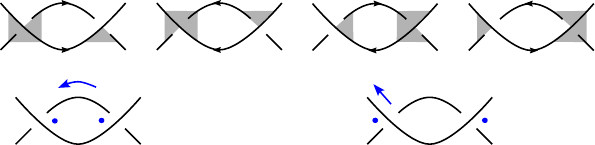}
\caption{
Top row: orientation signs for the four possible choices for strand orientations in $\Lambda^+$; shaded quadrants have $-$ orientation sign, unshaded quadrants $+$. For the first two diagrams, $|q_1|$ is even; for the last two, $|q_1|$ is odd. Bottom row: by inspection, in all four cases, the product of the orientation signs of the two marked quadrants, along with the sign of the arrow relative to the orientation of $\Lambda^+$, is as shown.
}
\label{fig:RII-ori}
\end{figure}

Define the algebra map $f_1 :\thinspace (\A^+)\tocomm \to (\A^-)\tocomm$ by
\[
f_1(q_1) = 0, \qquad f_1(q_2) = v,
\]
and $f_1$ is the identity on all other generators of $(\A^+)\tocomm$. As originally shown over $\Z/2$ by Chekanov \cite{Che} and extended to $\Z$ coefficients in \cite{ENS}, $f_1$ is a chain map between the commutative Chekanov--Eliashberg DGAs $((\A^+)\tocomm,\ell_1^+)$ and $((\A^-)\tocomm,\ell_1^-)$ inducing an isomorphism in homology. We record this fact in the following result.

\begin{proposition}[\cite{Che,ENS}]
The map $f_1$ is a chain map between the commutative Chekanov--Eliashberg DGAs $((\A^+)\tocomm,\ell_1^+)$ and $((\A^-)\tocomm,\ell_1^-)$: $f_1 \circ \ell_1^+ = \ell_1^- \circ f_1$. Furthermore, $f_1$ is a quasi-isomorphism.
\label{prop:RIIchain}
\end{proposition}

Proposition~\ref{prop:RIIchain} was originally shown over $\Z/2$ by Chekanov \cite{Che} in his proof that Legendrian contact homology is an isotopy invariant; the proof was extended to $\Z$ coefficients in \cite{ENS}. We will recall the proof that $f_1$ is a chain map because it will be useful in what follows. For the proof that $f_1$ is a quasi-isomorphism, see \cite{ENS}, specifically section 6.3 (the DGA is invariant up to stable tame isomorphism) and Theorem 3.14 (stable tame isomorphism is a quasi-isomorphism).

To set up the proof, we introduce some notation that we will use in the proof of the $L_2$ relation as well. For any $j$, write $\Delta_{p_j}^\pm$ for the subset of $\Delta(\Lambda^\pm)$ consisting of disks with a single positive corner at $a_j$ (i.e., a corner at $p_j$) and any number of negative corners, that is, the disks contributing to $\ell_1^\pm(q_j)$. For $j\neq 1,2$ and $m\geq 0$, write $\Delta_{p_j}^{+;m}$ for the subset of $\Delta_{p_j}^+$ consisting of disks with no negative corners at $a_1$ and $m$ negative corners at $a_2$ (i.e., $m$ corners at $q_2$ and no corners at $p_1$). Finally, the set $\Delta_{p_1}^+$ includes the bigon with corners at $p_1$ and $q_2$; we write ${\Delta^+_{p_1}}'$ for the complement of this bigon in $\Delta_{p_1}^+$. Note that no disk in ${\Delta^+_{p_1}}'$ has a corner at $q_2$, and that $v$ is the signed sum of terms corresponding to the disks in ${\Delta^+_{p_1}}'$.

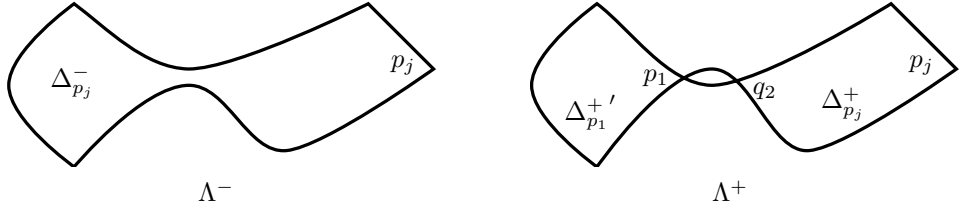
\begin{figure}
\labellist
\small\hair 2pt
\pinlabel $\Lambda^-$ at 48 2
\pinlabel $\Lambda^+$ at 166 2
\pinlabel $\Delta^-_{p_j}$ at 15 27
\pinlabel $p_j$ at 91 31
\pinlabel ${\Delta^+_{p_1}}'$ at 134 20
\pinlabel $\Delta^+_{p_j}$ at 192 22
\pinlabel $p_j$ at 210 31
\pinlabel $p_1$ at 149 28
\pinlabel $q_2$ at 174 25
\endlabellist
\centering
\includegraphics[width=\textwidth]{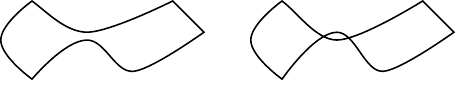}
\caption{
Following a disk in $\Delta_{p_j}^-$ from $\Lambda^-$ to $\Lambda^+$. On the right, the Reidemeister II move pinches the disk into two pieces connected by the $p_1q_2$ bigon.
}
\label{fig:RII-pinch}
\end{figure}

\begin{proof}[Proof of Proposition~\ref{prop:RIIchain}]
As mentioned above, we will show that $f_1$ is a chain map and refer the reader to \cite{Che,ENS} for the proof that it is a quasi-isomorphism.
It suffices to check $f_1\circ \ell_1^+(s) = \ell_1^- \circ f_1(s)$ on generators $s$ of $(\A^+)\tocomm$. For $s = t_i^{\pm 1}$, this is trivial since $\ell_1^\pm(t_i^{\pm 1}) = 0$; for $s=q_1$ and $s=q_2$, we have $f_1(\ell_1^+(q_1))=f_1(\epsilon(q_2-v))=0=\ell_1^-(f_1(q_1))$ and
\[
f_1(\ell_1^+(q_2))=f_1(\ell_1^+(\epsilon \ell_1^+(q_1)+v)) = f_1(\ell_1^+(v)) = 
\ell_1^-(f_1(v)) = \ell_1^-(f_1(q_2)).
\]

Now suppose that $s=q_j$ is a generator corresponding to a Reeb chord $a_j$ different from $a_1,a_2$; we need to check $f_1(\ell_1^+(q_j)) = \ell_1^-(q_j)$. Consider a disk $\Delta\in\Delta_{p_j}^-$. As we pass through the Reidemeister II move from $\Lambda^-$ to $\Lambda^+$, $\Delta$ may survive as an immersed disk, or it may pinch into a union of disks if some portion of $\Delta$ passes through the neck of $\Lambda^-$. Suppose that $\Delta$ passes through the neck of $\Lambda^-$ $m$ times. Then $\Delta$ will pinch into $m+1$ disks; see Figure~\ref{fig:RII-pinch} for an illustration for $m=1$. These $m+1$ disks in $\Lambda^+$ have corners collectively given by the corners of $\Delta$ along with $m$ $p_1,q_2$ pairs (i.e. a positive corner at $a_1$ and a negative corner at $a_2$). One of the disks, call it $\Delta_0$, contains the original positive corner at $a_j$. Since the rest of the disks must contain a positive corner by Stokes, they must each have a positive corner at $a_1$. It follows that each of these disks is in ${\Delta^+_{p_1}}'$, and that $\Delta_0$ is in $\Delta_{p_j}^{+;m}$. Conversely, we can glue together a disk in $\Delta_{p_j}^{+;m}$ and $m$ disks in ${\Delta^+_{p_1}}'$ to get a disk in $\Delta_{a_j}^-$. We conclude that there is a one-to-one correspondence
\[
\Delta_{p_j}^- \longleftrightarrow \bigcup_{m\geq 0} \left(\Delta_{p_j}^{+;m} \times ({\Delta^+_{p_1}}')^m\right).
\]

Suppose that $\Delta$ and $(\Delta_0,\Delta_1,\ldots,\Delta_m)$ are mapped to each other under this correspondence. Let $w$ denote the contribution of $\Delta$ to $\ell_1^-(q_j)$, and let $w_0$ and $w_1,\ldots,w_m$ denote the contributions of $\Delta_0$ and $\Delta_1,\ldots,\Delta_m$ to $\ell_1^+(q_j)$ and $v$, respectively. (Since $\ell_1^+(q_1) = \epsilon(q_2-v)$, $w_i$ is also $-\epsilon$ times the contribution of $\Delta_i$ to $\ell_1^+(q_1)$ for $i=1,\ldots,m$.) Then by the gluing process, $w$ is precisely the result of replacing the $m$ occurrences of $q_2$ in $w_0$ by $w_1,\ldots,w_m$. Indeed, this is true even if we account for the signs of $w,w_0,w_1,\ldots,w_m$: the cumulative difference in signs between the term in $\ell_1^-(q_j)$ and the terms in $\ell_1^+(q_j)$ and $\ell_1^+(q_1)$ is given by the product of the three signs in the bottom right diagram in Figure~\ref{fig:RII-ori} for each pinch. But from Figure~\ref{fig:RII-ori}, this product is equal to $-\epsilon$ and exactly cancels the difference between the sign of $w_i$ and the sign of its contribution to $\ell_1^+(q_1)$.

Restating the result of the previous paragraph in algebraic language, we have
$f_1(\ell_1^+(q_j)) = \ell_1^-(q_j)$, as desired.
\end{proof}

Having defined the map $f_1 :\thinspace (\A^+)\tocomm \to (\A^-)\tocomm$, we complete the construction of an $L_2$ morphism between $(\A^+)\tocomm$ and $(\A^-)\tocomm$ by defining a map $f_2 :\thinspace (\A^+)\tocomm \otimes (\A^+)\tocomm \to (\A^-)\tocomm$. Write the Reeb chords of $\Lambda^+$ as $a_1,a_2,\ldots,a_n$, where $a_1,a_2$ are the crossings involved in the Reidemeister II move. As usual, we will write $h_1^+$ and $h_2^+$ for the portions of the Hamiltonian of $\Lambda^+$ corresponding to disks with $1$ and $2$ positive punctures, respectively. 

Define $f_2$ on the generators of $(\A^+)\tocomm$ as follows: for $x_1,x_2 \in \{q_1,\ldots,q_n,t_1,\ldots,t_s\}$, $f_2(x_1,x_2) = 0$ unless at least one of $x_1,x_2$ is equal to $q_2$; and 
\begin{align*}
f_2(q_2,t_i) &= 0 \\
f_2(q_2,q_1) &= f_1\left( \frac{\epsilon}{2} \ell_2^+(q_1,q_1)\right) \\
f_2(q_2,q_2) &= \begin{cases} \epsilon f_1(\ell_2^+(q_1,q_2+v)) & |q_1| \text{ even} \\
0 & |q_1| \text{ odd} \end{cases} \\
f_2(q_2,q_j) &= f_1\left(
\epsilon \ell_2^+(q_1,q_j)
 - \frac{1}{2} \ell_2^+(q_1,q_1)\{p_2,\ell_1^+(q_j)\}\right).
\end{align*}
where the last equality is for all $j>2$. Extend $f_2$ to a bilinear form on all of $(\A^+)\tocomm$ by graded symmetry and the Leibniz rule.

For future use, we note the following.

\begin{lemma}
For any $x \in (\A^+)\tocomm$, we have
\label{lem:f2}
\[
f_2(q_1,x) = \frac{\epsilon(-1)^{|q_1|}}{2} f_1(\ell_2^+(q_1,q_1)\{p_2,x\}),
\]
and for any $x\in (\A^+)\tocomm$ that does not involve $q_2$, we have
\[
f_2(q_2,x) = f_1\left(\epsilon \ell_2^+(q_1,x)-\frac{1}{2}\ell_2^+(q_1,q_1)\{p_2,\ell_1^+(x)\}\right).
\]
\end{lemma}

\begin{proof}
By the definition of $f_2$, the first equality holds for $x\in\{q_1,\ldots,q_n,t_1,\ldots,t_s\}$; note that both sides are trivially $0$ unless $x=q_2$. Similarly, the second equality holds for $x\in\{q_1,q_3,\ldots,q_n,t_1,\ldots,t_s\}$. Now use the fact that both sides of both equalities are derivations with respect to $x$: we have
\begin{align*}
f_2(q_1,x_1x_2) &= f_2(q_1,x_1)f_1(x_2) + (-1)^{(|q_1|+1)|x_1|}f_1(x_1)f_2(q_1,x_2) \\
f_2(q_2,x_1x_2) &= f_2(q_2,x_2)f_1(x_2) + (-1)^{|q_1||x_1|}f_1(x_1)f_2(q_2,x_2)
\end{align*}
and it is straightforward to check that the right hand sides of the two equalities satisfy corresponding product rules in $x$.
\end{proof}

We can now state the invariance result for the Reidemeister II move.

\begin{proposition}
The maps $(f_1,f_2)$ are an $L_2$ equivalence between $((\A^+)\tocomm,\{\ell_k^+\})$ and $((\A^-)\tocomm,\{\ell_k^-\})$.
\label{prop:RII}
\end{proposition}

The remainder of this subsection is devoted to proving Proposition~\ref{prop:RII}. In light of Proposition~\ref{prop:RIIchain}, it suffices to verify the $L_2$ morphism property. 
To this end, define the map
$F :\thinspace  (\A^+)\tocomm \otimes (\A^+)\tocomm \to (\A^-)\tocomm$ by:
\[
F(x_1,x_2) = f_1(\ell_2^+(x_1,x_2)) - \ell_2^-(f_1(x_1),f_1(x_2)) - f_2(\ell_1^+(x_1),x_2)-(-1)^{|x_1|}f_2(x_1,\ell_1^+(x_2))-\ell_1^- (f_2(x_1,x_2)).
\]
We want to show that $F \equiv 0$. 

Since $f_1$ is an algebra map and $\ell_2^\pm$ and $f_2$ satisfy the Leibniz rule, it suffices to show that $F(x_1,x_2) = 0$ for generators $x_1,x_2 \in \{q_1,\ldots,q_n,t_1,\ldots,t_s\}$.
We break the verification of $F(x_1,x_2)=0$ for generators $x_1,x_2$ into several cases.

\vspace{11pt}

\noindent \textsc{Case 1: $x_1=q_i$ and $x_2=q_j$ for some $i,j>2$.}

From the definition of $F,f_1,f_2$, to show that $F(q_i,q_j) = 0$ we need to prove:
\[
\ell_2^-(q_i,q_j) = f_1(\ell_2^+(q_i,q_j)) - f_2(\ell_1^+(q_i),q_j) - (-1)^{|q_i|} f_2(q_i,\ell_1^+(q_j)).
\]
Note that we can write $\ell_2^\pm(q_i,q_j) = (\ell_2^\pm)\tosft(q_i,q_j)+(\ell_2^\pm)^\str(q_i,q_j)$, where $(\ell_2^\pm)^\str(q_i,q_j)$ is the string contribution to $\ell_2^\pm$ and
\[
(\ell_2^\pm)\tosft(q_i,q_j) = (-1)^{|q_i|} \{\{h_2^\pm,q_i\},q_j\}
\]
is the SFT contribution (counting disks with positive punctures at $p_i$ and $p_j$) to $\ell_2^\pm$. The string contributions $(\ell_2^-)^\str(q_i,q_j)$ and $(\ell_2^+)^\str(q_i,q_j)$ are equal since the configuration of base points relative to the endpoints of the Reeb chords $a_i,a_j$ is the same for $\Lambda^+$ and $\Lambda^-$. Thus we wish to show:
\begin{equation}
(\ell_2^-)\tosft(q_i,q_j) = f_1((\ell_2^+)\tosft(q_i,q_j)) - f_2(\ell_1^+(q_i),q_j) - (-1)^{|q_i|} f_2(q_i,\ell_1^+(q_j)).
\label{eq:qiqj}
\end{equation}

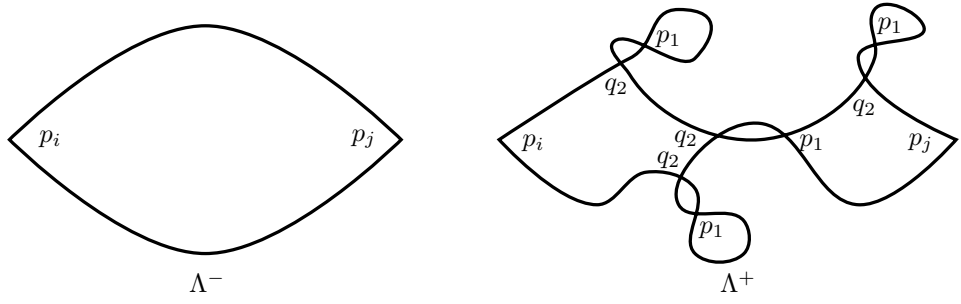
\begin{figure}
\labellist
\small\hair 2pt
\pinlabel $p_i$ at 10 34
\pinlabel $p_j$ at 82 34
\pinlabel $p_i$ at 121 33
\pinlabel $p_j$ at 210 33
\pinlabel $q_2$ at 156 34
\pinlabel $p_1$ at 185 33
\pinlabel $p_1$ at 152 57
\pinlabel $q_2$ at 140 46
\pinlabel $p_1$ at 162 13
\pinlabel $q_2$ at 152 29
\pinlabel $p_1$ at 203 60
\pinlabel $q_2$ at 197 41
\pinlabel $\Lambda^-$ at 46 1
\pinlabel $\Lambda^+$ at 168 1
\endlabellist
\centering
\includegraphics[width=\textwidth]{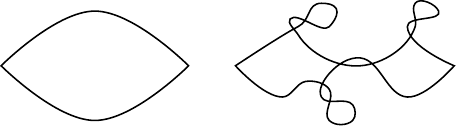}
\caption{
Following a disk in $\Delta_{p_i,p_j}^-$ from $\Lambda^-$ to $\Lambda^+$. On the right, a possibility for the corresponding picture in $\Lambda^+$. The core (``spine'') is a chain of two disks, one containing $p_i$ and one containing $p_j$, joined by a $p_1q_2$ bigon. There may also be some number of branches consisting of disks in ${\Delta^+_{p_1}}'$ joined to the core by a $p_1q_2$ bigon.
}
\label{fig:RII-pinch2}
\end{figure}

\begin{figure}
\labellist
\small\hair 2pt
\pinlabel $p_i$ at 17 86
\pinlabel $p_j$ at 81 86
\pinlabel $p_i$ at 140 86
\pinlabel $p_j$ at 222 86
\pinlabel $q_2$ at 166 86
\pinlabel $p_1$ at 195 86
\pinlabel $p_i$ at 8 30
\pinlabel $p_j$ at 90 30
\pinlabel $p_1$ at 35 30
\pinlabel $q_2$ at 64 30
\pinlabel $p_i$ at 113 30
\pinlabel $p_j$ at 248 30
\pinlabel $q_2$ at 140 30
\pinlabel $p_1$ at 169 30
\pinlabel $p_1$ at 192 30
\pinlabel $q_2$ at 221 30
\pinlabel I at 49 62
\pinlabel II at 180 62
\pinlabel III at 49 2
\pinlabel IV at 180 2
\endlabellist
\centering
\includegraphics[width=\textwidth]{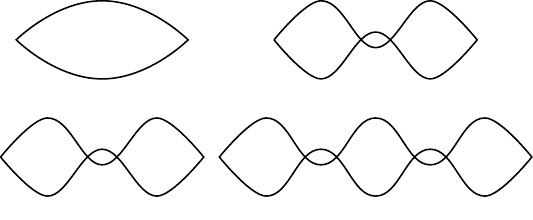}
\caption{
The four possibilities for the picture in $\Lambda^+$ corresponding to a disk in $\Lambda^-$ in $\Delta_{p_i,p_j}^-$. Each of these pictures may have branches (not depicted) given by disks in ${\Delta^+_{a_1}}'$ joined to the picture by a bigon.
}
\label{fig:RII-pinch3}
\end{figure}
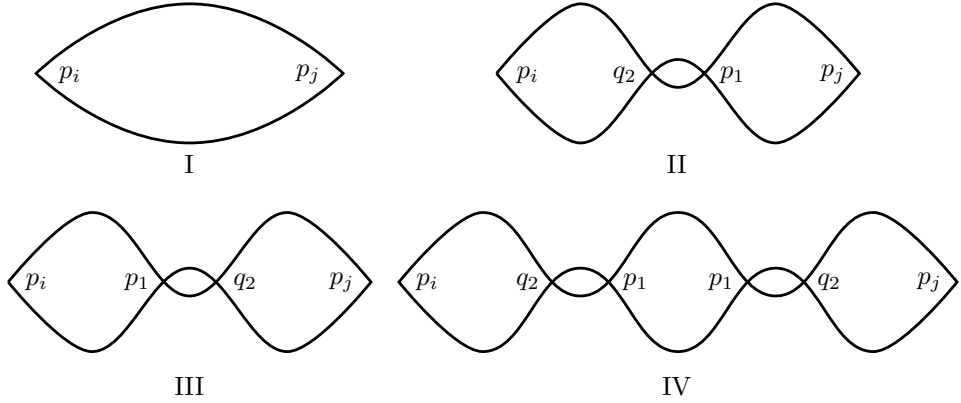

To this end, consider a disk $\Delta$ contributing to $(\ell_2^-)\tosft(q_i,q_j)$: this is a disk in $\Delta(\Lambda^-)$ with positive punctures at $p_i$ and $p_j$.
As in the proof of Proposition~\ref{prop:RIIchain}, follow $\Delta$ from $\Lambda^-$ through the Reidemeister II move to $\Lambda^+$. The move pinches $\Delta$ into a union of disks in $\Lambda^+$ connected by $p_1q_2$ bigons; see Figure~\ref{fig:RII-pinch2} for an illustration. We call this chain of disks $C$. Up to sign, the contribution of $\Delta$ to $(\ell_2^-)\tosft(q_i,q_j)$ is the product of the negative corners of $\Delta$. This is equal to the product of the negative corners of $C$, not counting the $q_2$ corners where disks are linked by bigons. We will show that this last product is a term on the right hand side of \eqref{eq:qiqj}, and conversely that every right hand term in \eqref{eq:qiqj} comes from some such configuration $C$. It will follow that there is a correspondence between the terms on the two sides of \eqref{eq:qiqj}, not including signs. We will then check that the signs of corresponding terms match.

Consider a chain $C$ of disks in $\Lambda^+$ corresponding to $\Delta$ as above.
Some of these disks may be disks in ${\Delta^+_{p_1}}'$, with a single positive puncture at $p_1$; we call these disks ``branches''. If we remove these branches, then what remains is either a disk with positive punctures at $p_i$ and $p_j$, or a chain of disks connecting a disk with positive puncture at $p_i$ and a disk with positive puncture at $p_j$, with the chain linked by $p_1q_2$ bigons. Since $\Delta$ has only positive punctures at $p_i$ and $p_j$, each pinch through the Reidemeister move introduces a $p_1,q_2$ pair of punctures, and each disk in $C$ must have at least one positive puncture, $C$ must fall into one of the four categories shown in Figure~\ref{fig:RII-pinch3}.

Up to sign, the contribution of $\Delta$ to $(\ell_2^-)\tosft(q_i,q_j)$ is the product of the negative corners of $\Delta$. We claim that, first up to sign, there is a corresponding term on the right hand side of \eqref{eq:qiqj}. Which term depends on whether the configuration in $\Lambda^+$ corresponding to $\Delta$, which we denote by $C$, is of type I, II, III, or IV, in the terminology of Figure~\ref{fig:RII-pinch3}. In each case, the negative corners of $\Delta$ are the same as the negative corners of $C$, not including any $q_2$ corners.

If $C$ is of type I, then the product of the negative corners of $C$, without the branches but including any $q_2$ adjacent to a branch, is a term in $\ell_2^+(q_i,q_j)$; thus, with the branches, the product of the negative corners besides $q_2$'s is a term in $f_1(\ell_2^+(q_i,q_j))$. If $C$ is of type II, write $\Delta_1$ and $\Delta_2$ for the $p_iq_2$ disk and the $p_1p_j$ disk, respectively, where each disk may contain $q_2$ corners attached to branches. The product of the negative corners of $\Delta_2$ is a term in $\ell_2^+(q_1,q_j)$; attaching branches and taking the product of all negative corners besides $q_2$'s then gives a term in $f_1(\ell_2^+(q_1,q_j))$ and thus a term in $f_2(q_2,q_j)$ by the definition of $f_2$. On the other hand, if the product of the negative corners of $\Delta_1$ is $vq_2$ for some $v$, then attaching branches to $\Delta_1$ and taking the product of all negative corners besides $q_2$'s gives a term in $f_1(v)$. The total product of negative corners between $\Delta_1$, $\Delta_2$, and their branches is thus a term in $f_1(v)f_2(q_2,q_j)$ and consequently a term in $f_2(\ell_1^+(q_i),q_j)$. Similarly, if $C$ is of type III, then the product of the negative corners of $C$ besides $q_2$'s is a term in $f_2(q_i,\ell_1^+(q_j))$.

If $C$ is of type IV, there are two contributions to the right hand side of \eqref{eq:qiqj} coming from $C$. One comes from $f_2(\ell_1^+(q_i),q_j)$: $f_2(q_2,q_j)$ contains a term of the form $f_1(\ell_2^+(q_1,q_1)\{p_2,\ell_1^+(q_j)\})$ which corresponds to the product of the negative corners of the $p_1p_1$ and $q_2p_j$ disks in $C$ along with their attached branches, and passing from $f_2(q_2,q_j)$ to $f_2(\ell_1^+(q_i),q_j)$ is multiplication by the negative corners of the $p_iq_2$ disk along with its branches. There is symmetrically another contribution from $C$ coming from $f_2(q_i,\ell_1^+(q_j))$. Each of these contributions is multiplied by $\frac{1}{2}$, and (again up to sign) it follows that the contribution of $C$ to the right hand side of \eqref{eq:qiqj} is equal to the contribution of $\Delta$ to the left hand side.

Conversely, every term on the right hand side of \eqref{eq:qiqj} comes from some chain $C$ of disks in $\Lambda^+$, and passing $C$ through the Reidemeister II move to $\Lambda^-$ gives an honest disk in $\Lambda^-$ and thus a term on the left hand side of \eqref{eq:qiqj}. Thus we have verified \eqref{eq:qiqj} up to signs.

Finally, we check that the signs agree. For type I, this is clear from the definition of $\ell_2^\pm$. We will check type II (type III is very similar) and type IV. 

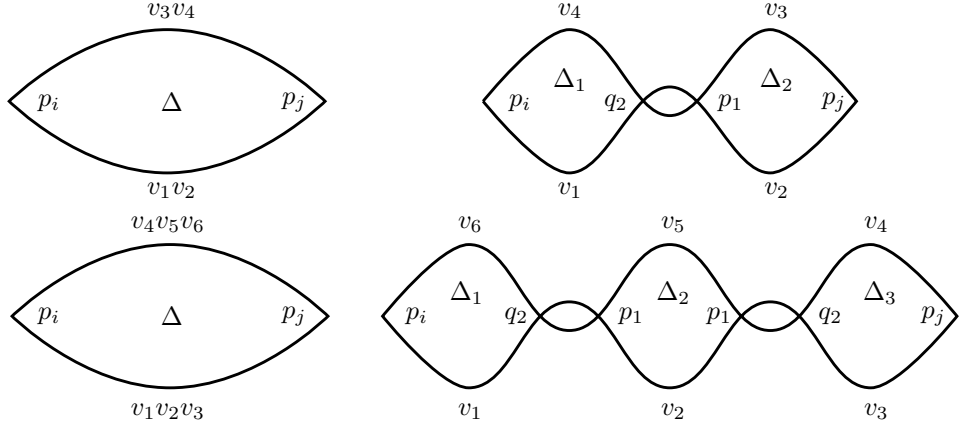
\begin{figure}
\labellist
\small\hair 2pt
\pinlabel $p_i$ at 11 86
\pinlabel $p_j$ at 75 86
\pinlabel $\Delta$ at 43 86
\pinlabel $v_1v_2$ at 43 63
\pinlabel $v_3v_4$ at 43 110
\pinlabel $p_i$ at 134 86
\pinlabel $p_j$ at 216 86
\pinlabel $q_2$ at 159 86
\pinlabel $p_1$ at 189 86
\pinlabel $\Delta_1$ at 147 92
\pinlabel $\Delta_2$ at 201 92
\pinlabel $v_1$ at 147 63
\pinlabel $v_2$ at 201 63
\pinlabel $v_3$ at 201 110
\pinlabel $v_4$ at 147 110
\pinlabel $p_i$ at 11 30
\pinlabel $p_j$ at 75 30
\pinlabel $\Delta$ at 43 30
\pinlabel $v_1v_2v_3$ at 42 5
\pinlabel $v_4v_5v_6$ at 42 54
\pinlabel $p_i$ at 107 30
\pinlabel $p_j$ at 242 30
\pinlabel $q_2$ at 133 30
\pinlabel $p_1$ at 163 30
\pinlabel $p_1$ at 186 30
\pinlabel $q_2$ at 215 30
\pinlabel $\Delta_1$ at 120 36
\pinlabel $\Delta_2$ at 174 36
\pinlabel $\Delta_3$ at 228 36
\pinlabel $v_1$ at 121 5
\pinlabel $v_6$ at 121 54
\pinlabel $v_2$ at 174 5
\pinlabel $v_5$ at 174 54
\pinlabel $v_3$ at 227 5
\pinlabel $v_4$ at 227 54
\endlabellist
\centering
\includegraphics[width=\textwidth]{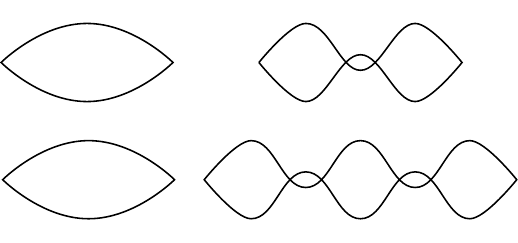}
\caption{
Top row: a disk $\Delta$ for $\Lambda^-$ with corresponding type II configuration $C$ for $\Lambda^+$; bottom row: a disk $\Delta$ for $\Lambda^-$ with corresponding type IV configuration $C$ for $\Lambda^-$. For the diagrams in $\Lambda^+$, the $v_i$'s denote the product of the disk corners along the indicated arcs; for the corresponding diagrams in $\Lambda^-$, the products of negative corners are products of $v_i$'s as shown.
}
\label{fig:RIIsigns}
\end{figure}

For type II, we label the negative corners of $\Delta$ and $C$ as shown in Figure~\ref{fig:RIIsigns}, where for simplicity we remove all branches from $C$ and correspondingly ignore the action of the $f_1$ map on the right hand side of \eqref{eq:qiqj}. (As shown earlier, adding the branches is precisely the same as applying $f_1$.) If we write $\sgn(\Delta)$ for the product of the SFT orientation signs of the corners of $\Delta$ according to Section~\ref{ssec:h} (cf.\ Figure~\ref{fig:signs}), then the disk $\Delta$ contributes $\epsilon_1(\sgn\Delta)v_1v_2p_jv_3v_4p_i$ to $h_2^-$ and thus
\[
\epsilon_1(\sgn\Delta)(-1)^{|q_i|}\{\{v_1v_2p_jv_3v_4p_i,q_i\},q_j\} = \epsilon_1(\sgn\Delta)(-1)^{|p_j|(|v_3|+|v_4|)+|q_i|}v_1v_2v_3v_4
\]
to $\ell_2^-(q_i,q_j)$. Similarly, the disk $\Delta_2$ contributes $\epsilon_2(\sgn\Delta_2)(-1)^{|p_j||v_3|+|q_1|}v_2v_3$ to $\ell_2^+(q_1,q_j)$ and the disk $\Delta_1$ contributes $\epsilon_1(\sgn\Delta_1)v_1q_2v_4$ to $\ell_1^+(q_i)$. From the definition of $f_2$ and the Leibniz rule, it follows that the configuration $C$ contributes
\[
-\epsilon_1\epsilon\epsilon_2(\sgn\Delta_1)(\sgn\Delta_2)(-1)^{|p_j||v_3|+|q_1|+|v_4||q_j|+|p_1|}v_1v_2v_3v_4
\]
to $-f_2(\ell_1^+(q_i),q_j)$. But we see from Figure~\ref{fig:RII-ori} that $\sgn\Delta = -\epsilon\epsilon_2(\sgn\Delta_1)(\sgn\Delta_2)$, and thus the above expression is equal to
$\epsilon_1(\sgn\Delta)(-1)^{|p_j||v_3|+|q_1|+|v_4||q_j|+|p_1|}v_1v_2v_3v_4$. Since $|p_k|+|q_k|=-1$ for all $k$ and $|p_i|+|v_1|+|q_2|+|v_4|=-2$ (since $|h_2^+|=-2$ and $\Delta_1$ gives a term in $h_2^+$), it follows readily that this expression is equal to the contribution of $\Delta$ to $\ell_2^-(q_i,q_j)$.

For type IV, label as shown in Figure~\ref{fig:RIIsigns}. It is readily checked that $\Delta_1$ contributes $\epsilon_1(\sgn\Delta_1)v_1q_2v_6$ to $\ell_1^+(q_i)$; $\Delta_2$ contributes $\epsilon_2(\sgn\Delta_2)(-1)^{|q_1|+|v_5||p_1|}v_2v_5$ to $\ell_2^+(q_1,q_1)$; and $\Delta_3$ contributes $\epsilon_3(\sgn\Delta_3)(-1)^{|q_2|+|v_4||p_j|}q_2v_3v_4$ to $\ell_1^+(q_j)$. From this it is straightforward to calculate that the chain of disks contributes
\[
-\frac{1}{2}\epsilon_2\epsilon_3(\sgn\Delta_2)(\sgn\Delta_3)(-1)^{|q_1|+|v_5||p_j|+|v_4||p_j|+|q_2|}v_2v_3v_4v_5
\]
to the $-\frac{1}{2}\ell_2^+(q_1,q_1)\{p_2,\ell_1^+(q_j)\}$ term in $f_2(q_2,q_j)$, and consequently it contributes
\begin{align*}
\frac{1}{2}&\epsilon_1\epsilon_2\epsilon_3(\sgn\Delta_1)(\sgn\Delta_2)(\sgn\Delta_3)(-1)^{|q_1|+(|v_4|+|v_5|+|v_6|)|p_j|+|q_2|+|v_1|}v_1v_2v_3v_4v_5v_6 \\
&=
\frac{1}{2}\epsilon_1(\sgn\Delta)(-1)^{|q_i|+(|v_4|+|v_5|+|v_6|)|p_j|}v_1v_2v_3v_4v_5v_6
\end{align*}
to $-f_2(\ell_1^+(q_i),q_j)$, where the equality follows from the fact that $(-1)^{|q_1|+1}(\sgn\Delta) = \epsilon_2\epsilon_3(\sgn\Delta_1)(\sgn\Delta_2)(\sgn\Delta_3)$, as can be deduced from Figure~\ref{fig:RII-ori}.
Similarly it can be computed that the chain of disks also contributes the same quantity, with the same sign, to 
$-(-1)^{|q_i|}f_2(q_i,\ell_1^+(q_j))$. These sum together to give 
\[
\epsilon_1(\sgn\Delta)(-1)^{|q_i|+(|v_4|+|v_5|+|v_6|)|p_j|}v_1v_2v_3v_4v_5v_6,
\]
which is exactly the contribution of the original disk $\Delta$ to $(\ell_2^-)\tosft(q_i,q_j)$. This completes the verification of \eqref{eq:qiqj} and finishes Case 1.

The remaining cases are easier to handle than Case 1. We first note the following result, which is readily checked from the definition of $F$, the $L_\infty$ relations for $\{\ell_k^\pm\}$, and the chain-map property of $f_1$ from Proposition~\ref{prop:RIIchain}:

\begin{lemma}
For any $x_1,x_2$, we have \label{lem:F}
$F(x_2,x_1) = (-1)^{|x_1||x_2|+1}F(x_1,x_2)$ and
\[
\ell_1^-(F(x_1,x_2)) = F(\ell_1^+(x_1),x_2) + (-1)^{|x_1|} F(x_1,\ell_1^+(x_2)).
\]
\end{lemma}

\noindent In particular, by the symmetry of $F$, it suffices to check that $F(x_1,x_2)=0$ in the following cases: $x_1=t_i$ and $x_2$ is arbitrary; $x_1=q_1$ and $x_2$ is some $q_j$; and $x_1=q_2$ and $x_2$ is some $q_j$.

\vspace{11pt}

\noindent \textsc{Case 2: $x_1=t_i$ for some $i$.}
By symmetry, we can assume $x_1=t_i$, in which case
\[
F(t_i,x_2) = f_1(\ell_2^+(t_i,x_2))-\ell_2^-(t_i,f_1(x_2)).
\]
We first consider the case where $x_2\in\{q_3,\ldots,q_n,t_1,\ldots,t_s\}$, i.e., $x_2\neq q_1,q_2$. By definition, when $t_i$ is one of the inputs to $\ell_2^\pm$, the result is precisely the string portion of $\ell_2^\pm$: $\ell_2^+(t_i,x_2) = (\ell_2^+)^\str(t_i,x_2) = \alpha^+ t_ix_2$ and $\ell_2^-(t_i,f_1(x_2)) = \ell_2^-(t_i,x_2) =(\ell_2^-)^\str(t_i,x_2) = \alpha^- t_ix_2$ for some $\alpha^\pm\in\frac{1}{2}\Z$. Furthermore, since the placement of base points relative to $x_2$ is the same in $\Lambda^+$ and $\Lambda^-$, we see from Proposition~\ref{prop:l2str} that $\alpha^+ = \alpha^-$. We conclude that $F(t_i,x_2) = f_1(\alpha^+t_ix_2)-\alpha^-t_ix_2 = 0$ in this case.

If $x_2=q_1$, then since $\ell_2^+(t_i,q_1)$ is a multiple of $t_iq_1$ and $f_1(q_1)=0$, we have $F(t_i,q_1)=0$. Finally, if $x_2=q_2$, it follows from Lemma~\ref{lem:F} that $\ell_1^-(F(t_i,q_1)) = F(t_i,\ell_1^+(q_1)) = \epsilon (F(t_i,q_2)-F(t_i,v))$. Since $v$ does not involve $q_2$, we now know that $F(t_i,v)=0$, whence $F(t_i,q_2)=0$.

\vspace{11pt}

\noindent \textsc{Case 3: $x_1=q_1$, $x_2 \in \{q_1,\ldots,q_n\}$.}
We further break this down into three subcases.

If $x_2=q_j$ for $j>2$, then we have
\[
F(q_1,q_j) = f_1(\ell_2^+(q_1,q_j))-f_2(\ell_1^+(q_1),q_j)+(-1)^{|q_1||q_j|} f_2(\ell_1^+(q_j),q_1).
\]
The second term on the right hand side is $-\epsilon f_2(q_2-v,q_j) = -\epsilon f_2(q_2,q_j)$ while the third term is 
$(-1)^{|q_1|+1}f_2(q_1,\ell_1^+(q_j)) = -\frac{\epsilon}{2} f_1(\ell_2^+(q_1,q_1)\{p_2,\ell_1^+(q_j)\})$ by Lemma~\ref{lem:f2}. It follows from the definition of $f_2(q_2,q_j)$ that $F(q_1,q_j) = 0$.

If $x_2=q_1$, note that $F(q_1,q_1)$ is trivially $0$ by symmetry if $|q_1|$ is even. If $|q_1|$ is odd, then 
\begin{align*}
F(q_1,q_1) &= f_1(\ell_2^+(q_1,q_1)) +2f_2(q_1,\ell_1^+(q_1)) \\
&= f_1(\ell_2^+(q_1,q_1)) - \epsilon f_1(\ell_2^+(q_1,q_1)\{p_2,\ell_1^+(q_1)\}) \\
& =0
\end{align*}
where the second equality comes from Lemma~\ref{lem:f2} and the third equality uses $\ell_1^+(q_1) = \epsilon(q_2-v)$.

Finally, if $x_2=q_2$, we want to show that $F(q_1,q_2)=0$. If $|q_1|$ is odd, then from Lemma~\ref{lem:F} and the fact that $F(q_1,q_1)=0$, we have $0 = 2F(q_1,\ell_1^+(q_1)) = 2\epsilon(F(q_1,q_2)-F(q_1,v))$. Since $v$ does not involve $q_2$, we know from the above cases that $F(q_1,v)=0$, whence $F(q_1,q_2)=0$. If $|q_1|$ is even, then 
$\ell_2^+(q_1,q_1)=0$ and thus $f_2(q_2,q_1) = 0$, and
we have
\begin{align*}
F(q_1,q_2) &= f_1(\ell_2^+(q_1,q_2))-f_2(\ell_1^+(q_1),q_2)-\ell_1^-(f_2(q_1,q_2)) \\
&= f_1(\ell_2^+(q_1,q_2)) +\epsilon f_2(q_2,v)- \epsilon f_2(q_2,q_2) \\
&=  f_1(\ell_2^+(q_1,q_2))+f_1(\ell_2^+(q_1,v))-\epsilon f_2(q_2,q_2) \\
&=0,
\end{align*}
where the third equality comes from Lemma~\ref{lem:f2} and the final equality is the definition of $f_2(q_2,q_2)$.

\noindent \textsc{Case 4: $x_1=q_2$, $x_2\in\{q_1,\ldots,q_n\}$.}
As in Case 2, we note from Lemma~\ref{lem:F} that
\[
\ell_1^-(F(q_1,x_2)) = F(\ell_1^+(q_1),x_2)+(-1)^{|q_1|}F(q_1,\ell_1^+(x_2)) = \epsilon F(q_2,x_2)-\epsilon F(v,x_2)+(-1)^{|q_1|}F(q_1,\ell_1^+(x_2)).
\]
Since $v$ does not involve $q_2$, if $x_2\neq q_2$ then all of $F(q_1,x_2)$, $F(v,x_2)$, and $F(q_1,\ell_1^+(x_2))$ are $0$ by previous cases; thus $F(q_2,x_2)=0$ if $x_2\neq q_2$. Finally, if $x_2=q_2$, then $F(v,q_2)=0$ since we have established that $F(x_2,q_2)=0$ if $x_2\neq q_2$. Thus again all of $F(q_1,q_2)$, $F(v,q_2)$, and $F(q_1,\ell_1^+(x_2))$ are $0$, and it follows that $F(q_2,q_2)=0$. 

This completes the proof of Proposition~\ref{prop:RII} and thus the proof of invariance.

\bibliographystyle{alpha}
\bibliography{linfinity-v2}

\end{document}